\documentclass[11pt, reqno]{amsart}
\usepackage{amsfonts,latexsym,enumerate}
\usepackage{amsmath}
\usepackage{amscd} 
\usepackage{float,amsmath,amssymb,mathrsfs,bm,multirow,graphics}
\usepackage[dvips]{graphicx}
\usepackage[percent]{overpic}
\usepackage{amsaddr}
\usepackage[numbers,sort&compress]{natbib}
\usepackage{xcolor}
\usepackage{todonotes}

\newcommand{\R}{{\Bbb R}}

\newcommand{\C}{{\Bbb C}}
\newcommand{\D}{{\Bbb D}}

\newcommand{\diag}{\text{\upshape diag\,}}
\newcommand{\re}{\text{\upshape Re\,}}
\newcommand{\im}{\text{\upshape Im\,}}
\newcommand{\ntlim}{\lim^\angle}

\allowdisplaybreaks

\newtheorem{theorem}{Theorem}[section]

\newtheorem{lemma}[theorem]{Lemma}

\newtheorem{remark}[theorem]{Remark}

\newtheorem{RHproblem}[theorem]{RH problem}
\newtheorem{figuretext}{Figure}

\numberwithin{equation}{section}

\usepackage[colorlinks=true]{hyperref}
\hypersetup{urlcolor=blue, citecolor=red, linkcolor=blue}

\usepackage{tikz}
\usepackage{pict2e}
\usetikzlibrary{arrows,calc,chains, positioning, shapes.geometric,shapes.symbols,decorations.markings,arrows.meta}
\usetikzlibrary{patterns,angles,quotes}
\usetikzlibrary{decorations.markings}
\tikzset{middlearrow/.style={
			decoration={markings,
				mark= at position 0.6 with {\arrow{#1}} ,
			},
			postaction={decorate}
		}
	}
\tikzset{->-/.style={decoration={
				markings,
				mark=at position #1 with {\arrow{latex}}},postaction={decorate}}}
	
\tikzset{-<-/.style={decoration={
				markings,
				mark=at position #1 with {\arrowreversed{latex}}},postaction={decorate}}}
				
				\tikzset{
	master/.style={
		execute at end picture={
			\coordinate (lower right) at (current bounding box.south east);
			\coordinate (upper left) at (current bounding box.north west);
		}
	},
	slave/.style={
		execute at end picture={
			\pgfresetboundingbox
			\path (upper left) rectangle (lower right);
		}
	}
}
\tikzset{cross/.style={cross out, draw, 
         minimum size=2*(#1-\pgflinewidth), 
         inner sep=0pt, outer sep=0pt}}

\def\XXint#1#2#3{{\setbox0=\hbox{$#1{#2#3}{\int}$ }
\vcenter{\hbox{$#2#3$ }}\kern-.59\wd0}}

\title[Boussinesq's equation: asymptotics in Sector I]
{Boussinesq's equation for water waves: \\ asymptotics in Sector I}

\author{C. Charlier$^{1}$ and J. Lenells$^{2}$}

\address{$^{1}$Centre for Mathematical Sciences, Lund University,
22100 Lund, Sweden. \\
$^{2}$Department of Mathematics, KTH Royal Institute of Technology, \\
10044 Stockholm, Sweden.}
\email{christophe.charlier@math.lu.se}
\email{jlenells@kth.se}

\begin{document}

\begin{abstract}
\noindent
In a recent paper, we showed that the large $(x,t)$ behavior of a class of physically relevant solutions of Boussinesq's equation for water waves is described by ten main asymptotic sectors. In the sector adjacent to the positive $x$-axis, referred to as Sector I, we stated without proof an exact expression for the leading asymptotic term together with an error estimate. Here we provide a proof of this asymptotic formula. 
\end{abstract}

\maketitle



%\setcounter{tocdepth}{1}
%\tableofcontents

\section{Introduction}
In a recent paper \cite{CLmain}, we developed an inverse scattering approach to the Boussinesq \cite{B1872} equation
\begin{align}\label{boussinesq}
u_{tt} = u_{xx} + (u^2)_{xx} + u_{xxxx}.
\end{align}
%where $u(x,t)$ is a real-valued function of a space-variable $x$ and a time-variable $t$. 
This approach yields an expression for the solution $u(x,t)$ in terms of the solution $n(x,t,k)$ of a $1\times 3$-row-vector Riemann--Hilbert (RH) problem, and by performing a Deift–Zhou steepest descent analysis of this RH problem, it is possible to derive asymptotic formulas for the large $(x,t)$ behavior of $u(x,t)$. In \cite{CLmain}, we identified in this way ten main sectors in the half-plane $t \geq 0$ describing the asymptotics of $u(x,t)$ for a class of physically relevant initial data (see Figure \ref{sectors.pdf}), and in each sector we provided an asymptotic formula for the solution. For conciseness, the formula in the sector adjacent to the positive $x$-axis, referred to as Sector I in \cite{CLmain}, was stated without proof. Here we provide a proof of this formula.

Equation (\ref{boussinesq}) first appeared in a paper by Boussinesq from 1872 \cite{B1872}, where it was derived as a model for long small-amplitude waves in a channel. About a century later, Hirota discovered that (\ref{boussinesq}) supports soliton solutions \cite{H1973}, and soon afterwards Zakharov \cite{Z1974} constructed a Lax pair for (\ref{boussinesq}), thus demonstrating its formal integrability. In the early 1980's, Deift, Tomei, and Trubowitz \cite{DTT1982} developed an inverse scattering formalism for an equation closely related to the Boussinesq equation (more precisely, for the equation obtained from (\ref{boussinesq}) by deleting the $u_{xx}$-term). However, the problem of developing an inverse scattering transform formalism for (\ref{boussinesq}) is substantially more complicated, because (\ref{boussinesq}) involves a spectral problem with coefficients that do not decay to zero at infinity. This introduces new jumps on the unit circle in the RH problem, and it turns out that all the main asymptotic contributions to $u(x,t)$ as $(x,t) \to \infty$ originate from these new jumps. In particular, as shown below, the main contributions to the asymptotic behavior of $u(x,t)$ in Sector I stem from a global parametrix which is used to cancel jumps on certain parts of the unit circle, as well as from six saddle points on the unit circle.

\begin{figure}
\bigskip\begin{center}
\begin{overpic}[width=.6\textwidth]{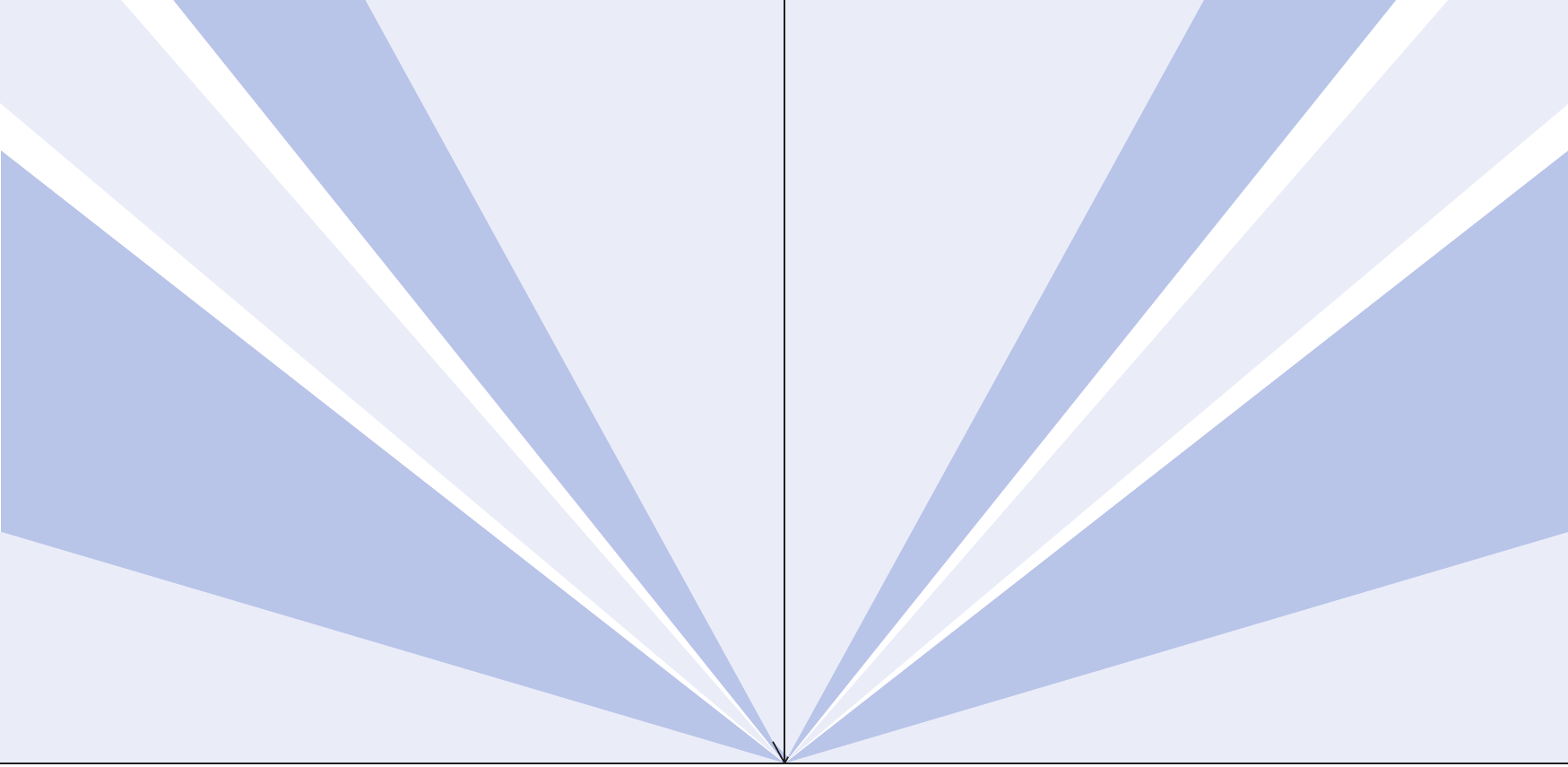}
      \put(102,-.3){\footnotesize $x$}
      \put(49.7,50.8){\footnotesize $t$}
      \put(90,4){\footnotesize I}
      \put(90,21){\footnotesize II}
      \put(86,36){\footnotesize III}
      \put(76,41){\footnotesize IV}
      \put(60.5,41){\footnotesize V}
      \put(37,41){\footnotesize VI}
      \put(19,41){\footnotesize VII}
      \put(9,36){\footnotesize VIII}
      \put(8,21){\footnotesize IX}
      \put(8,4){\footnotesize X}
    \end{overpic}
     \begin{figuretext}\label{sectors.pdf}
       The ten asymptotic sectors for (\ref{boussinesq}) in the $xt$-plane.
     \end{figuretext}
     \end{center}
\end{figure}

To describe the large $(x,t)$ asymptotics of $u(x,t)$, it is sufficient to consider Sectors I--V because Sectors IV--X are related to Sectors I--V by symmetry.
The derivation of the asymptotic formula in Sector I presented here features a number of differences compared to the derivations in Sectors II--V. For example, in Sector I, $x$ plays the role of the large parameter instead of $t$, because $t$ is not uniformly large as $(x,t)\to \infty$ in Sector I. 
The analysis needed for Sector I is also complicated by the fact that as $(x,t)$ approaches the $x$-axis, the six saddle points relevant for the analysis merge with the six points $\{ie^{\pi ij /3}\}_{j=0}^5$, which are all intersection points for the jump contour of the RH problem. This makes the local analysis near the saddle points more involved. It also forces us to perform an additional transformation of the RH problem (see Section \ref{nton1sec}) and for this purpose additional analytic approximations of the reflection coefficients $r_1$ and $r_2$ are required on the six rays that start at the origin and pass through $\{ie^{\pi ij /3}\}_{j=0}^5$. Since another set of analytic approximations of $r_1$ and $r_2$ is needed on the unit circle, this leads to the curious situation that we have two different approximations of the same functions on overlapping domains. Nevertheless, thanks to certain nonlinear relations obeyed by $r_1$ and $r_2$ (see (\ref{r1r2 relation on the unit circle}) and (\ref{r1r2 relation with kbar symmetry})), the relevant new jumps can still be estimated (see Lemma \ref{v1av4asmalllemma}).

%It turns out that the $k_1, k_2, \omega k_1, \omega k_2, \omega^2 k_1, \omega^2 k_2$, where $\omega := e^{\frac{2\pi i}{3}}$.

\section{Main result}\label{section:main results}
Let $u_0(x)$ and $u_1(x)$ be real-valued functions in the Schwartz class $\mathcal{S}(\R)$ and consider the initial value problem for (\ref{boussinesq}) on the line with initial conditions
\begin{align}\label{initialdata}
u(x,0) = u_{0}(x), \qquad u_{t}(x,0) = u_{1}(x) \qquad \text{for $x \in \R$}.
\end{align}

\subsection{Assumptions}
Our results will be valid under the same asumptions on $u_0$ and $u_1$ as in \cite[Theorem 2.14]{CLmain}. To formulate these assumptions, we define two scattering matrices $s(k)$ and $s^A(k)$ by
\begin{align*}
& s(k) = I - \int_\R e^{-x \mathcal{L}(k)}(\mathsf{U}X)(x,k)e^{x \mathcal{L}(k)}dx, & & s^A(k) = I + \int_\R e^{x \mathcal{L}(k)}(\mathsf{U}^T X^A)(x,k)e^{-x \mathcal{L}(k)}dx,
\end{align*}
where the following notation is used:
\begin{enumerate}[$\bullet$]
\item $\mathcal{L} = \diag(l_1 , l_2 , l_3)$ where the functions $\{l_j(k)\}_{j=1}^3$ are defined by
\begin{align}\label{ljdef}
& l_{j}(k) = i \frac{\omega^{j}k + (\omega^{j}k)^{-1}}{2\sqrt{3}}, \qquad k \in \C\setminus \{0\},
\end{align}
with $\omega := e^{\frac{2\pi i}{3}}$. 

\item $\mathsf{U}(x,k)$ is given by
\begin{align*}
\mathsf{U}(x,k) = P(k)^{-1} \begin{pmatrix}
0 & 0 & 0 \\
0 & 0 & 0 \\
-\frac{u_{0x}}{4}-\frac{iv_{0}}{4\sqrt{3}} & -\frac{u_{0}}{2} & 0
\end{pmatrix} P(k),
\end{align*} 
where $v_{0}(x) := \int_{-\infty}^{x}u_{1}(x')dx'$ and
$$P(k) := \begin{pmatrix}
1 & 1 & 1  \\
l_{1}(k) & l_{2}(k) & l_{3}(k) \\
l_{1}(k)^{2} & l_{2}(k)^{2} & l_{3}(k)^{2}
\end{pmatrix}.$$

\item $X(x,k)$ and $X^A(x,k)$ are the unique solutions of the Volterra equations
\begin{align*}  
 & X(x,k) = I - \int_x^{\infty} e^{(x-x')\mathcal{L}(k)} (\mathsf{U}X)(x',k) e^{-(x-x')\mathcal{L}(k)} dx',
	\\
 & X^A(x,k) = I + \int_x^{\infty} e^{-(x-x')\mathcal{L}(k)} (\mathsf{U}^T X^A)(x',k)e^{(x-x')\mathcal{L}(k)} dx'.
\end{align*}

\end{enumerate}

Our assumptions on $\{u_0, u_1\}$ can now be formulated as follows:
\begin{enumerate}[$(i)$]
\item Mass conservation: we suppose that $\int_{\mathbb{R}}u_{1}(x)dx = 0$.

\item Absence of solitons: we assume that $s(k)_{11}\neq 0$ for $k \in (\bar{D}_2 \cup \partial \D) \setminus \hat{\mathcal{Q}}$, where $\{D_n\}_1^6$ are the open subsets of the complex plane shown in Figure \ref{fig: Dn}, $\partial \D$ is the unit circle, and $\hat{\mathcal{Q}} = \{\kappa_{j}\}_{j=1}^{6} \cup \{0\}$ with $\kappa_{j} = e^{\frac{\pi i(j-1)}{3}}$, $j=1,\ldots,6$, the sixth roots of unity.

\item Generic behavior of $s$ and $s^{A}$  near $k=1$ and $k=-1$: for $k_{\star} =\pm 1$, we assume that
\begin{align*}
& \lim_{k \to k_{\star}} (k-k_{\star}) s(k)_{11} \neq 0, & & \hspace{-0.1cm} \lim_{k \to k_{\star}} (k-k_{\star}) s(k)_{13} \neq 0, & & \hspace{-0.1cm} \lim_{k \to k_{\star}} s(k)_{31} \neq 0, & & \hspace{-0.1cm} \lim_{k \to k_{\star}} s(k)_{33} \neq 0, \\
& \lim_{k \to k_{\star}} (k-k_{\star}) s^{A}(k)_{11} \neq 0, & & \hspace{-0.1cm} \lim_{k \to k_{\star}} (k-k_{\star}) s^{A}(k)_{31} \neq 0, & & \hspace{-0.1cm} \lim_{k \to k_{\star}} s^{A}(k)_{13} \neq 0, & & \hspace{-0.1cm} \lim_{k \to k_{\star}} s^{A}(k)_{33} \neq 0.
\end{align*}

\item Existence of a global solution of the initial value problem: we suppose that $r_{1}(k)=0$ for all $k \in [0,i]$, where $[0,i]$ is the vertical segment from $0$ to $i$.

\end{enumerate}

\begin{figure}
\begin{center}
\begin{tikzpicture}[scale=0.7]
\node at (0,0) {};
\draw[black,line width=0.45 mm,->-=0.4,->-=0.85] (0,0)--(30:4);
\draw[black,line width=0.45 mm,->-=0.4,->-=0.85] (0,0)--(90:4);
\draw[black,line width=0.45 mm,->-=0.4,->-=0.85] (0,0)--(150:4);
\draw[black,line width=0.45 mm,->-=0.4,->-=0.85] (0,0)--(-30:4);
\draw[black,line width=0.45 mm,->-=0.4,->-=0.85] (0,0)--(-90:4);
\draw[black,line width=0.45 mm,->-=0.4,->-=0.85] (0,0)--(-150:4);

\draw[black,line width=0.45 mm] ([shift=(-180:2.5cm)]0,0) arc (-180:180:2.5cm);
\draw[black,arrows={-Triangle[length=0.2cm,width=0.18cm]}]
($(3:2.5)$) --  ++(90:0.001);
\draw[black,arrows={-Triangle[length=0.2cm,width=0.18cm]}]
($(57:2.5)$) --  ++(-30:0.001);
\draw[black,arrows={-Triangle[length=0.2cm,width=0.18cm]}]
($(123:2.5)$) --  ++(210:0.001);
\draw[black,arrows={-Triangle[length=0.2cm,width=0.18cm]}]
($(177:2.5)$) --  ++(90:0.001);
\draw[black,arrows={-Triangle[length=0.2cm,width=0.18cm]}]
($(243:2.5)$) --  ++(330:0.001);
\draw[black,arrows={-Triangle[length=0.2cm,width=0.18cm]}]
($(297:2.5)$) --  ++(210:0.001);

\draw[black,line width=0.15 mm] ([shift=(-30:0.55cm)]0,0) arc (-30:30:0.55cm);

\node at (0.8,0) {$\tiny \frac{\pi}{3}$};

\node at (-1:2.9) {\footnotesize $\Gamma_8$};
\node at (60:2.9) {\footnotesize $\Gamma_9$};
\node at (120:2.9) {\footnotesize $\Gamma_7$};
\node at (181:2.9) {\footnotesize $\Gamma_8$};
\node at (240:2.83) {\footnotesize $\Gamma_9$};
\node at (300:2.83) {\footnotesize $\Gamma_7$};

\node at (105:1.45) {\footnotesize $\Gamma_1$};
\node at (138:1.45) {\footnotesize $\Gamma_2$};
\node at (223:1.45) {\footnotesize $\Gamma_3$};
\node at (-104:1.45) {\footnotesize $\Gamma_4$};
\node at (-42:1.45) {\footnotesize $\Gamma_5$};
\node at (43:1.45) {\footnotesize $\Gamma_6$};

\node at (97:3.3) {\footnotesize $\Gamma_4$};
\node at (144:3.3) {\footnotesize $\Gamma_5$};
\node at (217:3.3) {\footnotesize $\Gamma_6$};
\node at (-96:3.3) {\footnotesize $\Gamma_1$};
\node at (-35:3.3) {\footnotesize $\Gamma_2$};
\node at (36:3.3) {\footnotesize $\Gamma_3$};
\end{tikzpicture}
\hspace{1.7cm}
\begin{tikzpicture}[scale=0.7]
\node at (0,0) {};
\draw[black,line width=0.45 mm] (0,0)--(30:4);
\draw[black,line width=0.45 mm] (0,0)--(90:4);
\draw[black,line width=0.45 mm] (0,0)--(150:4);
\draw[black,line width=0.45 mm] (0,0)--(-30:4);
\draw[black,line width=0.45 mm] (0,0)--(-90:4);
\draw[black,line width=0.45 mm] (0,0)--(-150:4);

\draw[black,line width=0.45 mm] ([shift=(-180:2.5cm)]0,0) arc (-180:180:2.5cm);
\draw[black,line width=0.15 mm] ([shift=(-30:0.55cm)]0,0) arc (-30:30:0.55cm);

\node at (120:1.6) {\footnotesize{$D_{1}$}};
\node at (-60:3.7) {\footnotesize{$D_{1}$}};

\node at (180:1.6) {\footnotesize{$D_{2}$}};
\node at (0:3.7) {\footnotesize{$D_{2}$}};

\node at (240:1.6) {\footnotesize{$D_{3}$}};
\node at (60:3.7) {\footnotesize{$D_{3}$}};

\node at (-60:1.6) {\footnotesize{$D_{4}$}};
\node at (120:3.7) {\footnotesize{$D_{4}$}};

\node at (0:1.6) {\footnotesize{$D_{5}$}};
\node at (180:3.7) {\footnotesize{$D_{5}$}};

\node at (60:1.6) {\footnotesize{$D_{6}$}};
\node at (-120:3.7) {\footnotesize{$D_{6}$}};

\node at (0.8,0) {$\tiny \frac{\pi}{3}$};

\draw[fill] (0:2.5) circle (0.1);
\draw[fill] (60:2.5) circle (0.1);
\draw[fill] (120:2.5) circle (0.1);
\draw[fill] (180:2.5) circle (0.1);
\draw[fill] (240:2.5) circle (0.1);
\draw[fill] (300:2.5) circle (0.1);

\node at (0:2.9) {\footnotesize{$\kappa_1$}};
\node at (60:2.85) {\footnotesize{$\kappa_2$}};
\node at (120:2.85) {\footnotesize{$\kappa_3$}};
\node at (180:2.9) {\footnotesize{$\kappa_4$}};
\node at (240:2.85) {\footnotesize{$\kappa_5$}};
\node at (300:2.85) {\footnotesize{$\kappa_6$}};

\draw[dashed] (-6.3,-3.8)--(-6.3,3.8);

\end{tikzpicture}
\end{center}
\begin{figuretext}\label{fig: Dn}
The contour $\Gamma = \cup_{j=1}^9 \Gamma_j$ in the complex $k$-plane (left) and the open sets $D_{n}$, $n=1,\ldots,6$, together with the sixth roots of unity $\kappa_j$, $j = 1, \dots, 6$ (right).
\end{figuretext}
\end{figure}

\subsection{Statement of the main result}
It is shown in \cite{CLmain} that the solution $u(x,t)$ to the initial value problem for \eqref{boussinesq} with initial data $u_0, u_1$ satisfying the assumptions $(i)$--$(iv)$ can be expressed in terms of the solution $n$ of a row-vector RH problem with jump contour $\Gamma = \cup_{j=1}^9 \Gamma_j$ shown and oriented as in Figure \ref{fig: Dn}. In particular, the assumptions $(i)$--$(iv)$ imply that the solution $u(x,t)$ exists globally. Our main theorem provides the asymptotics of $u(x,t)$ in Sector I given by $\tau := t/x \in [0,\tau_{\max}]$, where $\tau_{\max} \in (0,1)$ is a constant. 
%The jump matrix for the RH problem for $n$ is given in terms of
The asymptotic formula is given in terms of two spectral functions $r_{1}(k)$ and $r_{2}(k)$ defined by
\begin{align}\label{r1r2def}
\begin{cases}
r_1(k) = \frac{(s(k))_{12}}{(s(k))_{11}}, & k \in \hat{\Gamma}_{1}\setminus \hat{\mathcal{Q}},
	\\ 
r_2(k) = \frac{(s^A(k))_{12}}{(s^A(k))_{11}}, \quad & k \in \hat{\Gamma}_{4}\setminus \hat{\mathcal{Q}},
\end{cases}
\end{align}	
where $\hat{\Gamma}_{j} = \Gamma_{j} \cup \partial \D$ denotes the union of $\Gamma_j$ and the unit circle. We define $\tilde{\Phi}_{ij}(\tau, k)$ for $1 \leq j<i \leq 3$ by
\begin{align}\label{def of Phi ij}
 \tilde{\Phi}_{ij}(\tau,k) = (l_{i}(k)-l_{j}(k)) + (z_{i}(k)-z_{j}(k))\tau, 
\end{align}
where $l_j(k)$ are as in (\ref{ljdef}) and the functions $z_j(k)$, $j = 1,2,3$, are defined similarly:
\begin{align}\label{zjdef}
& z_{j}(k) = i \frac{(\omega^{j}k)^{2} + (\omega^{j}k)^{-2}}{4\sqrt{3}}, \qquad k \in \C\setminus \{0\}.
\end{align}
Finally, we introduce the saddle points $\{k_{j}=k_{j}(\tau)\}_{j=1}^{4}$ of $\tilde{\Phi}_{21}$ which are given by
\begin{subequations}\label{def of kj}
\begin{align}
& k_{1} = \frac{1}{4\tau}\bigg(1 - \sqrt{8\tau^2 + 1} + i \sqrt{2}\sqrt{4\tau^2 - 1 + \sqrt{8\tau^2 + 1}} \bigg), && k_{2} = \bar{k}_1,  \\
& k_{3} = \frac{1}{4\tau}\bigg(1 + \sqrt{8\tau^2 + 1} + \sqrt{2}\sqrt{-4\tau^2 + 1 + \sqrt{8\tau^2 + 1}} \bigg), &&  k_{4}=k_{3}^{-1}, 
\end{align}
\end{subequations} 
and satisfy $|k_{1}|=1$, $k_{3}\in (1,+\infty)$, and $\arg k_{1} \in (\frac{\pi}{2},\frac{2\pi}{3})$. 
The following is our main result. 

\begin{theorem}[Asymptotics in Sector I]\label{asymptoticsth}
Let $u_0,u_1 \in \mathcal{S}(\R)$ be real-valued functions such that Assumptions $(i)$--$(iv)$ are fulfilled. For any integer $N \geq 1$, there exists a constant $\tau_{\max} \in (0,1)$ such that the global solution $u(x,t)$ of the initial value problem for (\ref{boussinesq}) with initial data $u_0, u_1$ satisfies
\begin{align}\label{uasymptotics}
& u(x,t) =
 \frac{A(\tau)}{\sqrt{x}} \cos \alpha(x, \tau) 
+ O\bigg(\frac{1}{x^N} +\frac{C_N(\tau) \ln x}{x}\bigg) \qquad \text{as $x \to +\infty$},
\end{align}
uniformly for $\tau = t/x \in [0, \tau_{\max}]$, where $C_{N}(\tau) \geq 0$ is a smooth function of $\tau$ that vanishes to all orders at $\tau = 0$, and $A(\tau)$ and $\alpha(x, \tau)$ are defined by
\begin{align}
& A(\tau) := 2\sqrt{3} \frac{\sqrt{-\nu}\sqrt{-1-2\cos (2\arg k_{1})}}{-ik_{1}z_{\star}} \im k_{1}, \nonumber \\
& \alpha(x, \tau) := \frac{3\pi}{4} + \arg r_{2}(k_{1}) + \arg \Gamma(i \nu) + \arg d_{0} + x \, \im \tilde{\Phi}_{21}(\tau,k_{1}). \nonumber
\end{align}
Here $\Gamma(k)$ is the Gamma function, $ z_{\star}$ is given by
\begin{align} 
 z_{\star} = z_{\star}(\tau) := \sqrt{2}e^{\frac{\pi i}{4}} \sqrt{\frac{4\tau -3k_{1} - k_{1}^{3}}{4k_{1}^{4}}},
\end{align}
where the branch of the square root is such that $-ik_{1}z_{\star}>0$, and
\begin{align} \nonumber
& \nu = \nu(k_1) := - \frac{1}{2\pi}\ln(1+r_{1}(k_{1})r_{2}(k_{1})) \leq 0, 
	 \\ \nonumber
& \arg d_{0} = \arg d_{0}(x, \tau) :=  \nu  \ln \bigg| \frac{(\frac{1}{\omega^{2}k_{1}}-k_{1})(\frac{1}{\omega k_{1}}-k_{1})}{3(\frac{1}{k_{1}}-k_{1})^{2}z_{\star}^{2}} \bigg| - \nu\ln x
	 \\ \label{def of nu beta}
& \hspace{2cm} + \frac{1}{2\pi} \int_{i}^{k_{1}} \ln \bigg| \frac{(k_{1}-s)^{2}(\frac{1}{\omega^{2}k_{1}}-s)(\frac{1}{\omega k_{1}}-s)}{(\frac{1}{k_{1}}-s)^{2}(\omega k_{1}-s)(\omega^{2} k_{1}-s)} \bigg| d \ln(1+r_{1}(s)r_{2}(s)), 
\end{align}
where the path of the integral $\int_{i}^{k_{1}}$ starts at $i$, follows the unit circle in the counterclockwise direction, and ends at $k_{1}$.
\end{theorem}

\begin{remark}
Since $k_1$ tends to $i$ as $\tau \to 0$ and $r_1(k)$ vanishes to all orders at $k = i$, it follows that $\nu$, and hence also $A(\tau)$, vanishes to all orders as $\tau \to 0$. Consequently, (\ref{uasymptotics}) implies that
\begin{align*}
& u(x,t) = O\bigg(\frac{1}{x^N} +\frac{C_N(\tau)}{\sqrt{x}}\bigg) \qquad \text{as $x \to +\infty$}
\end{align*}
uniformly for $\tau \in [0, \tau_{\max}]$. In particular, for any fixed $t \geq 0$, the formula (\ref{uasymptotics}) reduces to $u(x,t) = O(x^{-N})$ as $x \to +\infty$, in consistency with the fact that $u(x,t)$ is a Schwartz class solution.
\end{remark}

\begin{remark}
By expressing the asymptotic formula in Theorem \ref{asymptoticsth} in terms of $\zeta = x/t$ instead of $\tau = t/x$, we see that it is equivalent to the formula stated in \cite{CLmain} for Sector I (note that $z_{\star}$ here should be identified with $z_{\star}/\sqrt{\zeta}$ in \cite{CLmain}). We will use the variable $\tau$ in this paper, because it allows us to evaluate certain expressions at $t = 0$ and $\tau = 0$ without taking limits. 
\end{remark}

\begin{remark}
By combining the result of Theorem \ref{asymptoticsth} with the asymptotic formula for $u(x,t)$ in Sector II (see \cite{CLmain}), it can be seen that the conclusion of Theorem \ref{asymptoticsth} in fact holds for any choice of $\tau_{\max} \in (0,1)$.
\end{remark}

The proof of Theorem \ref{asymptoticsth} is given in Sections \ref{overviewsec}--\ref{uasymptoticssec}. It is based on a steepest descent analysis of a RH problem derived in \cite{CLmain} for a $1\times 3$-row-vector valued solution $n$, which we next recall.

\section{The RH problem for $n$}\label{RHnsec}
For $j = 1, \dots, 6$, let $\Gamma_{j'} = \Gamma_j \setminus \D$ be the part of $\Gamma_j$ outside the open unit disk $\D = \{k \in \C \, | \, |k| < 1\}$ and let $\Gamma_{j''} := \Gamma_j \setminus \Gamma_{j'}$ be the part inside $\D$, where $\Gamma_j$ is as in Figure \ref{fig: Dn}.
Let $\theta_{ij}(x,t,k) = x \, \tilde{\Phi}_{ij}(\tau,k)$ and denote by $\Gamma_{\star} = \{i\kappa_j\}_{j=1}^6 \cup \{0\}$ the set of intersection points of $\Gamma$.
Define the function $f$ on the unit circle by
\begin{align}\label{def of f}
f(k) := 1+r_{1}(k)r_{2}(k) + r_{1}(\tfrac{1}{\omega^{2}k})r_{2}(\tfrac{1}{\omega^{2}k}), \qquad k \in \partial \D,
\end{align}
and define the jump matrix $v(x,t,k)$ for $k \in \Gamma$ by
\begin{align}
& v_{1'} = \begin{pmatrix}
1 & -r_{1}(k)e^{-\theta_{21}} & 0 \\
0 & 1 & 0 \\
0 & 0 & 1
\end{pmatrix}, \; v_{1''} = \begin{pmatrix}
1 & 0 & 0 \\
r_{1}(\frac{1}{k})e^{\theta_{21}} & 1 & 0 \\
0 & 0 & 1
\end{pmatrix}, \; v_{2'} = \begin{pmatrix}
1 & 0 & 0 \\
0 & 1 & -r_{2}(\frac{1}{\omega k})e^{-\theta_{32}} \\
0 & 0 & 1
\end{pmatrix}, \nonumber \\
& v_{2''} = \begin{pmatrix}
1 & 0 & 0 \\
0 & 1 & 0 \\
0 & r_{2}(\omega k)e^{\theta_{32}} & 1
\end{pmatrix}, \;  v_{3'} = \begin{pmatrix}
1 & 0 & 0 \\
0 & 1 & 0 \\
-r_{1}(\omega^{2}k)e^{\theta_{31}} & 0 & 1
\end{pmatrix}, \; v_{3''} = \begin{pmatrix}
1 & 0 & r_{1}(\frac{1}{\omega^{2}k})e^{-\theta_{31}} \\
0 & 1 & 0 \\
0 & 0 & 1
\end{pmatrix}, \nonumber \\
& v_{4'} = \begin{pmatrix}
1 & -r_{2}(\frac{1}{k})e^{-\theta_{21}} & 0 \\
0 & 1 & 0 \\
0 & 0 & 1
\end{pmatrix}, \; v_{4''} = \begin{pmatrix}
1 & 0 & 0 \\
r_{2}(k)e^{\theta_{21}} & 1 & 0 \\
0 & 0 & 1
\end{pmatrix}, \; v_{5'} = \begin{pmatrix}
1 & 0 & 0 \\
0 & 1 & -r_{1}(\omega k)e^{-\theta_{32}} \\
0 & 0 & 1
\end{pmatrix}, \nonumber \\
& v_{5''} = \begin{pmatrix}
1 & 0 & 0 \\
0 & 1 & 0 \\
0 & r_{1}(\frac{1}{\omega k})e^{\theta_{32}} & 1
\end{pmatrix}, \;  v_{6'} = \begin{pmatrix}
1 & 0 & 0 \\
0 & 1 & 0 \\
-r_{2}(\frac{1}{\omega^{2} k})e^{\theta_{31}} & 0 & 1
\end{pmatrix}, \; v_{6''} = \begin{pmatrix}
1 & 0 & r_{2}(\omega^{2}k)e^{-\theta_{31}} \\
0 & 1 & 0 \\
0 & 0 & 1
\end{pmatrix}, \nonumber \\
& v_{7} = \begin{pmatrix}
1 & -r_{1}(k)e^{-\theta_{21}} & r_{2}(\omega^{2}k)e^{-\theta_{31}} \\
-r_{2}(k)e^{\theta_{21}} & 1+r_{1}(k)r_{2}(k) & \big(r_{2}(\frac{1}{\omega k})-r_{2}(k)r_{2}(\omega^{2}k)\big)e^{-\theta_{32}} \\
r_{1}(\omega^{2}k)e^{\theta_{31}} & \big(r_{1}(\frac{1}{\omega k})-r_{1}(k)r_{1}(\omega^{2}k)\big)e^{\theta_{32}} & f(\omega^{2}k)
\end{pmatrix}, \nonumber \\
& v_{8} = \begin{pmatrix}
f(k) & r_{1}(k)e^{-\theta_{21}} & \big(r_{1}(\frac{1}{\omega^{2} k})-r_{1}(k)r_{1}(\omega k)\big)e^{-\theta_{31}} \\
r_{2}(k)e^{\theta_{21}} & 1 & -r_{1}(\omega k) e^{-\theta_{32}} \\
\big( r_{2}(\frac{1}{\omega^{2}k})-r_{2}(\omega k)r_{2}(k) \big)e^{\theta_{31}} & -r_{2}(\omega k) e^{\theta_{32}} & 1+r_{1}(\omega k)r_{2}(\omega k)
\end{pmatrix}, \nonumber \\
& v_{9} = \begin{pmatrix}
1+r_{1}(\omega^{2}k)r_{2}(\omega^{2}k) & \big( r_{2}(\frac{1}{k})-r_{2}(\omega k)r_{2}(\omega^{2} k) \big)e^{-\theta_{21}} & -r_{2}(\omega^{2}k)e^{-\theta_{31}} \\
\big(r_{1}(\frac{1}{k})-r_{1}(\omega k) r_{1}(\omega^{2} k)\big)e^{\theta_{21}} & f(\omega k) & r_{1}(\omega k)e^{-\theta_{32}} \\
-r_{1}(\omega^{2}k)e^{\theta_{31}} & r_{2}(\omega k) e^{\theta_{32}} & 1
\end{pmatrix}, \label{vdef}
\end{align}
where $v_j, v_{j'}, v_{j''}$ are the restrictions of $v$ to $\Gamma_{j}$, $\Gamma_{j'}$, and $\Gamma_{j''}$, respectively. 
The jump matrix $v$ obeys the symmetries
\begin{align}\label{vsymm}
v(x,t,k) = \mathcal{A} v(x,t,\omega k)\mathcal{A}^{-1}
 = \mathcal{B} v(x,t, k^{-1})^{-1}\mathcal{B}, \qquad k \in \Gamma,
\end{align}
where
\begin{align}\label{def of Acal and Bcal}
\mathcal{A} := \begin{pmatrix}
0 & 0 & 1 \\
1 & 0 & 0 \\
0 & 1 & 0
\end{pmatrix} \qquad \mbox{ and } \qquad \mathcal{B} := \begin{pmatrix}
0 & 1 & 0 \\
1 & 0 & 0 \\
0 & 0 & 1
\end{pmatrix}.
\end{align}
The following RH problem was derived in \cite{CLmain}.

\begin{RHproblem}[RH problem for $n$]\label{RHn}
Find a $1 \times 3$-row-vector valued function $n(x,t,k)$ with the following properties:
\begin{enumerate}[$(a)$]
\item\label{RHnitema} $n(x,t,\cdot) : \C \setminus \Gamma \to \mathbb{C}^{1 \times 3}$ is analytic.

\item\label{RHnitemb} The limits of $n(x,t,k)$ as $k$ approaches $\Gamma \setminus \Gamma_\star$ from the left and right exist, are continuous on $\Gamma \setminus \Gamma_\star$, and are denoted by $n_+$ and $n_-$, respectively. Moreover, 
\begin{align}\label{njump}
  n_+(x,t,k) = n_-(x, t, k) v(x, t, k) \qquad \text{for} \quad k \in \Gamma \setminus \Gamma_\star.
\end{align}

\item\label{RHnitemc} $n(x,t,k) = O(1)$ as $k \to k_{\star} \in \Gamma_\star$.

\item\label{RHnitemd} For $k \in \C \setminus \Gamma$, $n$ obeys the symmetries
\begin{align}\label{nsymm}
n(x,t,k) = n(x,t,\omega k)\mathcal{A}^{-1} = n(x,t,k^{-1}) \mathcal{B}.
\end{align}

\item\label{RHniteme} $n(x,t,k) = (1,1,1) + O(k^{-1})$ as $k \to \infty$.
\end{enumerate}
\end{RHproblem}

\section{Brief overview of the proof}\label{overviewsec}
Since $u_0$ and $u_1$ satisfy Assumptions $(i)$--$(iv)$, \cite[Theorems 2.6 and 2.12]{CLmain} imply that RH problem \ref{RHn} has a unique solution $n(x,t,k)$ for each $(x,t) \in \R \times [0,\infty)$, that
$$n_{3}^{(1)}(x,t) := \lim_{k\to \infty} k (n_{3}(x,t,k) -1)$$ 
is well-defined and smooth for $(x,t) \in \R \times [0,\infty)$, and that 
\begin{align}\label{recoveruvn}
u(x,t) = -i\sqrt{3}\frac{\partial}{\partial x}n_{3}^{(1)}(x,t)
\end{align}
is a real-valued Schwartz class solution of (\ref{boussinesq}) on $\R \times [0,\infty)$ with initial data $u(x,0)=u_{0}(x)$ and $u_{t}(x,0)=u_{1}(x)$.
We can therefore obtain the asymptotics of $u(x,t)$ by analyzing the RH problem \ref{RHn} with the help of Deift--Zhou \cite{DZ1993} steepest descent techniques. In this analysis the saddle points of the three phase functions $\tilde{\Phi}_{21}$, $\tilde{\Phi}_{31}$, and $\tilde{\Phi}_{32}$ play a central role.
The saddle points of $\tilde{\Phi}_{21}$, i.e., the solutions of $\partial_k \tilde{\Phi}_{21}(\tau,k) = 0$, were denoted by $\{k_{j}\}_{j=1}^{4}$ in Section \ref{section:main results}. Since
$\tilde{\Phi}_{31}(\tau,k) = - \tilde{\Phi}_{21}(\tau,\omega^{2}k)$ and $\tilde{\Phi}_{32}(\tau, k) = \tilde{\Phi}_{21}(\tau, \omega k)$,
it follows that $\{\omega k_{j}\}_{j=1}^{4}$ are the saddle points of $\tilde{\Phi}_{31}$ and that $\{\omega^{2} k_{j}\}_{j=1}^{4}$ are the saddle points of $\tilde{\Phi}_{32}$. 
For Sector I only the six saddle points $\{\omega^{j}k_{1},\omega^{j}k_{2}\}_{j=0}^{2}$ play a role in the asymptotic analysis. The signature tables for $\tilde{\Phi}_{21}$, $\tilde{\Phi}_{31}$, and $\tilde{\Phi}_{32}$ are shown in Figure \ref{IIbis fig: Re Phi 21 31 and 32 for zeta=0.7}. 

\begin{figure}[h]
\begin{center}
\begin{tikzpicture}[master]
\node at (0,0) {\includegraphics[width=4.5cm]{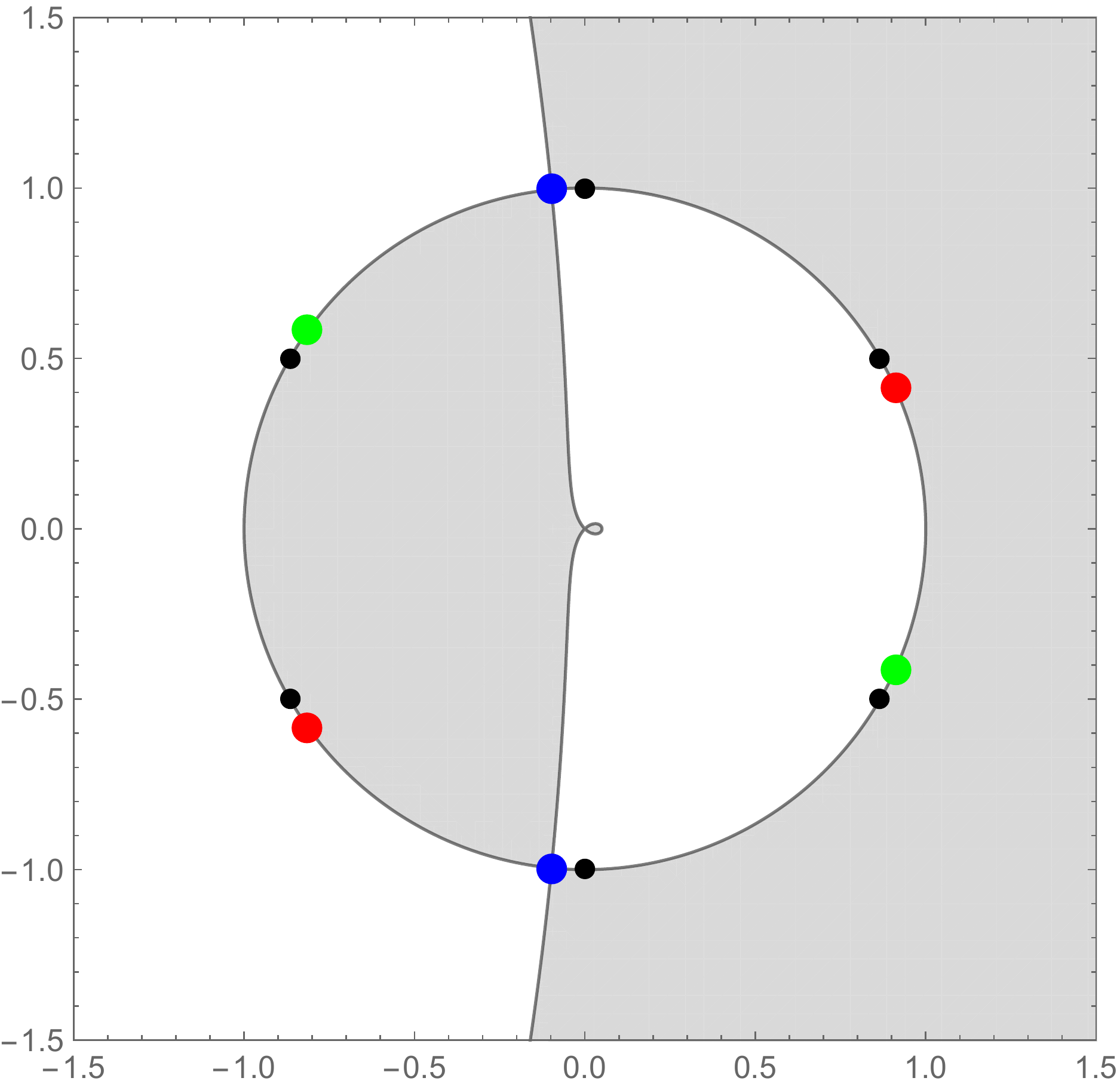}};

\node at (-0.22,1.55) {\tiny $k_1$};
\node at (-0.22,-1.43) {\tiny $k_2$};

\end{tikzpicture} \hspace{0.1cm} 
\begin{tikzpicture}[slave]
\node at (0,0) {\includegraphics[width=4.5cm]{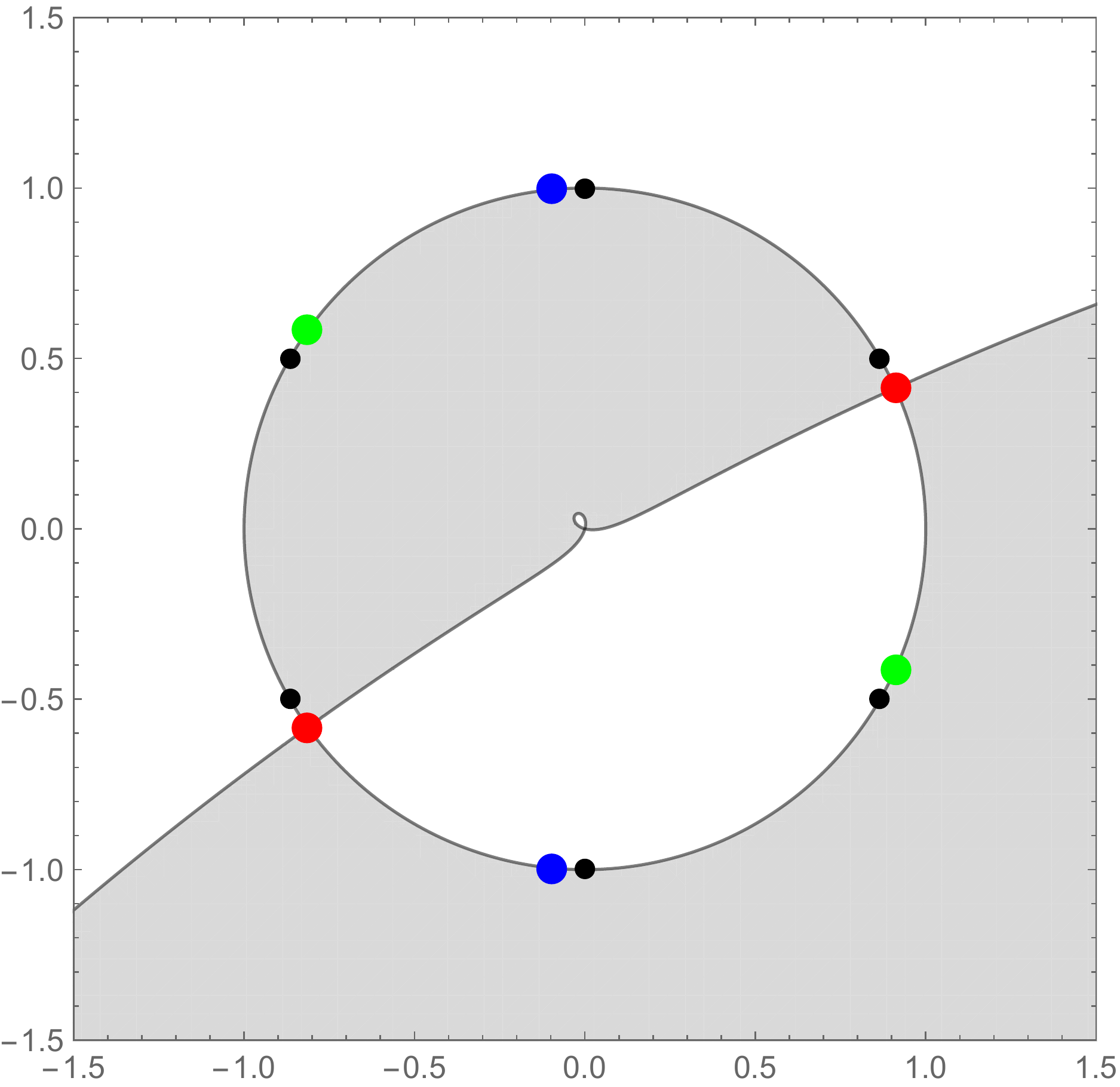}};

\node at (-1.06,-1.01) {\tiny $\omega k_1$};
\node at (1.7,0.55) {\tiny $\omega k_2$};

\end{tikzpicture} \hspace{0.1cm} 
\begin{tikzpicture}[slave]
\node at (0,0) {\includegraphics[width=4.5cm]{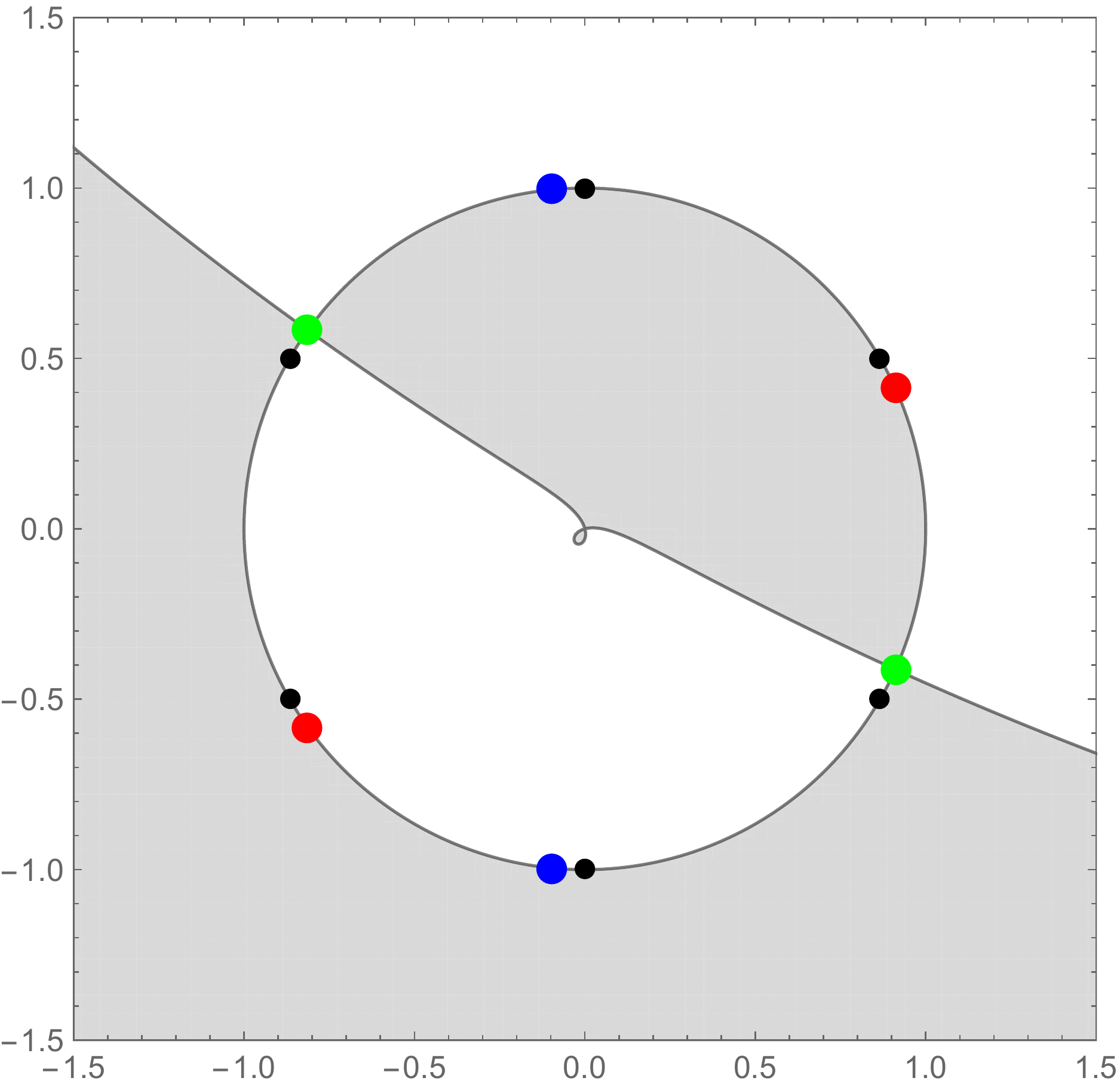}};

\node at (1.77,-0.4) {\tiny $\omega^2 k_1$};
\node at (-1.1,1.25) {\tiny $\omega^2 k_2$};

\end{tikzpicture}
\end{center}
\begin{figuretext}
\label{IIbis fig: Re Phi 21 31 and 32 for zeta=0.7}
From left to right: The signature tables for $\tilde{\Phi}_{21}$, $\tilde{\Phi}_{31}$, and $\tilde{\Phi}_{32}$ for $\tau=\frac{1}{10}$. The grey regions correspond to $\{k \,|\, \re \tilde{\Phi}_{ij}>0\}$ and the white regions to $\{k \,|\, \re \tilde{\Phi}_{ij}<0\}$. The saddle points $k_{1},k_{2}$ of $\tilde{\Phi}_{21}$ are shown blue, the saddle points $\omega k_{1},\omega k_{2}$ of $\tilde{\Phi}_{31}$ red, and the saddle points $\omega^{2} k_{1},\omega^{2} k_{2}$ of $\tilde{\Phi}_{32}$ green. The black dots are the points $i \kappa_j$, $j=1,\ldots,6$.
\end{figuretext}
\end{figure}

The steepest descent analysis proceeds via several transformations $n \mapsto n^{(1)} \mapsto n^{(2)} \mapsto n^{(3)} \mapsto n^{(4)} \mapsto \hat{n}$, such that the RH problems satisfied by $n^{(1)}, n^{(2)}, n^{(3)}, n^{(4)}$, and $\hat{n}$ are equivalent to RH problem \ref{RHn} satisfied by $n$. We will let $\Gamma^{(j)}, \hat{\Gamma}$ and $v^{(j)},\hat{v}$ denote the contours and jump matrices of the RH problems for $n^{(j)}, \hat{n}$, respectively. The symmetries (\ref{vsymm}) and (\ref{nsymm}) will be preserved by each transformation, so that, for $j = 1, 2,3,4$,
\begin{align}\label{vjsymm}
& v^{(j)}(x,t,k) = \mathcal{A} v^{(j)}(x,t,\omega k)\mathcal{A}^{-1}
 = \mathcal{B} v^{(j)}(x,t,k^{-1})^{-1}\mathcal{B}, & & k \in \Gamma^{(j)},
	\\ \label{mjsymm}
& n^{(j)}(x,t, k) = n^{(j)}(x,t,\omega k)\mathcal{A}^{-1}
 = n^{(j)}(x,t, k^{-1}) \mathcal{B}, & & k \in \C \setminus \Gamma^{(j)},
\end{align}
and similarly for $\hat{n}$ and $\hat{v}$. We will see that six local parametrices near the saddle points $\{k_{j},\omega k_{j},\omega^{2}k_{j}\}_{j=1}^{2}$ as well as a global parametrix $\Delta^{-1}$ contribute to the asymptotics of $u(x,t)$. 
However, due to the symmetries \eqref{vjsymm}--\eqref{mjsymm}, we will only need to construct the local parametrix near $k_{1}$ explicitly, and we will only need to define the transformations $n^{(j)} \mapsto n^{(j+1)}$ and $n^{(4)}\mapsto \hat{n}$ in the sector $\mathsf{S}:=\{k \in \mathbb{C}|  \arg k \in [\frac{\pi}{3},\frac{2\pi}{3}]\}$. 
In the end, we will find that $\hat{n}$ satisfies a small-norm RH problem whose solution can be expressed as an integral over the contour $\hat{\Gamma}$.
This leads to an asymptotic formula for $\hat{n}^{(1)}(x,t):=\lim_{k\to \infty} k(\hat{n}(x,t,k) - (1,1,1))$, and hence, by inverting the above transformations, an asymptotic formula also for $n_3^{(1)}(x,t)$. By substituting this formula into (\ref{recoveruvn}), we obtain the formula of Theorem \ref{asymptoticsth}.

The following properties of the functions $r_1$ and $r_2$ established in \cite[Theorem 2.3]{CLmain} will be used repeatedly throughout the proof: $r_1 \in C^\infty(\hat{\Gamma}_{1})$, $r_2 \in C^\infty(\hat{\Gamma}_{4}\setminus \{\pm \omega^{2}\})$, $r_{1}(\kappa_{j})\neq 0$ for $j=1,\ldots,6$, $r_{2}(k)$ has simple zeros at $k=\pm\omega$ and simple poles at $k=\pm \omega^2$, and $r_{1},r_{2}$ have rapid decay as $|k|\to \infty$. Moreover,
\begin{align}
& r_{1}(\tfrac{1}{\omega k}) + r_{2}(\omega k) + r_{1}(\omega^{2} k) r_{2}(\tfrac{1}{k}) = 0, & & k \in \partial \mathbb{D}\setminus \{\pm \omega\}, \label{r1r2 relation on the unit circle} \\
& r_{2}(k) = \tilde{r}(k)\overline{r_{1}(\bar{k}^{-1})}, \qquad \tilde{r}(k) :=\frac{\omega^{2}-k^{2}}{1-\omega^{2}k^{2}}, & & k \in \hat{\Gamma}_{4}\setminus \{0, \pm \omega^{2}\}, \label{r1r2 relation with kbar symmetry} \\
& r_{1}(1) = r_{1}(-1) = 1, \qquad r_{2}(1) = r_{2}(-1) = -1. \label{r1r2at0}
\end{align}

%The five other local parametrices and the transformations in the rest of the complex plane will be defined using the $\mathcal{A}$- and $\mathcal{B}$-symmetries. 

\section{The $n \to n^{(1)}$ transformation}\label{nton1sec}
Let $\tau = t/x$ and note that $x \tilde{\Phi}_{ij}(\tau, k) = \theta_{ij}$.
We will use the notation $\mathcal{I} := [0, \tau_{\max}]$, where $\tau_{\max} \in (0,1)$ is a constant, so that Sector I corresponds to $\tau \in \mathcal{I}$.
The jump matrices $v_{1''}$ and $v_{4'}$ involve the off-diagonal entries $r_{1}(\frac{1}{k})e^{x\tilde{\Phi}_{21}(\tau,k)} $ and $-r_{2}(\frac{1}{k})e^{-x\tilde{\Phi}_{21}(\tau,k)}$, respectively. As long as $\tau > 0$ stays away from $0$, the exponentials $e^{x\tilde{\Phi}_{21}(\tau,k)}$ and $e^{-x\tilde{\Phi}_{21}(\tau,k)}$ are exponentially small as $x \to \infty$ for $k \in (0,i)$ and $k \in (i,i\infty)$, respectively. However, 
$$\tilde{\Phi}_{21}(\tau, k) = \frac{k^2 - 1}{2k} - \frac{k^{4}-1}{4k^2}\tau,$$
so that $\tilde{\Phi}_{21}(0, i \im k) =  i\frac{(\im k)^2 + 1}{2\im k}$ is pure imaginary for all $k \in i\R$, meaning that the above exponentials lose their decay as $\tau \to 0$.
The goal of the first transformation $n \to n^{(1)}$ is to deform the jumps $v_{1''}$ and $v_{4'}$ into regions where the exponentials retain their decay as $x \to \infty$ uniformly for small $\tau$. For this purpose, we need to introduce analytic approximations of the functions $r_1$ and $r_2$.
Let $V_1 \subset \C$ be the open set (see Figure \ref{V1fig})
\begin{align*}
V_{1} = & \; \{k \in \C \,| \arg k \in (-\tfrac{\pi}{2},-\tfrac{\pi}{3}), \; |k| >1 \}.
\end{align*}

\begin{figure}
\begin{center}
\begin{tikzpicture}[master]
\node at (0,0) {\includegraphics[width=5cm]{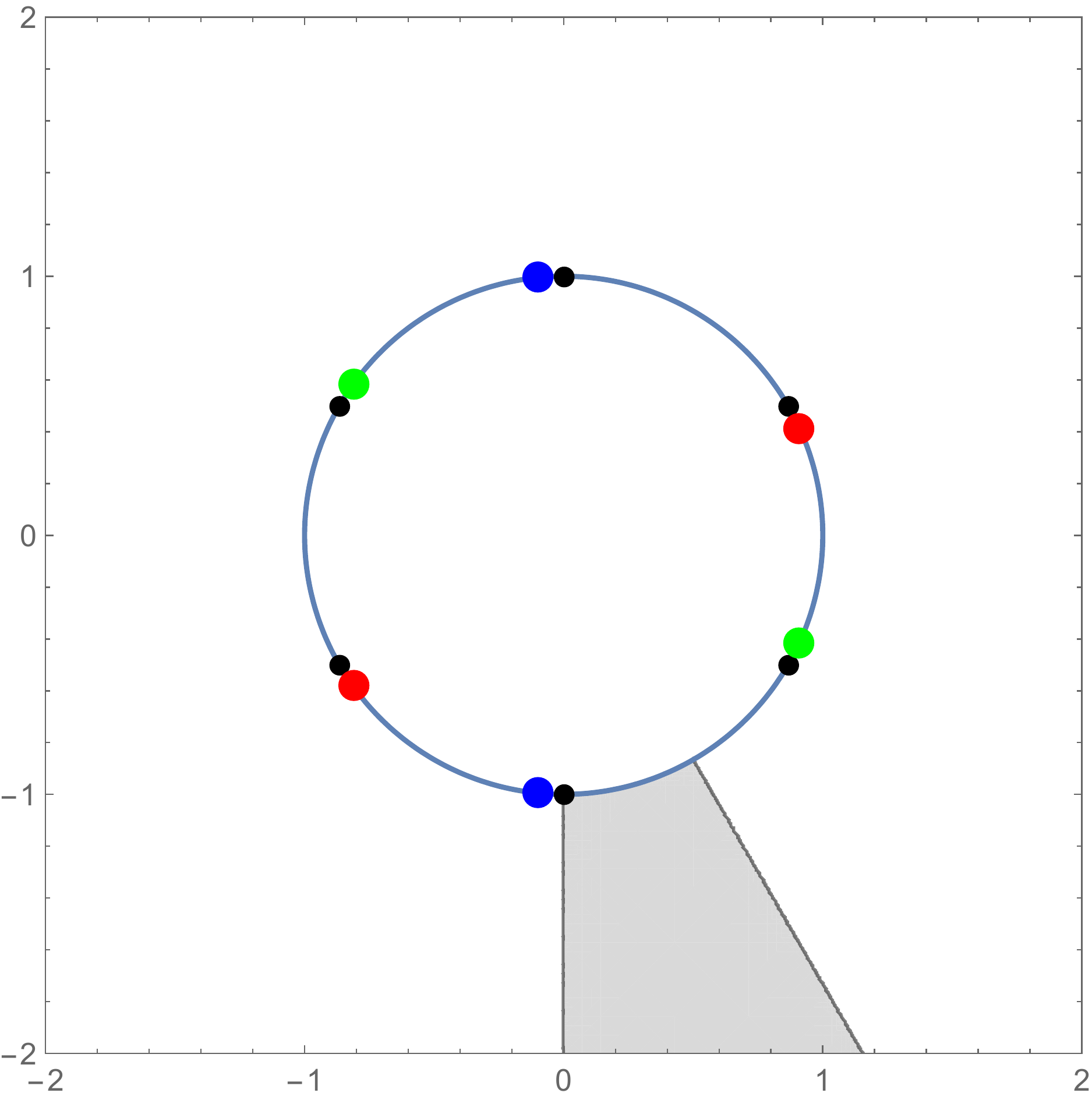}};
\node at (0.5,-1.5) {\tiny $V_{1}$};
\end{tikzpicture}
\end{center}
\begin{figuretext} \label{V1fig} 
The open subset $V_{1}$ of the complex $k$-plane.
\end{figuretext}
\end{figure}

\begin{lemma}[Decomposition lemma I]\label{decompositionlemma1}
For every integer $N \geq 1$, there exists a decomposition
\begin{align}\label{r1decomposition1}
& r_{1}(k) = \tilde{r}_{1,a}(x,k) + \tilde{r}_{1,r}(x,k), & & k \in \partial V_{1} \cap i\R =  (-i\infty, -i], 
\end{align}
such that the functions $\tilde{r}_{1,a}, \tilde{r}_{1,r}$ have the following properties:
\begin{enumerate}[$(a)$]
\item \label{decompositionlemma1parta}
For each $x \geq 1$, $\tilde{r}_{1,a}(x,k)$ is defined and continuous for $k \in \bar{V}_1$ and analytic for $k \in V_1$.

\item \label{decompositionlemma1partb}
For $x \geq 1$ and $\tau \in \mathcal{I}$, the function $\tilde{r}_{1,a}$ satisfies
\begin{align} 
& \Big| \tilde{r}_{1,a}(x,k)-\sum_{j=0}^{N}\frac{r_{1}^{(j)}(k_{\star})}{j!}(k-k_{\star})^{j}  \Big| \leq C |k-k_{\star}|^{N+1} e^{\frac{x}{4}|\re \tilde{\Phi}_{21}(\tau,k)|}, & & k \in \bar{V}_{1}, 
	\\ \label{tilder1aatinfinity}
& | \tilde{r}_{1,a}(x,k) | \leq \frac{C}{1+|k|^{N+1}} e^{\frac{x}{4}|\re \tilde{\Phi}_{21}(\tau,k)|}, & & k \in \bar{V}_{1}, 
\end{align}
where $k_{\star} = -i$ and the constant $C$ is independent of $x, \tau, k$.

\item \label{decompositionlemma1partc}
For each $1 \leq p \leq \infty$, the $L^p$-norm of $\tilde{r}_{1,r}(x,\cdot)$ on $\partial V_{1} \cap i\R$ is $O(x^{-N})$ as $x \to \infty$.
\end{enumerate}
\end{lemma}

\begin{remark}
We will later need to introduce yet another decomposition of $r_1$, see Lemma \ref{decompositionlemma2}. We use tildes on the functions $\tilde{r}_{1,a}$ and $\tilde{r}_{1,r}$ in Lemma \ref{decompositionlemma1} to distinguish them from the functions $r_{1,a}$ and $r_{1,r}$ in Lemma \ref{decompositionlemma2}. The functions $\tilde{r}_{1,a}$ and $r_{1,a}$ have partially overlapping domains of definition, but are in general not equal on the intersection of these domains. Since the derivations of Lemma \ref{decompositionlemma1} and Lemma \ref{decompositionlemma2} use different phase functions ($-i\tilde{\Phi}(0,k)$ and $-i\tilde{\Phi}(\tau,k)$, respectively), it seems difficult to construct $\tilde{r}_{1,a}$ in such a way that it coincides with $r_{1,a}$ on this intersection.
\end{remark}

\begin{proof}[Proof of Lemma \ref{decompositionlemma1}]
The proof follows the idea of \cite{DZ1993} but is nonstandard because the real part of $\tilde{\Phi}_{21}(\tau, k)$ is nonzero for $k \in (-i\infty, -i)$ whenever $\tau > 0$. We therefore present a proof, whose main idea is to use $-i\tilde{\Phi}_{21}(0, k)$ instead of $-i\tilde{\Phi}_{21}(\tau, k)$ as the phase function. 

Let $M \geq N + 1$ be a large integer.
There exists a rational function $f_0(k)$ with no poles in $\bar{V}_1$ such that $r_1$ and $f_0$ coincide to order $4M$ at $k = -i$ and such that $f_0(k) = O(k^{-4M})$ as $(-i\infty, -i] \ni k \to \infty$. Since $r_1 \in C^\infty((-i\infty,-i])$ has rapid decay as $k$ tends to $\infty$, $r_1$ and $f_0$ coincide to order $4M$ also at $k = \infty$.
Let $f_{1}(k) = r_1(k) - f_0(k)$. The map $k \mapsto \phi = \phi(k)$ where
\begin{align}\label{phidef}
 & \phi = -i\tilde{\Phi}_{21}(0, k) = -\frac{i}{2} \left(k-\frac{1}{k}\right) = \frac{1 + (\im k)^2}{2\im k}
\end{align}  
is a bijection  $(-i\infty, -i] \to (-\infty, -1]$, so we may define
\begin{align}\label{Fdef2}
F(\phi) := \begin{cases} \frac{k^{2M}}{(k-k_\star)^M} f_{1}(k), &  \phi \leq -1, \\
0, & \phi > -1,
\end{cases}
\end{align}
with $k_\star := -i$.
The function $F(\phi)$ is smooth for $\phi \in \R \setminus \{-1\}$. Also, for $n \geq 1$, 
\begin{align}\label{dnFdphin}
F^{(n)}(\phi) = \bigg(\frac{2ik^2}{(k-k_\star)(k-i)} \frac{d }{d k}\bigg)^n 
\bigg[\frac{k^{2M}}{(k-k_\star)^M} f_{1}(k)\bigg], \qquad \phi \leq -1,
\end{align}
so by choosing $M$ large enough, we can achieve that $F$ belongs to the Sobolev space $H^{N+1}(\R)$.
We conclude that 
\begin{align}\label{Fhatdef}
\hat{F}(s) = \frac{1}{2\pi} \int_{\R} F(\phi) e^{-i\phi s} d\phi, \qquad
F(\phi) =  \int_{\R} \hat{F}(s) e^{i\phi s} ds
\end{align}
and, by the Plancherel theorem, $\|s^{N+1} \hat{F}(s)\|_{L^2(\R)} = \|\frac{F^{(N+1)}(\phi)}{\sqrt{2\pi}}\|_{L^2(\R)} < \infty$.
By (\ref{Fdef2}) and (\ref{Fhatdef}),
$$ \frac{(k-k_\star)^M}{k^{2M}}\int_{\R} \hat{F}(s) e^{s \tilde{\Phi}_{21}(0, k)} ds 
= f_{1}(k)$$
for $k \in (-i\infty, -i]$. In particular, $f_{1}(k) = f_a(x,k) + f_r(x,k)$ for $x \geq 1$ and $k \in (-i\infty, -i]$, where
\begin{align*}
& f_a(x,k) := \frac{(k-k_\star)^M}{k^{2M}}\int_{-\infty}^{\frac{x}{4}} \hat{F}(s) e^{s\tilde{\Phi}_{21}(0, k)} ds, \qquad k \in \bar{V}_1,  
	\\
& f_r(x,k) := \frac{(k-k_\star)^M}{k^{2M}}\int_{\frac{x}{4}}^{+\infty} \hat{F}(s) e^{s\tilde{\Phi}_{21}(0, k)} ds,\qquad k \in (-i\infty, -i].
\end{align*}
Since $\re \tilde{\Phi}_{21}(0, k) = 0$ for $k \in (-i\infty, -i]$, we have
\begin{align}\nonumber
|f_r(x,k)| & \leq \frac{(k-k_\star)^M}{k^{2M}}  \int_{\frac{x}{4}}^{+\infty} s^{N+1} |\hat{F}(s)| s^{-N-1} ds
 	\\ \label{frest}
& \leq \frac{C}{1 + |k|^M}  \| s^{N+1} \hat{F}(s)\|_{L^2(\R)} \sqrt{\int_{\frac{x}{4}}^{+\infty} s^{-2N-2} ds}    \leq \frac{C}{1 + |k|^M} x^{-N-\frac{1}{2}}
\end{align}
for $x \geq 1$ and $k \in (-i\infty, -i]$. Hence the $L^1$ and $L^\infty$ norms of $f_r(x,\cdot)$ on $(-i\infty, -i]$ are $O(x^{-N-\frac{1}{2}})$. 
On the other hand, $f_a(x, \cdot)$ is clearly continuous on $\bar{V}_1$ and analytic in $V_1$. 
Since $\re \tilde{\Phi}_{21}(\tau, k) \geq \re \tilde{\Phi}_{21}(0, k) \geq 0$ for $k \in \bar{V}_1$ and $\tau \in \mathcal{I}$ (cf. Figure \ref{IIbis fig: Re Phi 21 31 and 32 for zeta=0.7}), it also holds that
\begin{align}\nonumber
 |f_a(x,k)| 
&\leq \frac{(k-k_\star)^M}{k^{2M}} \|\hat{F}\|_{L^1(\R)}  \sup_{s \leq \frac{x}{4}} e^{s \re \tilde{\Phi}_{21}(0, k)}
\leq C\frac{(k-k_\star)^M}{k^{2M}}  e^{\frac{x}{4} \re \tilde{\Phi}_{21}(0, k)} 
	\\ \label{faest}
& \leq C(k-k_\star)^M  e^{\frac{x}{4} \re \tilde{\Phi}_{21}(\tau, k)} 
\end{align}
for $x \geq 1$, $\tau \in \mathcal{I}$, and $k \in \bar{V}_1$.
Since $M \geq N+1$, it follows that the assignments
\begin{align*}
& \tilde{r}_{1,a}(x,k) := f_0(k) + f_a(x,k) & & \text{for} \quad k \in \bar{V}_1,
	\\
& \tilde{r}_{1,r}(x,k) := f_r(x,k)  & & \text{for} \quad  k \in (-i\infty, -i],
\end{align*}
yield a decomposition of $r_1$ with the desired properties. 
\end{proof}

Lemma \ref{decompositionlemma1} establishes a decomposition of $r_{1}$. The symmetry \eqref{r1r2 relation with kbar symmetry} then yields an analogous decomposition $r_{2}=\tilde{r}_{2,a}+\tilde{r}_{2,r}$ of $\tilde{r}_{2}$ as follows:
\begin{align*}
& \tilde{r}_{2,a}(x,k) := \tilde{r}(k)\overline{\tilde{r}_{1,a}(x,\bar{k}^{-1})}, \quad k \in V_{2};
& & \tilde{r}_{2,r}(x,k) := \tilde{r}(k)\overline{\tilde{r}_{1,r}(x,\bar{k}^{-1})}, \quad k \in (-i,0),
\end{align*}
where $V_{2}:=\{k|\bar{k}^{-1}\in V_{1}\}$.

We are now in a position to define the first transformation $n \mapsto n^{(1)}$. As explained in Section \ref{overviewsec}, we will focus our attention on the sector $\mathsf{S}$. Let $\Gamma^{(1)}$ be as in Figure \ref{Gamma1fig}. Note that 
\begin{align*}
\Gamma_{8}^{(1)}:= \{e^{i\theta} \,|\, \theta \in (\arg k_{1},\tfrac{2\pi}{3})\}, \qquad \Gamma_{2}^{(1)}:= \{e^{i\theta} \,|\, \theta \in (\tfrac{\pi}{3},\tfrac{\pi}{2})\}.
\end{align*}

\begin{figure}
\begin{center}
\begin{tikzpicture}[master,scale=0.9]
\node at (0,0) {};
\draw[black,line width=0.65 mm] (0,0)--(30:7.5);
\draw[black,line width=0.65 mm,->-=0.45,->-=0.91] (0,0)--(90:7);
\draw[black,line width=0.65 mm] (0,0)--(150:7.5);
\draw[dashed,black,line width=0.15 mm] (0,0)--(60:7.5);
\draw[dashed,black,line width=0.15 mm] (0,0)--(120:7.5);

\draw[black,line width=0.65 mm] ([shift=(30:3*1.5cm)]0,0) arc (30:150:3*1.5cm);
\draw[black,arrows={-Triangle[length=0.27cm,width=0.18cm]}]
($(73:3*1.5)$) --  ++(-15:0.001);

\node at (75.5:3.2*1.5) {\small $2$};

\node at (83.2:1.95*1.5) {\small $1_r$};

\node at (86.5:6.15) {\small $4_r$};

\draw[black,arrows={-Triangle[length=0.27cm,width=0.18cm]}]
($(90.5:3*1.5)$) --  ++(90.5-90:0.001);
\draw[black,arrows={-Triangle[length=0.27cm,width=0.18cm]}]
($(116-4:3*1.5)$) --  ++(113+90-4:0.001);
\node at (109:3.2*1.5) {\small $8$};
\node at (93:3.2*1.5) {\small $5$};

%Figure made for \zeta = 6.1
\draw[blue,fill] (96.9725:4.5) circle (0.12cm);
\draw[green,fill] (143.028:4.5) circle (0.12cm);
\node at (96:4.08) {\small $k_{1}$};
\node at (140:4.0) {\small $\omega^{2}k_2$};

\draw[black,line width=0.65 mm,->-=0.91] (90:4.5)--(70:6.25)--(69.5:7.5);
\draw[black,line width=0.65 mm,->-=0.45] (70:0)--(70:3.24)--(90:4.5);
\node at (79.5:2.0) {\small $1_a$};
\node at (72.7:6.9) {\small $4_a$};

\draw[black,line width=0.65 mm] (30:4.5)--(50:6.25)--(50.5:7.5);
\draw[black,line width=0.65 mm] (50:0)--(50:3.24)--(30:4.5);

\end{tikzpicture}
\end{center}
\begin{figuretext} \label{Gamma1fig}
The contour $\Gamma^{(1)}$ (solid), the boundary of $\mathsf{S}$ (dashed), and the saddle points $k_{1}$ (blue) and $\omega^{2}k_{2}$ (green).
\end{figuretext}
\end{figure}

The matrices $v_{1''}$ and $v_{4'}$ admit the factorizations
$$v_{1''} = v_{1_a}^{(1)}v_{1_r}^{(1)}, \qquad 
v_{4'} = v_{4_a}^{(1)}v_{4_r}^{(1)},$$
where
\begin{align}\nonumber
& v_{1_a}^{(1)} =
\begin{pmatrix}
1 & 0 & 0 \\
\tilde{r}_{1,a}(\frac{1}{k})e^{\theta_{21}} & 1 & 0 \\
0 & 0 & 1
\end{pmatrix}, && v_{1_r}^{(1)} =
\begin{pmatrix}
1 & 0 & 0 \\
\tilde{r}_{1,r}(\frac{1}{k})e^{\theta_{21}} & 1 & 0 \\
0 & 0 & 1
\end{pmatrix},
	\\\label{v1arv4ardef}
& v_{4_a}^{(1)} = \begin{pmatrix}
1 & -\tilde{r}_{2,a}(\frac{1}{k})e^{-\theta_{21}} & 0 \\
0 & 1 & 0 \\
0 & 0 & 1
\end{pmatrix}, &&
v_{4_r}^{(1)} = \begin{pmatrix}
1 & -\tilde{r}_{2,r}(\frac{1}{k})e^{-\theta_{21}} & 0 \\
0 & 1 & 0 \\
0 & 0 & 1
\end{pmatrix}.
\end{align}

The function $n^{(1)}$ is defined by
\begin{align}\label{n1def}
n^{(1)}(x,t,k) = n(x,t,k)G^{(1)}(x,t,k),\qquad k \in \C \setminus \Gamma^{(1)},
\end{align}
where the function $G^{(1)}$ is analytic in $\mathbb{C}\setminus \Gamma^{(1)}$. It is given for $k \in \mathsf{S}$ by
\begin{align}\label{G1def}
G^{(1)}  =  \begin{cases} 
v_{1_a}^{(1)}, & k \mbox{ on the $-$ side of }\Gamma_{1_r}^{(1)}, \\[0.1cm]
v_{4_a}^{(1)}, &  k \mbox{ on the $-$ side of }\Gamma_{4_r}^{(1)},  \\
I, & \mbox{otherwise},
\end{cases}
\end{align}
and extended to $\mathbb{C}\setminus \Gamma^{(1)}$ by means of the symmetry
\begin{align}\label{symmetry of G1}
G^{(1)}(x,t, k) = \mathcal{A} G^{(1)}(x,t,\omega k)\mathcal{A}^{-1}
 = \mathcal{B} G^{(1)}(x,t, k^{-1}) \mathcal{B}.
\end{align}
The next lemma follows from the signature tables of Figure \ref{IIbis fig: Re Phi 21 31 and 32 for zeta=0.7} and Lemma \ref{decompositionlemma1}.

\begin{lemma}\label{G1lemma}
$G^{(1)}(x,t,k)$ and $G^{(1)}(x,t,k)^{-1}$ are uniformly bounded for $k \in \mathbb{C}\setminus \Gamma^{(1)}$, $x \geq 1$, and $\tau \in \mathcal{I}$.
\end{lemma}

Using \eqref{n1def}, we infer that $n^{(1)}_{+}=n^{(1)}_{-}v_{j}^{(1)}$ on $\Gamma^{(1)}_{j}$, where $v_{j}^{(1)}$ is given by \eqref{v1arv4ardef} for $j=1_a, 1_r, 4_a, 4_r$ and for $j = 2,5,8$ by
\begin{align}\nonumber
v_{2}^{(1)} = v_9, \qquad v_{5}^{(1)} = v_7^{-1}, \qquad v_{8}^{(1)} = v_7. 
\end{align}

The $L^1$ and $L^\infty$ norms of $v^{(1)} - I$ on $\Gamma_{1_r}^{(1)} \cup \Gamma_{4_r}^{(1)}$ are uniformly of order $O(x^{-N})$ in Sector I as a consequence of Lemma \ref{decompositionlemma1}. The estimate (\ref{tilder1aatinfinity}) of $\tilde{r}_{1,a}$ ensures that $n^{(1)}$ has the same behavior as $n$ as $k \to \infty$ up to terms of $O(k^{-N})$.

\section{The $n^{(1)} \to n^{(2)}$ transformation}\label{n1ton2sec}
To implement the transformation $n^{(1)} \to n^{(2)}$, we need another set of decompositions of the functions $r_{1},r_{2}$. We also need decompositions of the functions $\hat{r}_{1},\hat{r}_{2}$ defined for all $k \in \partial \D$ with $\arg k \in [\tfrac{\pi}{2},\arg k_{1}]$ by
\begin{align}\label{def of rhat}
& \hat{r}_{j}(k) := \frac{r_{j}(k)}{1+r_{1}(k)r_{2}(k)} = \frac{r_{j}(k)}{1+\tilde{r}(k)|r_{1}(k)|^{2}}, && j=1,2,
\end{align}
where the second equality follows from \eqref{r1r2 relation with kbar symmetry}. Since $\arg k_{1} \in [\frac{\pi}{2},\frac{2\pi}{3})$ for $\tau \in [0,1)$, and since $\tilde{r}(e^{i\theta})>0$ for $\theta \in [\frac{\pi}{2},\frac{2\pi}{3})$, we have $1+r_{1}(k)r_{2}(k)=1+\tilde{r}(k)|r_{1}(k)|^{2}>1$ for $|k|=1, \arg k \in [\tfrac{\pi}{2},\arg k_{1}]$. Thus $\hat{r}_{1}$, $\hat{r}_{2}$ are well-defined.
Given $K>1$, we define the open sets $U_1 = U_1(\tau,K)$ and $\hat{U}_1 = \hat{U}_1(\tau,K) \subset \C$ by (see Figure \ref{fig: U1 and U2})
\begin{align*}
U_{1} = & \; \big(\{k \,| \arg k \in [-\pi,-\tfrac{5\pi}{6}) \cup (\arg k_{1},\pi], \; K^{-1}<|k|<1 \}
	\\
& \; \cup
\{k \,| \arg k \in (-\tfrac{\pi}{2}, \tfrac{\pi}{3}), \; 1<|k|<K \}\big) \cap \{k  \,|\, \re \Phi_{21}>0\}, \\
\hat{U}_{1} = & \; \{k \,| \arg k \in (\tfrac{\pi}{4},\arg k_{1}), \; 1<|k|<K \} \cap \{k \,|\, \re \Phi_{21}>0\}.
\end{align*}

\begin{figure}
\begin{center}
\begin{tikzpicture}[master]
\node at (0,0) {\includegraphics[width=5cm]{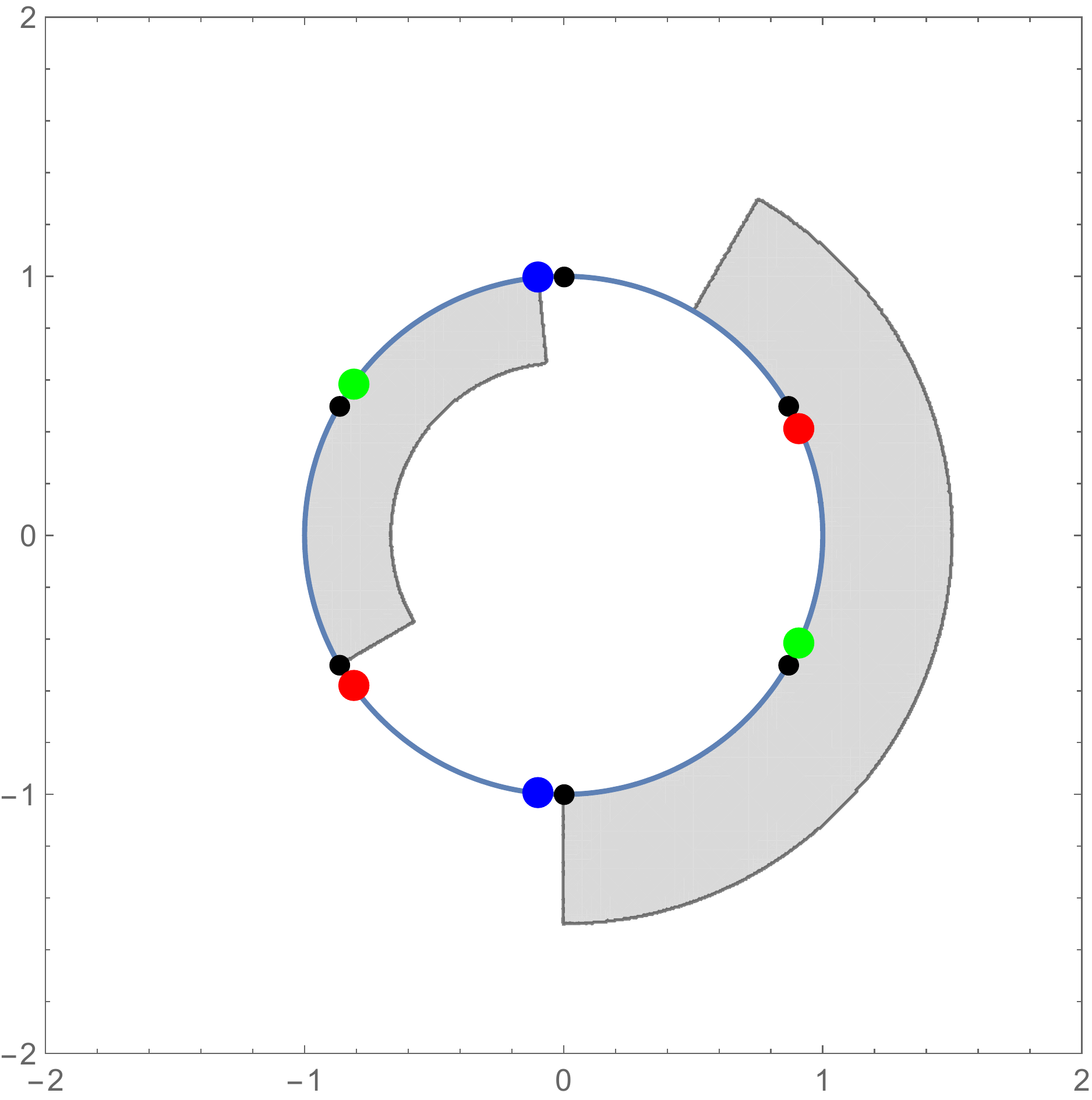}};
\node at (-0.87,0.25) {\tiny $U_{1}$};
\node at (1.2,-0.87) {\tiny $U_{1}$};
\end{tikzpicture} \hspace{0.1cm} \begin{tikzpicture}[slave]
\node at (0,0) {\includegraphics[width=5cm]{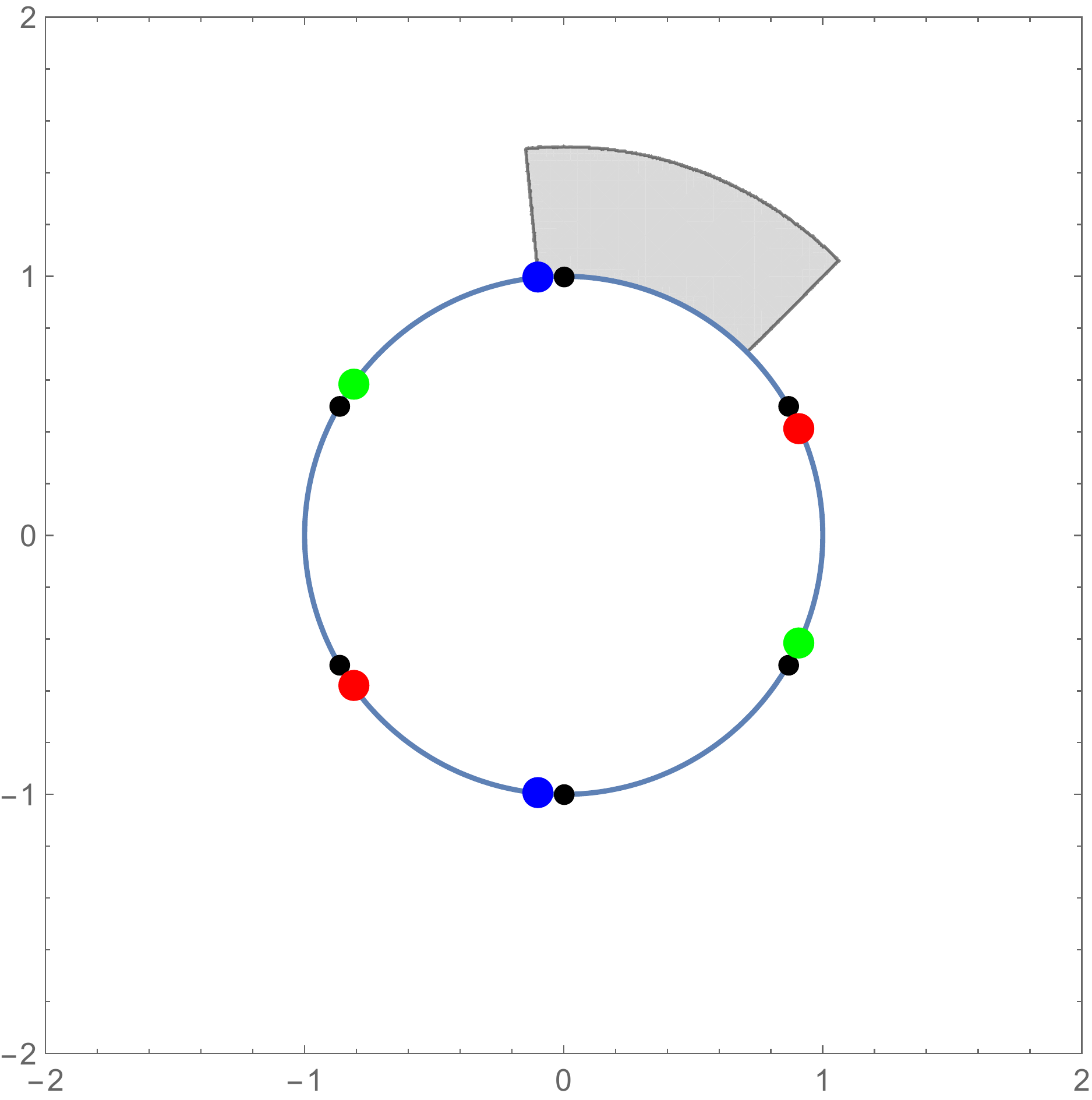}};
\node at (0.5,1.5) {\tiny $\hat{U}_{1}$};
\end{tikzpicture}
\end{center}
\begin{figuretext} \label{fig: U1 and U2} 
The open subsets $U_{1}$ and $\hat{U}_{1}$ of the complex $k$-plane.
\end{figuretext}
\end{figure}

\begin{lemma}[Decomposition lemma II]\label{decompositionlemma2}
For every integer $N \geq 1$, there exist $K>1$ and decompositions
\begin{align*}
& r_{1}(k) = r_{1,a}(x, t, k) + r_{1,r}(x, t, k), & & k \in \partial U_{1} \cap \partial \mathbb{D}, \\
& \hat{r}_{1}(k) = \hat{r}_{1,a}(x, t, k) + \hat{r}_{1,r}(x, t, k), & & k \in \partial \hat{U}_{1} \cap \partial \mathbb{D},
\end{align*}
such that $r_{1,a}, r_{1,r}, \hat{r}_{1,a}, \hat{r}_{1,r}$ satisfy the following properties:
\begin{enumerate}[$(a)$]
\item 
For each $x \geq 1$ and $\tau \in \mathcal{I}$, $r_{1,a}(x, t, k)$ is defined and continuous for $k \in \bar{U}_1$ and analytic for $k \in U_1$, while $\hat{r}_{1,a}(x, t, k)$ is defined and continuous for $k \in \bar{\hat{U}}_1$ and analytic for $k \in \hat{U}_1$.

\item For $x \geq 1$ and $\tau \in \mathcal{I}$, the functions $r_{1,a}$ and $\hat{r}_{1,a}$ obey
\begin{align*} 
& \Big| r_{1,a}(x, t, k)-\sum_{j=0}^{N}\frac{r_{1}^{(j)}(k_{\star})}{j!}(k-k_{\star})^{j}  \Big| \leq C |k-k_{\star}|^{N+1} e^{\frac{x}{4}|\re \tilde{\Phi}_{21}(\tau,k)|}, & & k \in \bar{U}_{1}, \, k_{\star} \in \mathcal{R}_{1}, 
	\\ 
& \Big| \hat{r}_{1,a}(x, t, k)-\sum_{j=0}^{N}\frac{\hat{r}_{1}^{(j)}(k_{\star})}{j!}(k-k_{\star})^{j} \Big|  \leq C |k-k_{\star}|^{N+1} e^{\frac{x}{4}|\re \tilde{\Phi}_{21}(\tau,k)|},  & & k \in \bar{\hat{U}}_{1}, \, k_{\star} \in \hat{\mathcal{R}}_{1}, 
\end{align*}
where $\mathcal{R}_{1}=\{k_{1},\pm \omega, \omega^{2}k_{2},\pm e^{\frac{\pi i}{6}},\pm e^{-\frac{\pi i}{6}},\pm 1,-i,\omega^{2}k_{1},\omega k_{2}\}$, $\hat{\mathcal{R}}_{1}=\{i,k_{1}\}$, and $C > 0$ is independent of $x, \tau, k$. In fact, for $k_\star = k_1$, the following stronger estimates hold: 
\begin{align*} 
& \Big| r_{1,a}(x, t, k)-\sum_{j=0}^{N}\frac{r_{1}^{(j)}(k_{\star})}{j!}(k-k_{\star})^{j}  \Big| \leq C_N(\tau) |k-k_{\star}|^{N+1} e^{\frac{x}{4}|\re \tilde{\Phi}_{21}(\tau,k)|}, & & k \in \bar{U}_{1}, 
	\\ 
& \Big| \hat{r}_{1,a}(x, t, k)-\sum_{j=0}^{N}\frac{\hat{r}_{1}^{(j)}(k_{\star})}{j!}(k-k_{\star})^{j} \Big| \leq C_N(\tau) |k-k_{\star}|^{N+1} e^{\frac{x}{4}|\re \tilde{\Phi}_{21}(\tau,k)|},  & & k \in \bar{\hat{U}}_{1}, 
\end{align*}
where $C_N(\tau) \geq 0$ is a smooth function of $\tau$ which is independent of $x, k$ and which vanishes to all orders at $\tau = 0$.

\item For each $1 \leq p \leq \infty$, the $L^p$-norm of $r_{1,r}(x,t,\cdot)$ on $\partial U_{1} \cap \partial \mathbb{D}$ is $O(x^{-N})$ and the $L^p$-norm of $\hat{r}_{1,r}(x,t,\cdot)$ on $\partial \hat{U}_{1} \cap \partial \mathbb{D}$ is $O(x^{-N})$ as $x \to \infty$ uniformly for $\tau \in \mathcal{I}$.
\end{enumerate}
\end{lemma}
\begin{proof}
On each arc of $\partial U_{1} \cap \partial \mathbb{D}$ and $\partial \hat{U}_{1} \cap \partial \mathbb{D}$, the function $\theta \mapsto -i \tilde{\Phi}_{21}(\tau,e^{i\theta})=(1-\tau \cos \theta)\sin \theta$ is real-valued and monotone. The statement therefore follows using the method introduced in \cite{DZ1993}; see also the proof of Lemma \ref{decompositionlemma1}.
\end{proof}

Lemma \ref{decompositionlemma2} provides decompositions of $r_{1}$ and $\hat{r}_{1}$; using the symmetry \eqref{r1r2 relation with kbar symmetry}, we now deduce decompositions of $r_{2}$ and $\hat{r}_{2}$. Let $U_{2}:=\{k|\bar{k}^{-1}\in U_{1}\}$ and $\hat{U}_{2}:=\{k|\bar{k}^{-1}\in \hat{U}_{1}\}$. We define decompositions $r_{2}=r_{2,a}+r_{2,r}$ and $\hat{r}_{2}=\hat{r}_{2,a}+\hat{r}_{2,r}$ by
\begin{align*}
& r_{2,a}(k) := \tilde{r}(k)\overline{r_{1,a}(\bar{k}^{-1})}, \quad k \in U_{2},
& & r_{2,r}(k) := \tilde{r}(k)\overline{r_{1,r}(\bar{k}^{-1})}, \quad k \in \partial U_{2}\cap \partial \mathbb{D},
	\\
& \hat{r}_{2,a}(k) := \tilde{r}(k)\overline{\hat{r}_{1,a}(\bar{k}^{-1})}, \quad k \in \hat{U}_{2},
& & \hat{r}_{2,r}(k) := \tilde{r}(k)\overline{\hat{r}_{1,r}(\bar{k}^{-1})}, \quad k \in \partial \hat{U}_{2}\cap \partial \mathbb{D},
\end{align*}
where we have omitted the $(x,t)$-dependence for conciseness.

On $\Gamma_{2}^{(2)}$, we factorize $v_{2}^{(1)}$ as follows:
\begin{align*}
v_{2}^{(1)} = v_{3}^{(2)}v_{2}^{(2)}v_{1}^{(2)},
\end{align*}
where
\begin{align}
& v_{3}^{(2)} = \begin{pmatrix}
1 & 0 & -r_{2,a}(\omega^{2}k)e^{-\theta_{31}} \\
r_{1,a}(\frac{1}{k})e^{\theta_{21}} & 1 & r_{1,a}(\omega k)e^{-\theta_{32}} \\
0 & 0 & 1
\end{pmatrix}, \; v_{1}^{(2)} = \begin{pmatrix}
1 & r_{2,a}(\frac{1}{k})e^{-\theta_{21}} & 0 \\
0 & 1 & 0  \\
-r_{1,a}(\omega^{2}k)e^{\theta_{31}} & r_{2,a}(\omega k)e^{\theta_{32}} & 1
\end{pmatrix}, \nonumber \\
& v_{2}^{(2)} = I + v_{2,r}^{(2)}, \qquad v_{2,r}^{(2)}= \begin{pmatrix}
r_{1,r}(\omega^{2}k)r_{2,r}(\omega^{2}k) & g_{2}(\omega k)e^{-\theta_{21}} & -r_{2,r}(\omega^{2}k)e^{-\theta_{31}} \\
g_{1}(\omega k)e^{\theta_{21}} & g(\omega k) & h_{1}(\omega k)e^{-\theta_{32}} \\
-r_{1,r}(\omega^{2}k)e^{\theta_{31}} & h_{2}(\omega k)e^{\theta_{32}} & 0
\end{pmatrix}, \label{vp1p 123}
\end{align}
and
\begin{align*}
h_{1}(k) = & \; r_{1,r}(k) + r_{1,a}(\tfrac{1}{\omega^{2}k})r_{2,r}(\omega k), \qquad h_{2}(k) =  r_{2,r}(k) + r_{2,a}(\tfrac{1}{\omega^{2}k})r_{1,r}(\omega k), \\
g_{1}(k) = & \; r_{1,r}(\tfrac{1}{\omega^{2}k})-r_{1,r}(\omega k) \big( r_{1,r}(k)+r_{1,a}(\tfrac{1}{\omega^{2}k})r_{2,r}(\omega k) \big), \\
g_{2}(k) = & \; r_{2,r}(\tfrac{1}{\omega^{2}k})-r_{2,r}(\omega k) \big( r_{2,r}(k)+r_{2,a}(\tfrac{1}{\omega^{2}k})r_{1,r}(\omega k) \big), \\
g(k) = & \; r_{1,r}(k)\big(r_{1,r}(\omega k)r_{2,a}(\tfrac{1}{\omega^{2}k})+r_{2,r}(k)\big) \\
& +r_{1,a}(\tfrac{1}{\omega^{2}k})r_{2,r}(\omega k)\big( r_{1,r}(\omega k)r_{2,a}(\tfrac{1}{\omega^{2}k})+r_{2,r}(k) \big) + r_{1,r}(\tfrac{1}{\omega^{2}k})r_{2,r}(\tfrac{1}{\omega^{2}k}).
\end{align*}
On $\Gamma_{5}^{(2)}$, we factorize $v_{5}^{(1)}$ as follows:
\begin{align}
& v_{5}^{(1)} = v_{6}^{(2)}v_{5}^{(2)}v_{4}^{(2)}, \qquad v_{5}^{(2)} = \begin{pmatrix}
1+r_{1}(k)r_{2}(k) & 0 & 0 \\
0 & \frac{1}{1+r_{1}(k)r_{2}(k)} & 0 \\
0 & 0 & 1
\end{pmatrix} + v_{5,r}^{(2)}, \nonumber \\
& v_{6}^{(2)} \hspace{-0.07cm} = \hspace{-0.07cm} \begin{pmatrix}
1 & 0 & r_{1,a}(\frac{1}{\omega^{2} k})e^{-\theta_{31}} \\[0.08cm]
\hat{r}_{2,a}(k)e^{\theta_{21}} & 1 & -r_{2,a}(\frac{1}{\omega k})e^{-\theta_{32}} \\
0 & 0 & 1
\end{pmatrix}, \; v_{4}^{(2)} \hspace{-0.07cm} = \hspace{-0.07cm} \begin{pmatrix}
1 & \hspace{-0.07cm} \hat{r}_{1,a}(k)e^{-\theta_{21}} & \hspace{-0.07cm} 0 \\
0 & \hspace{-0.07cm} 1 & \hspace{-0.07cm} 0 \\
r_{2,a}(\tfrac{1}{\omega^{2} k})e^{\theta_{31}} & \hspace{-0.07cm} -r_{1,a}(\tfrac{1}{\omega k})e^{\theta_{32}} & \hspace{-0.07cm} 1
\end{pmatrix}\hspace{-0.07cm}, \label{vp1p 789}
\end{align}
where we have used \eqref{r1r2 relation on the unit circle}. Finally, on $\Gamma_{8}^{(2)}$, we factorize $v_{8}^{(1)}$ as follows:
\begin{align}
& v_{8}^{(1)} = v_{7}^{(2)}v_{8}^{(2)}v_{9}^{(2)}, \quad v_{8}^{(2)} = I + v_{8,r}^{(2)}, \quad v_{8,r}^{(2)} = \begin{pmatrix}
0 & -r_{1,r}(k)e^{-\theta_{21}} & h_{2}(\omega^{2}k)e^{-\theta_{31}} \\
-r_{2,r}(k)e^{\theta_{21}} & r_{1,r}(k)r_{2,r}(k) & g_{2}(\omega^{2}k)e^{-\theta_{32}} \\
h_{1}(\omega^{2}k)e^{\theta_{31}} & g_{1}(\omega^{2}k)e^{\theta_{32}} & g(\omega^{2}k)
\end{pmatrix} \nonumber \\
& v_{7}^{(2)} = \begin{pmatrix}
1 & 0 & 0 \\
-r_{2,a}(k)e^{\theta_{21}} & 1 & 0 \\
r_{1,a}(\omega^{2}k)e^{\theta_{31}} & r_{1,a}(\frac{1}{\omega k})e^{\theta_{32}} & 1
\end{pmatrix}, \qquad v_{9}^{(2)} = \begin{pmatrix}
1 & -r_{1,a}(k)e^{-\theta_{21}} & r_{2,a}(\omega^{2}k)e^{-\theta_{31}} \\
0 & 1 & r_{2,a}(\frac{1}{\omega k})e^{-\theta_{32}} \\
0 & 0 & 1 
\end{pmatrix}. \label{vp1p 101112}
\end{align}
We do not write down the long expression for $v_{5,r}^{(2)}$, but note that Lemma \ref{decompositionlemma2} yields
\begin{align*}
\| v_{j,r}^{(2)} \|_{(L^{1}\cap L^{\infty})(\Gamma_{j}^{(2)})} = O(x^{-N}) \qquad \mbox{as } x \to \infty, \; j= 2,5,8.
\end{align*}

\begin{figure}
\begin{center}
\begin{tikzpicture}[master,scale=1.1]
\node at (0,0) {};
\draw[black,line width=0.65 mm] (0,0)--(30:7.5);

\draw[black,line width=0.65 mm,->-=0.35,->-=0.59,->-=0.685,->-=0.88] (0,0)--(90:7.5);
\node at (83.8:1.62*1.5) {\scriptsize $1_r$};
\node at (87:2.7*1.5) {\scriptsize $2_{r}$};
\node at (86.9:3.3*1.5) {\scriptsize $3_{r}$};
\node at (87.5:4.28*1.5) {\scriptsize $4_{r}$};

\draw[black,line width=0.65 mm] (0,0)--(150:7.5);
\draw[dashed,black,line width=0.15 mm] (0,0)--(60:7.5);
\draw[dashed,black,line width=0.15 mm] (0,0)--(120:7.5);

\draw[black,line width=0.65 mm] ([shift=(30:3*1.5cm)]0,0) arc (30:150:3*1.5cm);
\draw[black,arrows={-Triangle[length=0.27cm,width=0.18cm]}]($(73:3*1.5)$) --  ++(-17:0.001);

\draw[black,line width=0.65 mm] ([shift=(30:3*1.5cm)]0,0) arc (30:150:3*1.5cm);

\draw[black,line width=0.65 mm,-<-=0.64] ([shift=(50:6.25cm)]0,0) arc (50:70:6.25cm);
\node at (65:4.35*1.5) {\scriptsize $1$};

\node at (75:3.16*1.5) {\scriptsize $2$};

\draw[black,line width=0.65 mm,-<-=0.57] ([shift=(50:3.24cm)]0,0) arc (50:70:3.24cm);
\node at (65:2.32*1.5) {\scriptsize $3$};

\node at (101:3.45*1.5) {\scriptsize $7$};
\node at (103:2.55*1.5) {\scriptsize $9$};

\node at (94.6:3.42*1.5) {\scriptsize $4$};
\node at (94.5:2.63*1.5) {\scriptsize $6$};

%OLD: \draw[black,line width=0.65 mm,-<-=0.13,->-=0.78] (90:3.24)--($(96.9725:3*1.5)+(96.9725-135:0.5)$)--(96.9725:3*1.5)--($(96.9725:3*1.5)+(96.9725-45:0.5)$)--(90:6.25);
\draw[black,line width=0.65 mm,-<-=0.08,->-=0.86] (90:3.9)--(96.9725:3*1.5)--(90:5.3);

\draw[black,line width=0.65 mm,-<-=0.22,->-=0.6] (120:3.24)--($(96.9725:3*1.5)+(96.9725+135:1.1)$)--(96.9725:3*1.5)--($(96.9725:3*1.5)+(96.9725+45:1.1)$)--(120:6.25);
\draw[black,line width=0.65 mm] (120:3.24)--($(143.028:3*1.5)+(143.028-135:1.1)$)--(143.028:3*1.5)--($(143.028:3*1.5)+(143.028-45:1.1)$)--(120:6.25);

\draw[black,line width=0.65 mm] (150:3.9)--(143.028:3*1.5)--(150:5.3);

\draw[black,arrows={-Triangle[length=0.27cm,width=0.18cm]}]
($(90.5:3*1.5)$) --  ++(90.5-90:0.001);
\draw[black,arrows={-Triangle[length=0.27cm,width=0.18cm]}]
($(116-4:3*1.5)$) --  ++(113+90-4:0.001);
\node at (109:3.15*1.5) {\scriptsize $8$};
\node at (91.8:3.12*1.5) {\scriptsize $5$};

\draw[blue,fill] (96.9725:4.5) circle (0.12cm);
\draw[green,fill] (143.028:4.5) circle (0.12cm);
\node at (97:4.19) {\scriptsize $k_{1}$};
\node at (141:4.03) {\scriptsize $\omega^{2}k_2$};

\draw[black,line width=0.65 mm,->-=0.6] (60:4.5)--(60:6.25);
\draw[black,line width=0.65 mm,->-=0.75] (60:3.24)--(60:4.5);
\draw[black,line width=0.65 mm,->-=0.6] (120:4.5)--(120:6.25);
\draw[black,line width=0.65 mm,->-=0.75] (120:3.24)--(120:4.5);

\node at (122.7:5.35) {\scriptsize $1_{s}$};
\node at (123.5:4.0) {\scriptsize $2_{s}$};
\node at (56.8:5.35) {\scriptsize $3_{s}$};
\node at (56.3:4.0) {\scriptsize $4_{s}$};

\draw[black,line width=0.65 mm,->-=0.43] (90:4.5)--(70:6.25)--(69.5:7.5);
\draw[black,line width=0.65 mm,->-=0.83] (70:0)--(70:3.24)--(90:4.5);
\node at (77:3.9) {\scriptsize $1_a$};
\node at (74.5:5.3) {\scriptsize $4_a$};

\draw[black,line width=0.65 mm] (30:4.5)--(50:6.25)--(50.5:7.5);
\draw[black,line width=0.65 mm] (50:0)--(50:3.24)--(30:4.5);

\end{tikzpicture}
\end{center}
\begin{figuretext}\label{Gamma2fig}
The contour $\Gamma^{(2)}$ (solid), the boundary of $\mathsf{S}$ (dashed), and the saddle points $k_{1}$ (blue) and $\omega^{2}k_{2}$ (green).
\end{figuretext}
\end{figure}

Let $\Gamma^{(2)}$ be the contour shown in Figure \ref{Gamma2fig}. The function $n^{(2)}$ is defined by
\begin{align}\label{n2def}
n^{(2)}(x,t,k) = n^{(1)}(x,t,k)G^{(2)}(x,t,k),\qquad k \in \C \setminus \Gamma^{(2)},
\end{align}
where $G^{(2)}$ is analytic in $\mathbb{C}\setminus \Gamma^{(2)}$. It is defined for $k \in \mathsf{S}$ by
\begin{align}\label{G2def}
\hspace{-0.1cm} G^{(2)} \hspace{-0.1cm} = \hspace{-0.1cm} \begin{cases} 
v_{3}^{(2)}, & \hspace{-0.3cm} k \mbox{ on the $-$ side of }\Gamma_{2}^{(2)}, \\
(v_{1}^{(2)})^{-1}, & \hspace{-0.3cm} k \mbox{ on the $+$ side of }\Gamma_{2}^{(2)},  \\
v_{6}^{(2)}, & \hspace{-0.3cm} k \mbox{ on the $-$ side of }\Gamma_{5}^{(2)}, \\
(v_{4}^{(2)})^{-1}, & \hspace{-0.3cm} k \mbox{ on the $+$ side of }\Gamma_{5}^{(2)},
\end{cases} \; G^{(2)} \hspace{-0.1cm} = \hspace{-0.1cm} \begin{cases} 
v_{7}^{(2)}, & \hspace{-0.3cm} k \mbox{ on the $-$ side of }\Gamma_{8}^{(2)}, \\
(v_{9}^{(2)})^{-1}, & \hspace{-0.3cm} k \mbox{ on the $+$ side of }\Gamma_{8}^{(2)},  \\
I, & \hspace{-0.3cm} \mbox{otherwise},
\end{cases}
\end{align}
and extended to $\mathbb{C}\setminus \Gamma^{(2)}$ by means of the symmetries
\begin{align}\label{symmetry of G2}
G^{(2)}(x,t, k) = \mathcal{A} G^{(2)}(x,t,\omega k)\mathcal{A}^{-1}
 = \mathcal{B} G^{(2)}(x,t, k^{-1}) \mathcal{B}.
\end{align}
The next lemma follows from Lemma \ref{decompositionlemma2} and the signature tables of Figure \ref{IIbis fig: Re Phi 21 31 and 32 for zeta=0.7}.

\begin{lemma}\label{lemma:G1p1p}
$G^{(2)}(x,t,k)$ and $G^{(2)}(x,t,k)^{-1}$ are uniformly bounded for $k \in \mathbb{C}\setminus \Gamma^{(2)}$, $x \geq 1$, and $\tau \in \mathcal{I}$. Furthermore, $G^{(2)}(x,t,k)=I$ for all large enough $|k|$.
\end{lemma}

Using \eqref{n2def}, we obtain $n^{(2)}_{+}=n^{(2)}_{-}v_{j}^{(2)}$ on $\Gamma^{(2)}_{j}$, where $v_{j}^{(2)}$ is given by \eqref{vp1p 123} for $j=1,2,3$, by \eqref{vp1p 789} for $j=4,5,6$, by \eqref{vp1p 101112} for $j=7,8,9$, and
\begin{align}\nonumber
v_{1_r}^{(2)} = &\;  v_{1_r}^{(1)}, \qquad
v_{2_r}^{(2)} = v_{1_r}^{(1)} v_{6}^{(2)}, \qquad
v_{3_r}^{(2)} = v_{4_r}^{(1)}(v_{4}^{(2)})^{-1}, \qquad
v_{4_r}^{(2)} = v_{4_r}^{(1)}, 
	\\\nonumber
v_{1_a}^{(2)} = &\; (v_3^{(2)})^{-1}v_{1_a}^{(1)}, \qquad
v_{4_a}^{(2)} = v_1^{(2)} v_{4_a}^{(1)}, 
	  \\ \nonumber
 v_{1_s}^{(2)} = &\; G_-^{(2)}(k)^{-1}G_+^{(2)}(k)
= v_{7}^{(2)}(k)^{-1} \mathcal{A} \mathcal{B} v_{9}^{(2)}(\tfrac{1}{\omega k})^{-1} \mathcal{B} \mathcal{A}^{-1}  
= \begin{pmatrix}
 1 & 0 & 0
 \\
 0 & 1 & 0 \\
 (v_{1_s}^{(2)})_{31} & 0 & 1 \\
\end{pmatrix}, 
	\\ \nonumber
 (v_{1_s}^{(2)})_{31} := &\;  -\big(r_{1,a}(\omega^{2}k)+r_{2,a}(k)r_{1,a}(\tfrac{1}{\omega k})+r_{2,a}(\tfrac{1}{\omega^{2}k})\big) e^{\theta_{31}}.
\end{align}
%In general, $(v_{1_s}^{(2)})_{31} \neq 0$ for $k \in \Gamma_{1_{s}}^{(2)}$, even though $r_{1}(\omega^{2}k)+r_{2}(k)r_{1}(\tfrac{1}{\omega k})+r_{2}(\tfrac{1}{\omega^{2}k})=0$ for $k \in \partial \mathbb{D}$ by \eqref{r1r2 relation on the unit circle}. 
We omit the explicit expressions for $v_{2_{s}}, v_{3_{s}}, v_{4_{s}}$ which are similar to the expression for $v_{1_{s}}$.

\section{The $n^{(2)} \to n^{(3)}$ transformation}
For the third transformation, we will deform contours up and down (see \eqref{G3def} below) using the following factorizations:
\begin{align}\nonumber
& v_{4}^{(2)}\hspace{-0.07cm}=\hspace{-0.07cm}v_{4}^{(3)}v_{4u}^{(3)}, \;\;\; v_{4}^{(3)}\hspace{-0.07cm} =\hspace{-0.07cm} \begin{pmatrix}
1 & \hat{r}_{1,a}(k) e^{-\theta_{21}} & 0 \\
0 & 1 & 0 \\
0 & 0 & 1
\end{pmatrix}\hspace{-0.07cm}, \;\;\; v_{4u}^{(3)}\hspace{-0.07cm} =\hspace{-0.07cm} \begin{pmatrix}
1 & 0 & 0 \\
0 & 1 & 0 \\
r_{2,a}(\frac{1}{\omega^{2}k}) e^{\theta_{31}} & -r_{1,a}(\frac{1}{\omega k})e^{\theta_{32}} & 1
\end{pmatrix}\hspace{-0.07cm}, \nonumber \\
& v_{6}^{(2)}=v_{6d}^{(3)}v_{6}^{(3)}, \quad v_{6d}^{(3)} = \begin{pmatrix}
1 & 0 & r_{1,a}(\frac{1}{\omega^{2} k})e^{-\theta_{31}} \\[0.05cm]
0 & 1 & -r_{2,a}(\frac{1}{\omega k})e^{-\theta_{32}} \\
0 & 0 & 1
\end{pmatrix}, \quad v_{6}^{(3)} = \begin{pmatrix}
1 & 0 & 0 \\
\hat{r}_{2,a}(k) e^{\theta_{21}} & 1 & 0 \\
0 & 0 & 1
\end{pmatrix}, \nonumber \\
& v_{7}^{(2)} \hspace{-0.07cm} = \hspace{-0.07cm} v_{7u}^{(3)}v_{7}^{(3)}\hspace{-0.07cm}, \;\; v_{7}^{(3)} \hspace{-0.07cm}=\hspace{-0.07cm} \begin{pmatrix}
1 & \hspace{-0.07cm}0 & \hspace{-0.07cm}0 \\
-r_{2,a}(k)e^{\theta_{21}} & \hspace{-0.07cm}1 & \hspace{-0.07cm}0 \\
0 & \hspace{-0.07cm}0 & \hspace{-0.07cm}1
\end{pmatrix}\hspace{-0.07cm}, \;\; v_{9}^{(2)} \hspace{-0.07cm}=\hspace{-0.07cm} v_{9}^{(3)}v_{9d}^{(3)}, \;\; v_{9}^{(3)} \hspace{-0.07cm}=\hspace{-0.07cm} \begin{pmatrix}
1 & \hspace{-0.07cm}-r_{1,a}(k)e^{-\theta_{21}} & \hspace{-0.07cm}0 \\
0 & \hspace{-0.07cm}1 & \hspace{-0.07cm}0 \\
0 & \hspace{-0.07cm}0 & \hspace{-0.07cm}1
\end{pmatrix}\hspace{-0.07cm}, \nonumber \\
& v_{7u}^{(3)} = \begin{pmatrix}
1 & 0 & 0 \\
0 & 1 & 0 \\
(r_{1,a}(\omega^{2}k)+r_{1,a}(\frac{1}{\omega k})r_{2,a}(k))e^{\theta_{31}} & r_{1,a}(\frac{1}{\omega k})e^{\theta_{32}} & 1
\end{pmatrix}, \nonumber \\
& v_{9d}^{(3)} = \begin{pmatrix}
1 & 0 & (r_{2,a}( \omega^{2}k)+r_{1,a}(k)r_{2,a}(\frac{1}{\omega k}))e^{-\theta_{31}} \\
0 & 1 & r_{2,a}(\frac{1}{\omega k})e^{-\theta_{32}} \\
0 & 0 & 1
\end{pmatrix}. \label{Vv2def}
\end{align}
\begin{figure}
\begin{center}
\begin{tikzpicture}[master,scale=1.1]
\node at (0,0) {};
\draw[black,line width=0.65 mm] (0,0)--(30:7.5);

\draw[black,line width=0.65 mm,->-=0.51,->-=0.59,->-=0.67,->-=0.8] (0,0)--(90:7.5);
\node at (92.8:2.35*1.5) {\scriptsize $1_{r}$};
\node at (87.1:2.71*1.5) {\scriptsize $2_{r}$};
\node at (87.2:4.9) {\scriptsize $3_r$};
\node at (92:3.92*1.5) {\scriptsize $4_r$};

\draw[black,line width=0.65 mm] (0,0)--(150:7.5);
\draw[dashed,black,line width=0.15 mm] (0,0)--(60:7.5);
\draw[dashed,black,line width=0.15 mm] (0,0)--(120:7.5);

\draw[black,line width=0.65 mm] ([shift=(30:3*1.5cm)]0,0) arc (30:150:3*1.5cm);
\draw[black,arrows={-Triangle[length=0.27cm,width=0.18cm]}]
($(73:3*1.5)$) --  ++(-17:0.001);

\draw[black,line width=0.65 mm] ([shift=(30:3*1.5cm)]0,0) arc (30:150:3*1.5cm);

%\node at (75:4.35*1.5) {\scriptsize $1$};
\node at (75:3.14*1.5) {\scriptsize $2$};
%\node at (65:2.32*1.5) {\scriptsize $3$};

\node at (105:3.29*1.5) {\scriptsize $7$};
\node at (107:2.73*1.5) {\scriptsize $9$};

\draw[black,arrows={-Triangle[length=0.27cm,width=0.18cm]}]
($(90.5:3*1.5)$) --  ++(90.5-90:0.001);
\draw[black,arrows={-Triangle[length=0.27cm,width=0.18cm]}]
($(116-4:3*1.5)$) --  ++(113+90-4:0.001);
\node at (109.5:3.13*1.5) {\scriptsize $8$};
\node at (92.5:3.12*1.5) {\scriptsize $5$};

%\draw[black,line width=0.65 mm,-<-=0.08,->-=0.86] (90:3.9)--(96.9725:3*1.5)--(90:5.3);
\draw[black,line width=0.65 mm,-<-=0.04,-<-=0.28,->-=0.64,->-=0.92] (81:3.24)--($(96.9725:4.5)+(96.9725-138:1.1)$)--(96.9725:4.5)--($(96.9725:4.5)+(96.9725-45:1.1)$)--(84:6.25);
\node at (93.2:3.4*1.5) {\scriptsize $4$};
\node at (93.6:2.67*1.5) {\scriptsize $6$};
\node at (84.4:3.85*1.5) {\scriptsize $4$};
\node at (81.4:2.38*1.5) {\scriptsize $6$};

\draw[black,line width=0.65 mm,-<-=0.22,->-=0.6] (120:3.24)--($(96.9725:3*1.5)+(96.9725+135:1.1)$)--(96.9725:3*1.5)--($(96.9725:3*1.5)+(96.9725+45:1.1)$)--(120:6.25);
\draw[black,line width=0.65 mm] (120:3.24)--($(143.028:3*1.5)+(143.028-135:1.1)$)--(143.028:3*1.5)--($(143.028:3*1.5)+(143.028-45:1.1)$)--(120:6.25);

\draw[black,line width=0.65 mm] (150:3.9)--(143.028:3*1.5)--(150:5.3);

\draw[black,line width=0.65 mm,->-=0.6] (60:4.5)--(60:6.25);
\draw[black,line width=0.65 mm,->-=0.75] (60:3.24)--(60:4.5);
\draw[black,line width=0.65 mm,->-=0.6] (120:4.5)--(120:6.25);
\draw[black,line width=0.65 mm,->-=0.75] (120:3.24)--(120:4.5);

%\draw[black,line width=0.65 mm,-<-=0.70] ([shift=(30:3.24cm)]0,0) arc (30:90:3.24cm);
%\draw[black,line width=0.65 mm,-<-=0.71] ([shift=(30:6.25cm)]0,0) arc (30:90:6.25cm);

\node at (122.7:5.35) {\scriptsize $1_{s}$};
\node at (123.5:4.0) {\scriptsize $2_{s}$};
\node at (56.8:5.35) {\scriptsize $3_{s}$};
\node at (56.3:4.0) {\scriptsize $4_{s}$};

\draw[black,line width=0.65 mm] (143.028:3.24)--(143.028:6.25);
\draw[black,line width=0.65 mm,-<-=0.08,->-=0.85] (96.9725:3.24)--(96.9725:6.25);

\draw[black,line width=0.65 mm] ([shift=(30:3.24cm)]0,0) arc (30:150:3.24cm);
\draw[black,line width=0.65 mm] ([shift=(30:6.25cm)]0,0) arc (30:150:6.25cm);

\node at (99.6:5.6) {\scriptsize $5_{s}$};
\node at (100:3.65) {\scriptsize $6_{s}$};

\draw[green,fill] (143.028:3*1.5) circle (0.12cm);
\draw[blue,fill] (96.9725:3*1.5) circle (0.12cm);

\draw[black,line width=0.65 mm,->-=0.43] (90:4.5)--(70:6.25)--(69.5:7.5);
\draw[black,line width=0.65 mm,->-=0.89] (70:0)--(70:3.24)--(90:4.5);
\node at (81:4.1) {\scriptsize $1_a$};
\node at (74.5:5.3) {\scriptsize $4_a$};

\draw[black,line width=0.65 mm] (30:4.5)--(50:6.25)--(50.5:7.5);
\draw[black,line width=0.65 mm] (50:0)--(50:3.24)--(30:4.5);

%reflections of contour 4 and 6
\draw[black,line width=0.65 mm] (39:3.24)--($(30-6.9725:4.5)+(30-6.9725+138:1.1)$)--(30:4.05);
\draw[black,line width=0.65 mm] (30:5.16)--($(30-6.9725:4.5)+(30-6.9725+45:1.1)$)--(36:6.25);

\end{tikzpicture}
\end{center}
\begin{figuretext}\label{Gamma3fig}
The contour $\Gamma^{(3)}$ (solid), the boundary of $\mathsf{S}$ (dashed), and the saddle points $k_{1}$ (blue) and $\omega^{2}k_{2}$ (green).
\end{figuretext}
\end{figure}
Let $\Gamma^{(3)}$ be the contour shown in Figure \ref{Gamma3fig}. We define $n^{(3)}$ by
\begin{align}\label{n3def}
n^{(3)}(x,t,k) = n^{(2)}(x,t,k)G^{(3)}(x,t,k), \qquad k \in \C \setminus \Gamma^{(3)},
\end{align}
where $G^{(3)}$ is given for $k \in \mathsf{S}$ by
\begin{align}\label{G3def}
G^{(3)}(x,t,k) = \begin{cases} 
(v_{4}^{(2)})^{-1}, & k \mbox{ on the $+$ side of }\Gamma_{4_a}^{(3)},
	\\[0.05cm]
v_{6}^{(2)}, & k \mbox{ on the $+$ side of }\Gamma_{1_a}^{(3)},
	\\[0.05cm]
(v_{4u}^{(3)})^{-1}, & k \mbox{ on the $+$ side of }\Gamma_{4}^{(3)}, \\
v_{6d}^{(3)}, & k \mbox{ on the $-$ side of }\Gamma_{6}^{(3)}, \\
v_{7u}^{(3)}, & k \mbox{ on the $-$ side of }\Gamma_{7}^{(3)}, \\
(v_{9d}^{(3)})^{-1}, & k \mbox{ on the $+$ side of }\Gamma_{9}^{(3)}, \\
I, & \mbox{otherwise},
\end{cases}
\end{align}
and $G^{(3)}$ is defined to $\mathbb{C}\setminus \Gamma^{(3)}$ using the $\mathcal{A}$- and $\mathcal{B}$-symmetries (as in \eqref{symmetry of G1}). The following lemma can be deduced from Lemma \ref{decompositionlemma2} and Figure \ref{IIbis fig: Re Phi 21 31 and 32 for zeta=0.7}.

\begin{lemma}
$G^{(3)}(x,t,k)$ and $G^{(3)}(x,t,k)^{-1}$ are uniformly bounded for $k \in \mathbb{C}\setminus \Gamma^{(3)}$, $\tau \in \mathcal{I}$, and $x\geq 1$. Furthermore, $G^{(3)}(x,t,k)=I$ for all large enough $|k|$.
\end{lemma}
The jump matrices $v_j^{(3)}$ are given for $j = 4,6,7,9$ by (\ref{Vv2def}) and by
\begin{align}\nonumber
& v_j^{(3)} := v_j^{(2)}, \; j=2,5,8,1_{s},2_{s},3_{s},4_{s}, \qquad
v_{5_s}^{(3)} :=  v_{4u}^{(3)} v_{7u}^{(3)},\qquad v_{6_s}^{(3)} :=  v_{9d}^{(3)} v_{6d}^{(3)},
	\\\nonumber
& v_{1_r}^{(3)} = (v_{6d}^{(3)})^{-1}v_{1_r}^{(2)}v_{6d}^{(3)}, \quad
v_{2_r}^{(3)} = (v_{6}^{(2)})^{-1}v_{2_r}^{(2)}, 
\quad
v_{3_r}^{(3)} = v_{4}^{(2)} v_{3_r}^{(2)} , \quad
v_{4_r}^{(3)} = v_{4u}^{(3)} v_{4_r}^{(2)}(v_{4u}^{(3)})^{-1},
	\\ \label{II jumps vj diagonal p2p}
& v_{1_a}^{(3)} = v_{1_a}^{(2)} v_{6}^{(2)}, \qquad
v_{4_a}^{(3)} = v_{4_a}^{(2)} (v_{4}^{(2)})^{-1}.
\end{align}

On the subcontours of $\Gamma^{(3)}\cap \mathsf{S}$ that are unlabeled in Figure \ref{Gamma3fig}, the matrix $v^{(3)}$ is close to $I$ as $x \to \infty$, so we omit its explicit expression. On $\Gamma^{(3)}\setminus \mathsf{S}$, $v^{(3)}$ can be obtained using \eqref{vjsymm}. 

Our next lemma shows that the jumps on $\Gamma_{j}^{(1)}$, $j = 1_r, \dots, 4_r$, are uniformly small for large $x$.

\begin{lemma}\label{vjrsmalllemma}
For any $1 \leq p \leq \infty$ and $j = 1_r, 2_r, 3_r, 4_r$, the $L^p$-norm of $v^{(3)} - I$ on $\Gamma_j^{(3)}$ is $O(x^{-N})$ as $x \to \infty$ uniformly for $\tau \in \mathcal{I}$.
\end{lemma}
\begin{proof}
We have
\begin{align*}
& v_{1_r}^{(3)} - I = (v_{6d}^{(3)})^{-1}(v_{1_r}^{(1)} - I)v_{6d}^{(3)}, \qquad
&& v_{2_r}^{(3)} -I = (v_{6}^{(2)})^{-1}(v_{1_r}^{(1)} -I)v_{6}^{(2)}.
\end{align*}
The matrices $(v_{6d}^{(3)})^{\pm 1}$ and $(v_{6}^{(2)})^{\pm 1}$ are uniformly bounded on $\Gamma_{1_r}^{(3)}$ and $\Gamma_{2_r}^{(3)}$, respectively. Indeed, consider for example the $(13)$-entry of $v_{6d}^{(3)}$ given by $(v_{6d}^{(3)})_{13} = r_{1,a}(\frac{1}{\omega^{2} k})e^{-\theta_{31}}$. Using Lemma \ref{decompositionlemma2} and the fact that $\tilde{\Phi}_{21}(\tau,\frac{1}{\omega^{2} k}) = \tilde{\Phi}_{31}(\tau,k)$ has nonnegative real part on $k \in \Gamma_{1_r}^{(3)}$, we obtain
$$|(v_{6d}^{(3)}(x,t,k))_{13}| \leq C e^{\frac{x}{4}|\re \tilde{\Phi}_{21}(\tau,\frac{1}{\omega^{2} k})|} e^{- x \re \tilde{\Phi}_{31}(\tau, k)}
\leq C e^{-\frac{3x}{4}|\re \tilde{\Phi}_{31}(\tau, k)|}
\leq C$$
uniformly for $k \in \Gamma_{1_r}^{(3)}$ and $\tau \in \mathcal{I}$, as desired.
On the other hand, for any $1 \leq p \leq \infty$, the $L^p$-norm of the matrix $v_{1_r}^{(1)} - I$ on $[0,i]$ is of order $O(x^{-N})$ uniformly for $\tau \in \mathcal{I}$ by Lemma \ref{decompositionlemma1}(\ref{decompositionlemma1partc}). This proves the desired statement $j = 1_r, 2_r$; the proof for $j = 3_r, 4_r$ is analogous.
\end{proof}

\begin{lemma}\label{v1av4asmalllemma}
For any $1 \leq p \leq \infty$, the $L^p$-norm of $v^{(3)} - I$ on $\Gamma_{1_a}^{(3)} \cup \Gamma_{4_a}^{(3)}$ is $O(x^{-N})$ as $x \to \infty$ uniformly for $\tau \in \mathcal{I}$.
\end{lemma}
\begin{proof}
We give a proof for $v_{1_a}^{(3)}$; the proof for $v_{4_a}^{(3)}$ is analogous.
We have
\begin{align*}
& v_{1_a}^{(3)} - I 
%= (v_3^{(2)})^{-1}v_{1_a}^{(1)} v_{6}^{(2)} - I
= \begin{pmatrix}
 0 & 0 & \big(r_{1,a}(\frac{1}{k \omega ^2})+r_{2,a}(k \omega ^2)\big)e^{-\theta_{31}} \\
\big(\tilde{r}_{1,a}(\frac{1}{k}) - r_{1,a}(\frac{1}{k}) + \hat{r}_{2,a}(k) \big)e^{\theta_{21}} & 0 & (v_{1_a}^{(3)})_{23} \\
 0 & 0 & 0
 \end{pmatrix},
 	\\
& (v_{1_a}^{(3)})_{23} := \Big(\big(\tilde{r}_{1,a}(\tfrac{1}{k}) - r_{1,a}(\tfrac{1}{k})\big) r_{1,a}(\tfrac{1}{\omega^2 k}) 
-\big(r_{1,a}(\omega k) + r_{2,a}(\tfrac{1}{\omega k}) + r_{1,a}(\tfrac{1}{k}) r_{2,a}(\omega^2 k) \big) \Big)e^{-\theta_{32}}.
\end{align*}

Let us first estimate the $(21)$-entry of $v_{1_a}^{(3)}$.
Lemma \ref{decompositionlemma1} (with $N$ replaced by $2N$) implies that
$$\bigg|\tilde{r}_{1,a}(\tfrac{1}{k}) - \sum_{j=0}^{2N} \frac{\partial_k^j\big(r_{1}(\frac{1}{k})\big)\big|_{k=i}}{j!}(k - i)^{j}  \bigg| \leq C |k- i |^{2N+1} e^{\frac{x}{4}|\re \tilde{\Phi}_{21}(\tau,\frac{1}{k})|}, \qquad k \in \Gamma_{1_a}^{(3)},$$
and, similarly, Lemma \ref{decompositionlemma2} implies that
\begin{align}\label{tilder1a1overk}
\bigg|r_{1,a}(\tfrac{1}{k}) - \sum_{j=0}^{2N} \frac{\partial_k^j\big(r_{1}(\frac{1}{k})\big)\big|_{k=i}}{j!}(k - i)^{j}  \bigg| \leq C |k- i |^{2N+1} e^{\frac{x}{4}|\re \tilde{\Phi}_{21}(\tau,\frac{1}{k})|}, \qquad k \in \Gamma_{1_a}^{(3)}.
\end{align}
Hence, applying the triangle inequality and the identity $\re \tilde{\Phi}_{21}(\tau,\frac{1}{k}) = -\re \tilde{\Phi}_{21}(\tau,k)$,
\begin{align}\label{tilder1aminusr1a}
\big| \tilde{r}_{1,a}(\tfrac{1}{k}) - r_{1,a}(\tfrac{1}{k}) \big|  \leq 2 C |k- i |^{2N+1} e^{\frac{x}{4}|\re \tilde{\Phi}_{21}(\tau,k)|}.
\end{align}
On the other hand, the function $\hat{r}_2(k)$ vanishes to all orders at $k = i$ and there exists a $c > 0$ such that $\re \tilde{\Phi}_{21}(\tau, k) \leq -c |k-i|^2$ uniformly for $k \in \Gamma_{1_a}^{(3)}$ and $\tau \in \mathcal{I}$.
Consequently, using (\ref{tilder1aminusr1a}) and the estimate for $\hat{r}_{2,a}(k)$ that follows from Lemma \ref{decompositionlemma2}, we arrive at
$$|(v_{1_a}^{(3)}(x,t,k))_{21}| \leq C |k-i|^{2N+1} e^{-\frac{3}{4}x |\re \tilde{\Phi}_{21}(\tau, k)|} \leq C |k-i|^{2N+1} e^{-c x |k-i|^2} \leq C x^{-N}$$
for $k \in \Gamma_{1_a}^{(3)}$ and $\tau \in \mathcal{I}$, which gives the desired estimate of the $(21)$-entry of $v_{1_a}^{(3)} - I$.

Let us next consider the $(23)$-entry of $v_{1_a}^{(3)}$.
We have $(v_{1_a}^{(3)})_{23} = (g_a(k) - h_a(k)) e^{-x \tilde{\Phi}_{32}(\tau, k)}$, where
$$g_a(k):= \big(\tilde{r}_{1,a}(\tfrac{1}{k}) - r_{1,a}(\tfrac{1}{k})\big) r_{1,a}(\tfrac{1}{\omega^2 k}), \quad h_a(k) := r_{1,a}(\omega k) + r_{2,a}(\tfrac{1}{\omega k}) + r_{1,a}(\tfrac{1}{k}) r_{2,a}(\omega^2 k).$$
Recalling (\ref{tilder1aminusr1a}) and estimating $r_{1,a}(\tfrac{1}{\omega^2 k})$ by means of Lemma \ref{decompositionlemma2}, we get the bound
$$\big|g_a(k) e^{-x \tilde{\Phi}_{32}(\tau, k)}\big| \leq C |k- i |^{2N+1} e^{\frac{x}{4}|\re \tilde{\Phi}_{21}(\tau,k)|} e^{\frac{x}{4}|\re \tilde{\Phi}_{21}(\tau,\tfrac{1}{\omega^2 k})|} e^{-x |\tilde{\Phi}_{32}(\tau, k)|}.$$
Since $-\tilde{\Phi}_{21}(\tau, k) + \tilde{\Phi}_{21}(\tau, \tfrac{1}{\omega^2 k}) = \tilde{\Phi}_{21}(\tau, \omega k)= \tilde{\Phi}_{32}(\tau, k)$, the above can be rewritten as
\begin{align}
\big|g_a(k) e^{-x \tilde{\Phi}_{32}(\tau, k)}\big| & \leq C |k- i |^{2N+1} e^{ - x (\re \tilde{\Phi}_{32}(\tau, k) - \frac{1}{4} |\re \tilde{\Phi}_{21}(\tau, \omega k)|)} \nonumber \\
& \leq C |k- i |^{2N+1} e^{ - \frac{3x}{4} \re \tilde{\Phi}_{32}(\tau, k)}. \label{gabound}
\end{align}
On the other hand, considering the relation (\ref{r1r2 relation on the unit circle}) with $k$ replaced by $\tfrac{1}{\omega^2 k}$, we infer that the function
$$h(k) := r_{1}(\omega k) + r_{2}(\tfrac{1}{\omega k}) + r_{1}(\tfrac{1}{k}) r_{2}(\omega^2 k)$$
vanishes to all orders at $k = i$. 
Hence, since $\re \tilde{\Phi}_{32}(\tau, k) \geq c |k-i|$ for $k \in \Gamma_{1_a}^{(3)}$ (cf. Figure \ref{IIbis fig: Re Phi 21 31 and 32 for zeta=0.7}),
\begin{align}\nonumber
\big|h_a(k) e^{-x \tilde{\Phi}_{32}(\tau, k)}\big|  & \leq |h_a(k) - h(k)|e^{-x \re \tilde{\Phi}_{32}(\tau, k)} + |h(k)|e^{-x \re \tilde{\Phi}_{32}(\tau, k)}
	\\ \label{v1a323estimate}
& \leq |h_a(k) - h(k)|e^{-x \re \tilde{\Phi}_{32}(\tau, k)} + C x^{-N}.
\end{align}
To estimate $|h_a(k) - h(k)|$, we note that Lemma \ref{decompositionlemma2} yields
\begin{align*} 
& \Big| r_{1,a}(\omega k)-\sum_{j=0}^{2N}\frac{\partial_k^j\big(r_{1}(\omega k)\big)\big|_{k=i}}{j!}(k- i)^{j}  \Big| \leq C |k - i|^{2N+1} e^{\frac{x}{4}|\re \tilde{\Phi}_{21}(\tau,\omega k)|}, \qquad k \in \Gamma_{1_a}^{(3)},
	\\ 
& \Big| r_{2,a}(\omega^2 k)-\sum_{j=0}^{2N}\frac{\partial_k^j\big(r_{2}(\omega^2 k)\big)\big|_{k=i}}{j!}(k - i)^{j}  \Big| \leq C |k - i|^{2N+1} e^{\frac{x}{4}|\re \tilde{\Phi}_{21}(\tau,\omega^2 k)|} , \qquad k \in \Gamma_{1_a}^{(3)},
	\\
& \Big| r_{2,a}(\tfrac{1}{\omega k}) - \sum_{j=0}^{2N}\frac{\partial_k^j\big(r_{2}(\frac{1}{\omega k})\big)\big|_{k=i}}{j!}(k-i)^{j}  \Big| \leq C |k-i|^{2N+1} e^{\frac{x}{4}|\re \tilde{\Phi}_{21}(\tau, \frac{1}{\omega k})|}, \qquad k \in \Gamma_{1_a}^{(3)},
\end{align*}
uniformly for $\tau \in \mathcal{I}$. 
The identities 
$$\tilde{\Phi}_{21}(\tau, \omega k) = -\tilde{\Phi}_{21}(\tau, \tfrac{1}{\omega k}) = \tilde{\Phi}_{21}(\tau, \tfrac{1}{k}) - \tilde{\Phi}_{21}(\tau, \omega^2 k), \qquad k \in \C,$$ 
together with Figure \ref{IIbis fig: Re Phi 21 31 and 32 for zeta=0.7}, imply that
$$|\re \tilde{\Phi}_{21}(\tau, \omega k)| = |\re \tilde{\Phi}_{21}(\tau, \tfrac{1}{\omega k})| = |\re \tilde{\Phi}_{21}(\tau, \tfrac{1}{k})| + |\re \tilde{\Phi}_{21}(\tau, \omega^2 k)|, \qquad k \in \Gamma_{1_a}^{(3)},$$
and hence, using also (\ref{tilder1a1overk}) and the relation $\tilde{\Phi}_{21}(\tau, \omega k)= \tilde{\Phi}_{32}(\tau, k)$,
\begin{align}
|h_a(k) - h(k)|e^{-x \re \tilde{\Phi}_{32}(\tau, k)} & \leq 
C |k-i|^{2N+1} e^{ - x (|\re \tilde{\Phi}_{32}(\tau, k)| - \frac{1}{4} |\re \tilde{\Phi}_{21}(\tau, \omega k)|)} \nonumber \\
& \leq C |k-i|^{2N+1} e^{ - \frac{3x}{4} |\re \tilde{\Phi}_{32}(\tau, k)|} \label{haminush}
\end{align}
for $k \in \Gamma_{1_a}^{(3)}$ and $\tau \in \mathcal{I}$. 
Shrinking the size of the lenses if necessary, there exists a $c > 0$ such that
\begin{align}\label{tildePhi32}
 \re \tilde{\Phi}_{32}(\tau, k) \geq c|k-i|
\end{align}
uniformly for $k \in \Gamma_{1_a}^{(3)}$ and $\tau \in \mathcal{I}$. 
Combining the inequalities (\ref{gabound}), (\ref{v1a323estimate}), (\ref{haminush}), and (\ref{tildePhi32}), it transpires that
\begin{align}\nonumber
|(v_{1_a}^{(3)}(x,t,k))_{23}| \leq C |k-i|^{2N+1} e^{-c|k-i|} + C x^{-N} \leq C x^{-N},
\end{align}
which proves the desired estimate for the $(23)$-entry of $v_{1_a}^{(3)} - I$; the $(13)$-entry can be handled similarly since $r_{1}(\frac{1}{k\omega^{2}})+r_{2}(k \omega^{2})$ vanishes to all orders at $k=i$ (this follows from \eqref{r1r2 relation on the unit circle} with $k$ replaced by $\omega k$ and the assumption that $r_{1}(k)$ vanishes to all orders at $k=i$).
\end{proof}

The next lemma is proved in the same way as \cite[Lemma 6.3]{CLsectorIV}.

\begin{lemma}\label{vjslemma}
It is possible to choose the analytic approximations of Lemma \ref{decompositionlemma2} so that the $L^\infty$-norm of $v_j^{(3)} - I$ on $\Gamma_j^{(3)}$, $j = 1_s, \dots, 6_s$, is $O(x^{-N})$ as $x \to \infty$ uniformly for $\tau \in \mathcal{I}$. 
\end{lemma}

Using Lemmas \ref{vjrsmalllemma}, \ref{v1av4asmalllemma}, and \ref{vjslemma}, we conclude that $v^{(3)}-I=O(x^{-N})$ as $x \to \infty$, uniformly for $k \in \Gamma^{(3)} \setminus \big(\cup_{j=0}^{2}\cup_{\ell=1}^{2} D_{\epsilon}(\omega^{j} k_{\ell}) \cup \partial \D \big)$.

\section{Global parametrix}
For each $\tau \in \mathcal{I}$, let $\delta(\tau, \cdot): \mathbb{C}\setminus \Gamma_5^{(3)} \to \mathbb{C}$ be a function obeying the jump condition
\begin{align*}
& \delta_{+}(\tau, k) = \delta_{-}(\tau, k)(1 + r_{1}(k)r_{2}(k)), \qquad k \in \Gamma_{5}^{(3)},
\end{align*}
and the normalization condition $\delta(\tau, k) = 1 + O(k^{-1})$ as $k \to \infty$. As mentioned earlier (see the text below \eqref{def of rhat}), $1 + r_{1}(k)r_{2}(k)>1$ for all $k \in \Gamma_{5}^{(3)}$. Therefore, by taking the logarithm and using Plemelj's formula, we obtain
\begin{equation} \label{delta1def}
\delta(\tau, k) = \exp \left\{ \frac{-1}{2\pi i} \int_{i}^{k_{1}} \frac{\ln(1 + r_1(s)r_{2}(s))}{s - k} ds \right\}, \qquad k \in \C \setminus \Gamma_{5}^{(3)},
\end{equation}
where the path of integration follows the unit circle in the counterclockwise direction, and the principal branch is used for the logarithm. 

\begin{lemma}\label{IIbis deltalemma}
The function $\delta(\tau, k)$ has the following properties:
\begin{enumerate}[$(a)$]
\item\label{deltaparta}
 $\delta(\tau,k)$ admits the representation
\begin{align}
& \delta(\tau,k) = e^{-i \nu \ln_{k_{1}}(k-k_{1})}e^{-\chi(\tau,k)}, \label{delta expression in terms of log and chi} \\
& \chi(\tau,k) = \frac{-1}{2\pi i} \int_{i}^{k_{1}}  \ln_{s}(k-s) d\ln(1+r_1(s)r_{2}(s)), \label{def of chi}
\end{align}
where the path of integration follows $\partial \mathbb{D}$ in the counterclockwise direction and $\nu \leq 0$ is defined in \eqref{def of nu beta}. 

For $s \in \bar{\Gamma}_{5}^{(3)}$, $k \mapsto \ln_{s}(k-s)=\ln |k-s|+i \arg_{s}(k-s)$ has a cut along $\{e^{i \theta}: \theta \in [\frac{\pi}{2},\arg s], \; \arg s \in [\frac{\pi}{2},\frac{2\pi}{3}]\}\cup(i,i\infty)$ and $\arg_{s}(1)=2\pi$.

\item\label{deltapartb} 
For each $\tau \in \mathcal{I}$, $\delta(\tau, k)$ and $\delta(\tau, k)^{-1}$ are analytic functions of $k \in \mathbb{C}\setminus \Gamma_{5}^{(3)}$. Furthermore,
\begin{align*}
& \sup_{\tau \in \mathcal{I}} \sup_{\substack{k \in \C \setminus \Gamma_{5}^{(3)}}} |\delta(\tau,k)^{\pm 1}| < \infty.
\end{align*}

\item\label{deltapartc} 
For $k \in \partial \mathbb{D}\setminus \Gamma_{5}^{(3)}$, 
\begin{align}\label{re chi on part of the unit circle}
\re \chi(\tau,k) = \frac{\pi + \arg_{i} k}{2}\nu + \frac{1}{2\pi} \int_{\Gamma_{5}^{(3)}} \frac{\arg s}{2} d\ln(1+r_1(s)r_{2}(s)),
\end{align}
where $\arg_i k \in (\frac{\pi}{2},\frac{5\pi}{2})$ and $\arg s \in (-\pi,\pi)$. In particular, for $k \in \partial \mathbb{D}\setminus \Gamma_{5}^{(3)}$, $|\delta(\tau,k)|$ is constant and given by
\begin{align}\label{|delta| is constant on a part of the unit circle}
|\delta(\tau,k)| = \exp \bigg( \nu \frac{\arg k_{1}}{2} - \frac{1}{2\pi} \int_{\Gamma_{5}^{(3)}} \frac{\arg s}{2} d\ln(1+r_1(s)r_{2}(s)) \bigg),
\end{align}
where $\arg k_{1} \in [\frac{\pi}{2},\frac{2\pi}{3}]$.

\item\label{deltapartd} 
As $k \to k_{1}$ along a path which is nontangential to $\Gamma_{5}^{(3)}$, we have
\begin{align}
& |\chi(\tau,k)-\chi(\tau,k_{1})| \leq C |k-k_{1}| (1+|\ln|k- k_{1}||), \label{II asymp chi at k1}
\end{align}
where $C$ is independent of $\tau \in \mathcal{I}$.  

\item\label{deltaparte}
 As $k \to \infty$,
\begin{align} \label{deltaasymptotics}
\delta(\tau, k) = e^{\frac{1}{2\pi i k} \int_{i}^{k_1} \ln(1 + r_1(s)r_{2}(s)) ds + O(k^{-2})}.
\end{align}

\end{enumerate}
\end{lemma}
\begin{proof}
To prove part $(\ref{deltaparta})$, it suffices to perform an integration by parts in \eqref{delta1def}. Part $(\ref{deltapartb})$ is then a consequence of \eqref{delta expression in terms of log and chi} and the fact that $\im \nu = 0$. By \eqref{def of chi} we have, for $k \in \partial \mathbb{D}\setminus \Gamma_{5}^{(3)}$,
\begin{align*}
\re \chi(\tau,k) & = \frac{1}{2\pi} \int_{\Gamma_{5}^{(3)}}  \arg_{s}(k-s) d\ln(1+r_1(s)r_{2}(s)) \\
& = \frac{1}{2\pi} \int_{\Gamma_{5}^{(3)}} \frac{\pi + \arg_{i} k + \arg s}{2} d\ln(1+r_1(s)r_{2}(s)),
\end{align*}
and \eqref{re chi on part of the unit circle} then follows from the assumption that $r_1 = 0$ on $[0, i]$ and the definition \eqref{def of nu beta} of $\nu$. By \eqref{delta expression in terms of log and chi}, 
\begin{align*}
|\delta(\tau,k)| = e^{\nu \arg_{k_{1}}(k-k_{1})}e^{-\re \chi(\tau,k)}.
\end{align*}
We obtain \eqref{|delta| is constant on a part of the unit circle} after substituting \eqref{re chi on part of the unit circle} and $\arg_{k_{1}}(k-k_{1}) = \frac{\pi + \arg_{i} k + \arg k_{1}}{2}$ into the above equation. This finishes the proof of part $(\ref{deltapartc})$. We obtain part $(\ref{deltapartd})$ using \eqref{def of chi} and standard estimates and part $(\ref{deltaparte})$ is immediate from \eqref{delta1def}.
\end{proof}
For $\tau \in \mathcal{I}$ and $k \in \mathbb{C}\setminus \partial \mathbb{D}$, we let $\{\Delta_{jj}(\tau,k)\}_{j=1}^{3}$ be given by
\begin{align*}
\Delta_{33}(\tau,k) = \frac{\delta(\tau,\omega k)}{\delta(\tau,\omega^{2} k)}\frac{\delta(\tau,\frac{1}{\omega^{2} k})}{\delta(\tau,\frac{1}{\omega k})}, \quad \Delta_{11}(\tau,k)=\Delta_{33}(\tau,\omega k), \quad \Delta_{22}(\tau,k)=\Delta_{33}(\tau,\omega^{2} k).
\end{align*}
We have $\Delta_{33}(\tau,\frac{1}{k}) = \Delta_{33}(\tau,k)$ and, using (\ref{deltaasymptotics}),
\begin{align}\nonumber
\Delta_{33}(\tau,k) & = e^{\frac{1}{2\pi i}(\frac{1}{\omega k} - \frac{1}{\omega^2 k} )\int_{i}^{k_1} \ln(1 + r_1(s)r_{2}(s)) (1 + s^{-2}) ds + O(k^{-2})}
	\\ \label{Delta33asymptotics}
& = 1  - \frac{\sqrt{3}}{2\pi k}\int_{i}^{k_1} \ln(1 + r_1(s)r_{2}(s)) (1 + s^{-2}) ds + O(k^{-2}), \qquad k \to \infty.
\end{align}
We define the inverse of the global parametrix by
\begin{align}\label{Deltadef}
\Delta(\tau,k) = \begin{pmatrix}
\Delta_{11}(\tau,k) & 0 & 0 \\
0 & \Delta_{22}(\tau,k) & 0 \\
0 & 0 & \Delta_{33}(\tau,k)
\end{pmatrix}, \quad \tau \in \mathcal{I}, \; k \in \mathbb{C}\setminus \partial \mathbb{D}.
\end{align}
It satisfies the symmetries
\begin{align*}
\Delta(\tau,k) = \mathcal{A}\Delta(\tau,\omega k)\mathcal{A}^{-1} = \mathcal{B}\Delta(\tau,\tfrac{1}{k})\mathcal{B},
\end{align*}
and obeys
\begin{subequations}\label{IIbis jumps Delta11}
\begin{align}
& \Delta_{+}^{-1}(\tau,k) = \Delta_{-}^{-1}(\tau,k) \begin{pmatrix}
1+r_{1}(k)r_{2}(k) & 0 & 0 \\
0 & \frac{1}{1+r_{1}(k)r_{2}(k)} & 0 \\
0 & 0 & 1
\end{pmatrix}, & & k \in \Gamma_{5}^{(3)}, \\
& \Delta_{+}^{-1}(\tau,k) = \Delta_{-}^{-1}(\tau,k), & & k \in \Gamma_{2}^{(3)}\cup \Gamma_{8}^{(3)}.
\end{align}
\end{subequations}
Using \eqref{Deltadef} and Lemma \ref{IIbis deltalemma}, we deduce the following.

\begin{lemma}\label{Deltalemma}
$\Delta(\tau,k)$ and $\Delta(\tau,k)^{-1}$ are uniformly bounded for $k \in \mathbb{C}\setminus \partial \mathbb{D}$ and $\tau \in \mathcal{I}$. Furthermore, $\Delta(\tau,k)=I+O(k^{-1})$ as $k \to \infty$.
\end{lemma} 

\section{The $n^{(3)}\to n^{(4)}$ transformation}
The purpose of the fourth transformation $n^{(3)}\to n^{(4)}$ is to make $v^{(4)}-I$ uniformly small as $x \to \infty$ for $|k|=1$. Let $n^{(4)}$ be given by
\begin{align}\label{IIbis def of mp3p}
n^{(4)}(x,t,k) = n^{(3)}(x,t,k)\Delta(\tau,k), \qquad  k \in \mathbb{C}\setminus \Gamma^{(4)},
\end{align}
where $\Gamma^{(4)}=\Gamma^{(3)}$. This function satisfies $n_{+}^{(4)}=n_{-}^{(4)}v^{(4)}$ on $\Gamma^{(4)}$, where $v^{(4)}=\Delta_{-}^{-1}v^{(3)}\Delta_{+}$. We define a small cross centered at $k_{1}$ by $\mathcal{X}^{\epsilon}:= \cup_{j=1}^{4}\mathcal{X}_{j}^{\epsilon}$, where
\begin{align*}
&\mathcal{X}_{1}^{\epsilon} = \Gamma_{4}^{(4)}\cap D_{\epsilon}(k_{1}), & & \mathcal{X}_{2}^{\epsilon} = \Gamma_{7}^{(4)}\cap D_{\epsilon}(k_{1}), & & \mathcal{X}_{3}^{\epsilon} = \Gamma_{9}^{(4)}\cap D_{\epsilon}(k_{1}), & & \mathcal{X}_{4}^{\epsilon} = \Gamma_{6}^{(4)}\cap D_{\epsilon}(k_{1}),
\end{align*}
are oriented outwards from $k_1$. We also define the union of six small crosses centered at $\{\omega^{j}k_{1},\omega^{j}k_{2}\}_{j=0}^{2}$ by $\hat{\mathcal{X}}^\epsilon = \mathcal{X}^\epsilon \cup \omega \mathcal{X}^\epsilon \cup \omega^2 \mathcal{X}^\epsilon \cup (\mathcal{X}^\epsilon)^{-1} \cup (\omega \mathcal{X}^\epsilon)^{-1} \cup (\omega^2 \mathcal{X}^\epsilon)^{-1}$.

\begin{lemma}\label{II v3lemma}
$v^{(4)}$ converges to the identity matrix $I$ as $x \to \infty$ uniformly for $\tau \in \mathcal{I}$ and $k \in \Gamma^{(4)}\setminus \hat{\mathcal{X}}^\epsilon$. More precisely, for $\tau \in \mathcal{I}$,
\begin{align}\label{II v3estimatesa}
& \|v^{(4)} - I\|_{(L^1 \cap L^\infty)(\Gamma^{(4)}\setminus \hat{\mathcal{X}}^\epsilon)} \leq Cx^{-N}.
\end{align}
\end{lemma}
\begin{proof}
Using \eqref{vp1p 123}, \eqref{vp1p 789}, \eqref{vp1p 101112}, and \eqref{II jumps vj diagonal p2p} together with the jumps \eqref{IIbis jumps Delta11} of $\Delta$, we obtain
\begin{align*}
& v_j^{(4)} = \Delta_{-}^{-1}v_j^{(3)}\Delta_{+} = I + \Delta_{-}^{-1}v_{j,r}^{(3)}\Delta_{+}, \quad j = 2,5,8.
\end{align*}
Using Lemmas \ref{decompositionlemma2} and \ref{Deltalemma}, we infer that the matrices $\Delta_{-}^{-1}v_{j,r}^{(3)}\Delta_{+}$, $j = 2,5,8$, are $O(x^{-N})$ as $x \to \infty$.
Using also Lemmas \ref{vjrsmalllemma}, \ref{v1av4asmalllemma}, and \ref{vjslemma}, we conclude that the $L^1$ and $L^\infty$ norms of $v^{(4)}-I$ are $O(x^{-N})$ as $x \to \infty$, uniformly for $\tau \in \mathcal{I}$ and $k \in \mathsf{S}\cap \Gamma^{(4)}\setminus \hat{\mathcal{X}}^\epsilon$. Then \eqref{II v3estimatesa} directly follows from \eqref{vjsymm}.
\end{proof}

The matrix $v^{(4)}-I$ is not uniformly small on $\hat{\mathcal{X}}^\epsilon$. The goal of the next section is to build a local parametrix that approximates $n^{(4)}$ in a neighborhood of $k_{1}$. The local parametrices near $\{k_{2},\omega k_{1},\omega k_{2},\omega^{2}k_{1},\omega^{2}k_{2}\}$ will then be obtained using the $\mathcal{A}$- and $\mathcal{B}$-symmetries. %In Section \ref{section:small norm}, we prove that these parametrices approximate $n^{(4)}$ in $\cup_{j=0}^{2}\big(D_{\epsilon}(\omega^{j} k_{1})\cup D_{\epsilon}(\omega^{j} k_{2}) \big)$.

\section{Local parametrix near $k_{1}$}\label{section:loc param 1}
As $k \to k_{1}$, we have
\begin{align*}
& \tilde{\Phi}_{21}(\tau,k)-\tilde{\Phi}_{21}(\tau,k_{1}) =  \tilde{\Phi}_{21,k_{1}}(k-k_{1})^{2} +O((k-k_{1})^{3}), & & \tilde{\Phi}_{21,k_{1}}= \frac{4\tau -3k_{1}  - k_{1}^{3}}{4k_{1}^{4}}.
\end{align*}
Let $z=z(x,\tau,k)$ be given by
\begin{align*}
z=z_{\star}\sqrt{x} (k-k_{1}) \hat{z}, \qquad \hat{z}= \sqrt{\frac{2i(\tilde{\Phi}_{21}(\tau,k)-\tilde{\Phi}_{21}(\tau,k_{1}))}{z_{\star}^{2}(k-k_{1})^{2}}},
\end{align*}
where the principal branch is taken for $\hat{z}=\hat{z}(\tau,k)$, and 
\begin{align*}
z_{\star} = \sqrt{2} e^{\frac{\pi i}{4}} \sqrt{\tilde{\Phi}_{21,k_{1}}}, \qquad -ik_{1}z_{\star}>0.
\end{align*}
We have $\hat{z}(\tau,k_{1})=1$, and $x(\tilde{\Phi}_{21}(\tau,k)-\tilde{\Phi}_{21}(\tau,k_{1})) = -\frac{iz^{2}}{2}$. We fix $\epsilon>0$ and choose it so small that the function $z$ is conformal from $D_{\epsilon}(k_{1})$ to a neighborhood of 0. As $k \to k_{1}$,
\begin{align*}
z = z_{\star}\sqrt{x}(k-k_{1})(1+O(k-k_{1})),
\end{align*}
and for all $k \in D_\epsilon(k_1)$, we have
\begin{align*}
\ln_{k_{1}}(k-k_{1}) = \ln_{0}[z_{\star}(k-k_{1})\hat{z}]- \ln \hat{z} -\ln z_{\star}
\end{align*}
where $\ln_{0}(k):= \ln|k|+i\arg_{0}k$, $\arg_{0}(k)\in (0,2\pi)$, and $\ln$ is the principal logarithm. Reducing $\epsilon > 0$ if needed, $k \mapsto \ln \hat{z}$ is analytic in $D_{\epsilon}(k_{1})$, $\ln \hat{z} = O(k-k_{1})$ as $k \to k_{1}$, and
\begin{align*}
\ln z_{\star} = \ln |z_{\star}| + i \arg z_{\star} = \ln |z_{\star}| + i \big( \tfrac{\pi}{2}-\arg k_{1} \big),
\end{align*}
where $\arg k_{1} \in (\tfrac{\pi}{2},\tfrac{2\pi}{3})$. By \eqref{delta expression in terms of log and chi}, we have
\begin{align*}
& \delta(\tau,k) = e^{-i\nu\ln_{k_{1}}(k-k_{1})}e^{-\chi(\tau,k)} = e^{-i\nu \ln_{0} z} x^{\frac{i \nu}{2}} e^{i \nu \ln \hat{z}} e^{i \nu \ln z_{\star}} e^{-\chi(\tau,k)}.
\end{align*}
From now on, for brevity we will use the notation $z_{(0)}^{- i \nu}:=e^{-i\nu \ln_{0} z}$, $\hat{z}^{i\nu}:=e^{i \nu \ln \hat{z}}$, and $z_{\star}^{i \nu}:=e^{i \nu \ln z_{\star}}$, so that
\begin{align*}
\delta(\tau,k) = z_{(0)}^{- i \nu} x^{\frac{i \nu}{2}} \hat{z}^{i\nu} z_{\star}^{i \nu} e^{-\chi(\tau,k)}.
\end{align*}
Using this, we get
\begin{align}
& \frac{\Delta_{11}(\tau,k)}{\Delta_{22}(\tau,k)} = \frac{\delta(\frac{1}{k})^{2} \delta(\omega k)\delta(\omega^{2}k)}{\delta(k)^{2}\delta(\frac{1}{\omega^{2}k})\delta(\frac{1}{\omega k})} = z_{(0)}^{2i \nu}d_{0}(x,\tau)d_{1}(\tau,k), \label{IIbis Delta33 Delta11} 
	\\ \label{d0def}
& d_{0}(x,\tau) = e^{2\chi(\tau,k_{1})}x^{-i \nu} z_{\star}^{-2i\nu} \frac{\delta(\frac{1}{k_{1}})^{2} \delta(\omega k_{1})\delta(\omega^{2}k_{1})}{\delta(\frac{1}{\omega^{2}k_{1}})\delta(\frac{1}{\omega k_{1}})}, 
	 \\
& d_{1}(\tau,k) = e^{2(\chi(\tau, k)-\chi(\tau, k_{1}))} \hat{z}^{-2i\nu} \frac{\delta(\frac{1}{k})^{2} \delta(\omega k)\delta(\omega^{2}k)}{\delta(\frac{1}{\omega^{2}k})\delta(\frac{1}{\omega k})} \bigg( \frac{\delta(\frac{1}{k_{1}})^{2} \delta(\omega k_{1})\delta(\omega^{2}k_{1})}{\delta(\frac{1}{\omega^{2}k_{1}})\delta(\frac{1}{\omega k_{1}})} \bigg)^{-1}. \nonumber
\end{align}
Define
\begin{align*}
Y(x, \tau) = d_{0}(x, \tau)^{-\frac{\tilde{\sigma}_{3}}{2}}\tilde{r}(k_{1})^{-\frac{1}{4}\tilde{\sigma}_{3}}e^{-\frac{x}{2}\tilde{\Phi}_{21}(\tau,k_{1})\tilde{\sigma}_{3}},
\end{align*}
where $\tilde{\sigma}_{3} = \diag (1,-1,0)$ and  $\tilde{r}(k_{1})^{\pm \frac{1}{4}}>0$. For $k \in D_\epsilon(k_1)$, we define
\begin{align}\label{ntildedef}
\tilde{n}(x,t,k) = n^{(4)}(x,t,k)Y(x, \tau).
\end{align}
Let $\tilde{v}$ be the jump matrix of $\tilde{n}$, and let $\tilde{v}_{j}$ be the restriction of $\tilde{v}$ to $\Gamma_{j}^{(4)}\cap D_{\epsilon}(k_{1})$. By \eqref{Vv2def}, \eqref{IIbis def of mp3p}, and \eqref{IIbis Delta33 Delta11}, we have
\begin{align*}
& \tilde{v}_{4} = \begin{pmatrix}
1 & \tilde{r}(k_{1})^{\frac{1}{2}}\hat{r}_{1,a}(k) d_{1}^{-1}z_{(0)}^{-2i\nu}e^{\frac{iz^{2}}{2}} & 0 \\
0 & 1 & 0 \\
0 & 0 & 1
\end{pmatrix}, & & \tilde{v}_{7} = \begin{pmatrix}
1 & 0 & 0 \\
-\tilde{r}(k_{1})^{-\frac{1}{2}}r_{2,a}(k) d_{1}z_{(0)}^{2i\nu} e^{-\frac{iz^{2}}{2}} & 1 & 0 \\
0 & 0 & 1
\end{pmatrix}, \\
& \tilde{v}_{9} = \begin{pmatrix}
1 & -\tilde{r}(k_{1})^{\frac{1}{2}}r_{1,a}(k)d_{1}^{-1}z_{(0)}^{-2i\nu}e^{\frac{iz^{2}}{2}} & 0 \\
0 & 1 & 0 \\
0 & 0 & 1 
\end{pmatrix}, & & \tilde{v}_{6} = \begin{pmatrix}
1 & 0 & 0 \\
\tilde{r}(k_{1})^{-\frac{1}{2}}\hat{r}_{2,a}(k)d_{1}z_{(0)}^{2i\nu}e^{-\frac{iz^{2}}{2}} & 1 & 0 \\
0 & 0 & 1
\end{pmatrix}.
\end{align*}
Define $q:=-\tilde{r}(k_{1})^{-\frac{1}{2}}r_{2}(k_{1})$, and note that $-\tilde{r}(k_{1})^{\frac{1}{2}}r_{1}(k_{1})=\bar{q}$, $1+r_{1}(k_{1})r_{2}(k_{1}) = 1+|q|^{2}$ and $\nu = -\frac{1}{2\pi}\ln (1+|q|^{2})$. We also have $d_{1}(\tau,k_{1}) = 1$. Hence, we expect that $\tilde{n}$ should be close to $(1,1,1)Y\tilde{m}^{k_{1}}$, where $\tilde{m}^{k_{1}}(x,t,k)$ is analytic for $k \in D_{\epsilon}(k_{1}) \setminus \mathcal{X}^{\epsilon}$, satisfies $\tilde{m}^{k_{1}}_+ = \tilde{m}^{k_{1}}_- \tilde{v}_{\mathcal{X}^{\epsilon}}^{X}$ on $\mathcal{X}^{\epsilon}$ with
\begin{align*}
& \tilde{v}_{\mathcal{X}_{1}^{\epsilon}}^{X} = \begin{pmatrix}
1 & \frac{\tilde{r}(k_{1})^{\frac{1}{2}}r_{1}(k_{1})}{1+r_{1}(k_{1})r_{2}(k_{1})}z_{(0)}^{-2i\nu}e^{\frac{iz^{2}}{2}} & 0 \\
0 & 1 & 0 \\
0 & 0 & 1
\end{pmatrix} = \begin{pmatrix}
1 & \frac{-\bar{q}}{1+|q|^{2}}z_{(0)}^{-2i\nu}e^{\frac{iz^{2}}{2}} & 0 \\
0 & 1 & 0 \\
0 & 0 & 1
\end{pmatrix}, \\
& \tilde{v}_{\mathcal{X}_{2}^{\epsilon}}^{X} = \begin{pmatrix}
1 & 0 & 0 \\
-\tilde{r}(k_{1})^{-\frac{1}{2}}r_{2}(k_{1}) z_{(0)}^{2i\nu} e^{-\frac{iz^{2}}{2}} & 1 & 0 \\
0 & 0 & 1
\end{pmatrix} = \begin{pmatrix}
1 & 0 & 0 \\
q z_{(0)}^{2i\nu} e^{-\frac{iz^{2}}{2}} & 1 & 0 \\
0 & 0 & 1
\end{pmatrix}, \\
& \tilde{v}_{\mathcal{X}_{3}^{\epsilon}}^{X} = \begin{pmatrix}
1 & -\tilde{r}(k_{1})^{\frac{1}{2}}r_{1}(k_{1}) z_{(0)}^{-2i\nu}e^{\frac{iz^{2}}{2}} & 0 \\
0 & 1 & 0 \\
0 & 0 & 1 
\end{pmatrix} = \begin{pmatrix}
1 & \bar{q} z_{(0)}^{-2i\nu}e^{\frac{iz^{2}}{2}} & 0 \\
0 & 1 & 0 \\
0 & 0 & 1 
\end{pmatrix}, \\
& \tilde{v}_{\mathcal{X}_{4}^{\epsilon}}^{X} = \begin{pmatrix}
1 & 0 & 0 \\
\frac{\tilde{r}(k_{1})^{-\frac{1}{2}}r_{2}(k_{1})}{1+r_{1}(k_{1})r_{2}(k_{1})} z_{(0)}^{2i\nu} e^{-\frac{iz^{2}}{2}} & 1 & 0 \\
0 & 0 & 1
\end{pmatrix} = \begin{pmatrix}
1 & 0 & 0 \\
\frac{-q}{1+|q|^{2}} z_{(0)}^{2i\nu} e^{-\frac{iz^{2}}{2}} & 1 & 0 \\
0 & 0 & 1
\end{pmatrix},
\end{align*}
and is such that $\tilde{m}^{k_{1}}(x,t,k) \to I$ as $t\to\infty$ uniformly for $k\in \partial D_{\epsilon}(k_{1})$.
Let $m^{k_{1}}$ be given by
\begin{align}\label{mk1def}
m^{k_{1}}(x,t,k) = Y(x, \tau)m^{X}(q,z(x, \tau,k))Y(x, \tau)^{-1}, \qquad k \in D_\epsilon(k_1),
\end{align}
where $m^{X}$ is the unique solution to the model RH problem of Lemma \ref{IIbis Xlemma 3}.

\begin{lemma}\label{lemma: bound on Y}
The function $Y(x, \tau)$ is uniformly bounded:
\begin{align}\label{Ybound}
\sup_{\tau \in \mathcal{I}} \sup_{x \geq 2} | Y(x, \tau)^{\pm 1}| \leq C.
\end{align}
Furthermore, $d_0(x, \tau)$ and $d_1(\tau, k)$ satisfy
\begin{align}
& |d_0(x, \tau)| = e^{2\pi \nu}, & & x \geq 2, \  \tau \in \mathcal{I}, \label{d0estimate} \\
& |d_1(\tau, k) - 1| \leq C |k - k_1| (1+ |\ln|k-k_1||), & & \tau \in \mathcal{I}, \ k \in \mathcal{X}^{\epsilon}.\label{d1estimate}
\end{align}
\end{lemma}
\begin{proof}
The inequalities \eqref{Ybound} and \eqref{d1estimate} are direct consequences of Lemma \ref{IIbis deltalemma}. We now turn to the proof of \eqref{d0estimate}. Recall that $\nu \in \mathbb{R}$. Hence, by \eqref{d0def},
\begin{align}\label{lol3}
|d_{0}(x, \tau)| = e^{2 \, \re\chi(\tau,k_{1})} e^{2\nu \arg z_{\star}} \left|\frac{\delta(\frac{1}{k_{1}})^{2} \delta(\omega k_{1})\delta(\omega^{2}k_{1})}{\delta(\frac{1}{\omega^{2}k_{1}})\delta(\frac{1}{\omega k_{1}})} \right|.
\end{align}
Recall also that $\arg z_{\star} = \tfrac{\pi}{2}-\arg k_{1}$. Using \eqref{|delta| is constant on a part of the unit circle}, we get
\begin{align}\label{lol1}
\left|\frac{\delta(\frac{1}{k_{1}})^{2} \delta(\omega k_{1})\delta(\omega^{2}k_{1})}{\delta(\frac{1}{\omega^{2}k_{1}})\delta(\frac{1}{\omega k_{1}})} \right| = \exp \bigg( \nu \arg k_{1} - \frac{1}{2\pi} \int_{\Gamma_{5}^{(3)}} \arg s \; d\ln(1+r_1(s)r_{2}(s)) \bigg).
\end{align}
We find \eqref{d0estimate} after substituting \eqref{re chi on part of the unit circle} (with $k$ replaced by $k_1$) and \eqref{lol1} into \eqref{lol3}. 
\end{proof}

In the statement of the next lemma, and in what follows, $C_N(\tau)$ denotes a generic nonnegative smooth function of $\tau$ that vanishes to all orders at $\tau = 0$.

\begin{lemma}\label{k0lemma}
For each $x \geq 2$ and $\tau \in \mathcal{I}$, $m^{k_1}(x,t,k)$ defined in (\ref{mk1def}) is analytic for $k \in D_\epsilon(k_1) \setminus \mathcal{X}^\epsilon$. Moreover, $m^{k_1}(x,t,k)$ is uniformly bounded for $x \geq 2$, $\tau \in \mathcal{I}$ and $k \in D_\epsilon(k_1) \setminus \mathcal{X}^\epsilon$.
On $\mathcal{X}^\epsilon$, $m^{k_1}$ satisfies $m_+^{k_1} =  m_-^{k_1} v^{k_1}$, where the jump matrix $v^{k_1}$ obeys
\begin{align}\label{v4vk1estimate}
\begin{cases}
 \| v^{(4)} - v^{k_1} \|_{L^1(\mathcal{X}^\epsilon)} \leq C_N(\tau) x^{-1} \ln x,
	\\
\| v^{(4)} - v^{k_1} \|_{L^\infty(\mathcal{X}^\epsilon)} \leq C_N(\tau) x^{-1/2} \ln x,
\end{cases} \qquad x \geq 2, \  \tau \in \mathcal{I}.
\end{align}
Furthermore, as $x \to \infty$,
\begin{align}\label{mk1estimate1}
& \| m^{k_1}(x,t,\cdot) - I \|_{L^\infty(\partial D_\epsilon(k_1))} \leq C_N(\tau) x^{-1/2},
	\\ \label{mk1estimate2}
& m^{k_1}(x,t,k) - I = \frac{Y(x, \tau)m_1^X(q) Y(x, \tau)^{-1}}{z_{\star}\sqrt{x} (k-k_{1}) \hat{z}(\tau,k)} + O(C_N(\tau) x^{-1})
\end{align}
uniformly for $\tau \in \mathcal{I}$, where $m_1^X(q)$ is given by \eqref{mXasymptotics} with $q=-\tilde{r}(k_{1})^{-\frac{1}{2}}r_{2}(k_{1})$.
\end{lemma}
\begin{proof}
By \eqref{ntildedef} and \eqref{mk1def}, we have
$v^{(4)} - v^{k_1} = Y (\tilde{v} - v^X)Y^{-1}$.
Thus, in view of Lemma \ref{lemma: bound on Y}, the estimates in (\ref{v4vk1estimate}) will follow if we can show that
\begin{subequations}\label{II v3k0}
\begin{align}\label{II v3k0L1}
  &\| \tilde{v}(x,t,\cdot) - v^{X}(x,t,z(x, \tau,\cdot)) \|_{L^1(\mathcal{X}_{j}^\epsilon)} \leq C_N(\tau) x^{-1} \ln x,
	\\ \label{II v3k0Linfty}
&  \| \tilde{v}(x,t,\cdot) - v^{X}(x,t,z(x, \tau,\cdot)) \|_{L^\infty(\mathcal{X}_{j}^\epsilon)} \leq C_N(\tau) x^{-1/2} \ln x,
\end{align}
\end{subequations}
for $j = 1, \dots, 4$. We show (\ref{II v3k0}) for $j = 4$; similar arguments apply when $j = 1,2,3$.

For $k \in \mathcal{X}_{4}^\epsilon$, only the $(21)$-entry of the matrix $\tilde{v} - v^{X}$ is nonzero. 
Note from Lemma \ref{decompositionlemma2} that $|\hat{r}_{2,a}(k)| \leq C_N(\tau)e^{\frac{x}{4}|\re \tilde{\Phi}_{21}(\tau, k)|}$ for all $k \in \mathcal{X}_{4}^\epsilon$ and $\tau \in \mathcal{I}$.
Using also the identities 
$$\frac{-q}{1+|q|^2} = \tilde{r}(k_{1})^{-\frac{1}{2}} \hat{r}_{2}(k_{1}), \qquad
\big| e^{-\frac{iz_{1}^2}{2}} \big| = e^{-\re (\frac{iz_{1}^2}{2})} = e^{-\frac{|z_{1}|^2}{2}},
$$
and Lemma \ref{decompositionlemma2}, we find, for $k \in \mathcal{X}_{4}^\epsilon$,
\begin{align*}
|(\tilde{v} - v^{X})_{21}| = &\; \bigg|\tilde{r}(k_{1})^{-\frac{1}{2}}\hat{r}_{2,a}(k)d_{1}z_{(0)}^{2i\nu}e^{-\frac{iz^{2}}{2}}
- \frac{-q}{1 + |q|^2}z_{(0)}^{2i\nu} e^{-\frac{iz^2}{2}}\bigg|
	\\
=& \;  \big| z_{(0)}^{2i \nu} \big| \Big|  (d_1 - 1) \tilde{r}(k_{1})^{-\frac{1}{2}}\hat{r}_{2,a}(k) +\tilde{r}(k_{1})^{-\frac{1}{2}}\big(\hat{r}_{2,a}(k)
- \hat{r}_{2}(k_1)\big)\Big| \big| e^{-\frac{iz^2}{2}} \big|
	\\
\leq &\; C \bigg(|d_1 - 1| |\hat{r}_{2,a}(k)| + |\hat{r}_{2,a}(k)
- \hat{r}_{2}(k_1) | \bigg)e^{-\frac{|z|^2}{2}}
	\\
\leq &\; C_N(\tau) \bigg(|d_1 - 1| + |k-k_1| \bigg) e^{\frac{x}{4}|\re \tilde{\Phi}_{21}(\tau, k)|} e^{-\frac{|z|^2}{2}}
	\\
\leq &\; C_N(\tau) \bigg(|d_1 - 1| + |k-k_1| \bigg)e^{-c x |k-k_1|^2}.
\end{align*}
Using (\ref{d1estimate}), this gives
\begin{align*}
|(\tilde{v} - v^{X})_{21}|  \leq C_N(\tau) |k- k_1|(1+ |\ln|k- k_1||)e^{-c x |k- k_1|^2}, \qquad k \in \mathcal{X}_{4}^\epsilon.
\end{align*}
Hence
\begin{align*}
& \|\tilde{v} - v^{X}\|_{L^1(\mathcal{X}_{4}^\epsilon)}
\leq C_N(\tau) \int_0^\infty u(1+ |\ln u|) e^{-c x u^2} du \leq C_N(\tau) x^{-1}\ln x,
	\\
& \|\tilde{v} - v^{X}\|_{L^{\infty}(\mathcal{X}_{4}^\epsilon)}
\leq C_N(\tau) \sup_{u \geq 0} u(1+ |\ln u|) e^{-c x u^2} \leq C_N(\tau)x^{-1/2}\ln x,
\end{align*}
which completes the proof of (\ref{II v3k0}), and hence also of (\ref{v4vk1estimate}).

For $k \in \partial D_\epsilon(k_1)$, we have $|z| = |z_{\star}| \sqrt{x} |k- k_{1}| \, | \hat{z}| \geq c \sqrt{x}$, so the variable $z$ tends to infinity as $x \to \infty$.
Thus (\ref{mXasymptotics}) can be used to infer that
\begin{align*}
& m^{X}(q, z(x, \tau, k)) = I + \frac{m_1^{X}(q)}{z_{\star}\sqrt{x} (k- k_1) \hat{z}(\tau,k)} + O(q x^{-1}),
\qquad  x \to \infty,
\end{align*}
uniformly for $k \in \partial D_\epsilon(k_1)$ and $\tau \in \mathcal{I}$. 
Recalling the definition (\ref{mk1def}) of $m^{k_1}$, this gives
\begin{align}\label{mk1minusI}
 & m^{k_1} - I = \frac{Y(x, \tau)m_1^{X}(q) Y(x, \tau)^{-1}}{z_{\star}\sqrt{x} (k- k_1) \hat{z}(\tau,k)} + O(q x^{-1}), \qquad  x \to \infty,
\end{align}
uniformly for $k \in \partial D_\epsilon(\omega k_4)$ and $\tau \in \mathcal{I}$. 
Since $r_2(k)$ vanishes to all orders at $k = i$ and $q = -\tilde{r}(k_{1})^{-\frac{1}{2}}r_{2}(k_{1})$, we see that $|q| = O(C_N(\tau))$.
Since $m_{1}^{X}(q) = O(q)$ as $q \to 0$, the estimates (\ref{mk1estimate1}) and (\ref{mk1estimate2}) are a consequence of (\ref{mk1minusI}).
%Since $\hat{z}(\tau, k_1)=1$ and $\partial D_{\epsilon}(k_1)$ is oriented negatively, equation (\ref{mk1estimate2}) follows from (\ref{mk1minusI}) and Cauchy's formula. 
\end{proof}

\section{The $n^{(4)}\to \hat{n}$ transformation}\label{section:small norm}

The symmetries
\begin{align*}
m^{k_1}(x,t,k) = \mathcal{A} m^{k_1}(x,t,\omega k)\mathcal{A}^{-1} = \mathcal{B} m^{k_1}(x,t,k^{-1}) \mathcal{B}
\end{align*}
allow us to extend the domain of definition of $m^{k_1}$ from $D_\epsilon(k_1)$ to $\mathcal{D}$, where $\mathcal{D} := D_\epsilon(k_1) \cup \omega D_\epsilon(k_1) \cup \omega^2 D_\epsilon(k_1) \cup D_\epsilon(k_1^{-1}) \cup \omega D_\epsilon(k_1^{-1}) \cup \omega^2 D_\epsilon(k_1^{-1})$. 
Shrinking $\tau_{\max}$ if necessary, we may assume that $i$ belongs to $D_\epsilon(k_1)$ for all $\tau \in \mathcal{I}$.
Our next goal is to prove that
\begin{align}\label{Sector I final transfo}
\hat{n} :=
\begin{cases}
n^{(4)} (m^{k_1})^{-1}, & k \in \mathcal{D}, \\
n^{(4)}, & \text{elsewhere},
\end{cases}
\end{align}
is close to $I$ for large $x$ and $\tau \in \mathcal{I}$. Let $\hat{\Gamma} = \Gamma^{(4)} \cup \partial \mathcal{D}$ be the contour shown in Figure \ref{fig:Gammahat}, and let us orient each circle that is part of $\partial \mathcal{D}$ in the clockwise direction. Define
\begin{align}\label{def of vhat II}
\hat{v}= \begin{cases}
v^{(4)}, & k \in \hat{\Gamma} \setminus \bar{\mathcal{D}},
	\\
m^{k_1}, & k \in \partial \mathcal{D},
	\\
m_-^{k_1} v^{(4)}(m_+^{k_1})^{-1}, & k \in \hat{\Gamma} \cap \mathcal{D}.
\end{cases}
\end{align}
We denote the set of self-intersection points of $\hat{\Gamma}$ by $\hat{\Gamma}_{\star}$. The function $\hat{n}$ satisfies the following RH problem: (a) $\hat{n}:\mathbb{C}\setminus \hat{\Gamma}\to \mathbb{C}^{1\times 3}$ is analytic, (b) $\hat{n}_{+} = \hat{n}_{-}\hat{v}$ for $k \in \hat{\Gamma}\setminus \hat{\Gamma}_{\star}$, (c) $\hat{n}(x,t,k) = (1,1,1)+O(k^{-1})$ as $k \to \infty$, and (d) $\hat{n}(x,t,k)=O(1)$ as $k\to k_{\star}\in \hat{\Gamma}_{\star}$.

\begin{figure}
\begin{center}
\begin{tikzpicture}[master,scale=1.1]
\node at (0,0) {};
\draw[black,line width=0.65 mm] (0,0)--(30:7.5);

\draw[black,line width=0.65 mm] (0,0)--(90:7.5);

\draw[black,line width=0.65 mm] (0,0)--(150:7.5);
\draw[dashed,black,line width=0.15 mm] (0,0)--(60:7.5);
\draw[dashed,black,line width=0.15 mm] (0,0)--(120:7.5);

\draw[black,line width=0.65 mm] ([shift=(30:3*1.5cm)]0,0) arc (30:150:3*1.5cm);

\draw[black,line width=0.65 mm] ([shift=(30:3*1.5cm)]0,0) arc (30:150:3*1.5cm);

\draw[black,line width=0.65 mm] (81:3.24)--($(96.9725:4.5)+(96.9725-138:1.1)$)--(96.9725:4.5)--($(96.9725:4.5)+(96.9725-45:1.1)$)--(84:6.25);

\draw[black,line width=0.65 mm] (120:3.24)--($(96.9725:3*1.5)+(96.9725+135:1.1)$)--(96.9725:3*1.5)--($(96.9725:3*1.5)+(96.9725+45:1.1)$)--(120:6.25);
\draw[black,line width=0.65 mm] (120:3.24)--($(143.028:3*1.5)+(143.028-135:1.1)$)--(143.028:3*1.5)--($(143.028:3*1.5)+(143.028-45:1.1)$)--(120:6.25);

\draw[black,line width=0.65 mm] (150:3.9)--(143.028:3*1.5)--(150:5.3);

\draw[black,line width=0.65 mm] (60:4.5)--(60:6.25);
\draw[black,line width=0.65 mm] (60:3.24)--(60:4.5);
\draw[black,line width=0.65 mm] (120:4.5)--(120:6.25);
\draw[black,line width=0.65 mm] (120:3.24)--(120:4.5);

\draw[black,line width=0.65 mm] (143.028:3.24)--(143.028:6.25);
\draw[black,line width=0.65 mm] (96.9725:3.24)--(96.9725:6.25);

\draw[black,line width=0.65 mm] ([shift=(30:3.24cm)]0,0) arc (30:150:3.24cm);
\draw[black,line width=0.65 mm] ([shift=(30:6.25cm)]0,0) arc (30:150:6.25cm);

\draw[green,fill] (143.028:3*1.5) circle (0.12cm);
\draw[blue,fill] (96.9725:3*1.5) circle (0.12cm);

\draw[black,line width=0.65 mm] (90:4.5)--(70:6.25)--(69.5:7.5);
\draw[black,line width=0.65 mm] (70:0)--(70:3.24)--(90:4.5);

\draw[black,line width=0.65 mm] (30:4.5)--(50:6.25)--(50.5:7.5);
\draw[black,line width=0.65 mm] (50:0)--(50:3.24)--(30:4.5);

%reflections of contour 4 and 6
\draw[black,line width=0.65 mm] (39:3.24)--($(30-6.9725:4.5)+(30-6.9725+138:1.1)$)--(30:4.05);
\draw[black,line width=0.65 mm] (30:5.16)--($(30-6.9725:4.5)+(30-6.9725+45:1.1)$)--(36:6.25);

\draw[green,fill] (143.028:3*1.5) circle (0.12cm);
\draw[blue,fill] (96.9725:3*1.5) circle (0.12cm);

\draw[black,line width=0.65 mm] ([shift=(-60:1.1cm)]143.028:3*1.5) arc (-60:180:1.1cm);
\draw[black,line width=0.65 mm] (96.9725:3*1.5) circle (1.1cm);

\draw[black,line width=0.65 mm] ([shift=(60:1.1cm)]30-6.9725:3*1.5) arc (60:180:1.1cm);

\end{tikzpicture}
\end{center}
\begin{figuretext}
\label{fig:Gammahat}
The contour $\hat{\Gamma}=\Gamma^{(4)}\cup \partial \mathcal{D}$ for $\arg k \in [\frac{\pi}{6},\frac{5\pi}{6}]$ (solid), the boundary of $\mathsf{S}$ (dashed), and the saddle points $k_{1}$ (blue) and $\omega^{2}k_{2}$ (green).
\end{figuretext}
\end{figure}

\begin{lemma}\label{whatlemma}
Let $\hat{w} = \hat{v}-I$. The following inequalities hold uniformly for $x \geq 2$ and $\tau \in \mathcal{I}$:
\begin{subequations}\label{hatwestimate}
\begin{align}\label{hatwestimate1}
& \| \hat{w}\|_{(L^1\cap L^\infty)(\hat\Gamma \setminus (\partial \mathcal{D} \cup \hat{\mathcal{X}}^\epsilon))} \leq Cx^{-N},
	\\\label{hatwestimate3}
& \| \hat{w}\|_{(L^1\cap L^{\infty})(\partial \mathcal{D})} \leq C_N(\tau) x^{-1/2},	\\\label{hatwestimate4}
& \| \hat{w}\|_{L^1(\hat{\mathcal{X}}^\epsilon)} \leq C_N(\tau) x^{-1}\ln x,
	\\\label{hatwestimate5}
& \| \hat{w}\|_{L^\infty(\hat{\mathcal{X}}^\epsilon)} \leq C_N(\tau) x^{-1/2}\ln x.
\end{align}
\end{subequations}
\end{lemma}
\begin{proof}
Note that $\hat{\Gamma} \setminus (\partial \mathcal{D} \cup \hat{\mathcal{X}}^\epsilon) = \Gamma^{(4)}\setminus \hat{\mathcal{X}}^\epsilon$. Thus \eqref{hatwestimate1} is a consequence of \eqref{def of vhat II}, Lemma \ref{II v3lemma}, and the uniform boundedness of $m^{k_{1}}$. The estimate \eqref{hatwestimate3} follows from \eqref{def of vhat II} and \eqref{mk1estimate1}. The inequalities \eqref{hatwestimate4} and \eqref{hatwestimate5} are direct consequences of \eqref{def of vhat II}, \eqref{v4vk1estimate}, and the uniform boundedness of $m^{k_1}$.
\end{proof}
The Cauchy transform $\hat{\mathcal{C}}h$ of a function $h$ defined on $\hat{\Gamma}$ is defined by
\begin{align*}
(\hat{\mathcal{C}}h)(z) = \frac{1}{2\pi i} \int_{\hat{\Gamma}} \frac{h(z')dz'}{z' - z}, \qquad z \in \C \setminus \hat{\Gamma}.
\end{align*}
The inequalities in Lemma \ref{whatlemma} imply that
\begin{align}\label{hatwLinfty}
\begin{cases}
\|\hat{w}\|_{L^1(\hat{\Gamma})}\leq Cx^{-N} + C_N(\tau) x^{-1/2},
	\\
\|\hat{w}\|_{L^2(\hat{\Gamma})}\leq 
Cx^{-N} + C_N(\tau) x^{-1/2}(\ln x)^{1/2},
	\\
\|\hat{w}\|_{L^\infty(\hat{\Gamma})}\leq Cx^{-N} + C_N(\tau) x^{-1/2}\ln x,
\end{cases}	 \qquad x \geq 2, \ \tau \in \mathcal{I}.
\end{align}
In particular, we have $\hat{w} \in L^{2}(\hat{\Gamma}) \cap L^{\infty}(\hat{\Gamma})$, so that the operator $\hat{\mathcal{C}}_{\hat{w}}=\hat{\mathcal{C}}_{\hat{w}(x,t,\cdot)}: L^{2}(\hat{\Gamma})+L^{\infty}(\hat{\Gamma}) \to L^{2}(\hat{\Gamma})$, $h \mapsto \hat{\mathcal{C}}_{\hat{w}}h := \hat{\mathcal{C}}_{-}(h \hat{w})$ is well-defined. We denote the space of bounded linear operators on $L^{2}(\hat{\Gamma})$ by $\mathcal{B}(L^{2}(\hat{\Gamma}))$. Using \eqref{hatwLinfty} we infer that there exists an $X > 0$ such that $I - \hat{\mathcal{C}}_{\hat{w}(x, t, \cdot)} \in \mathcal{B}(L^{2}(\hat{\Gamma}))$ is invertible for all $x \geq X$ and $\tau \in \mathcal{I}$. Hence $\hat{n}$ satisfies a small-norm RH problem and can be expressed as
\begin{align}\label{IIbis hatmrepresentation}
\hat{n}(x, t, k) = (1,1,1) + \hat{\mathcal{C}}(\hat{\mu}\hat{w}) = (1,1,1) + \frac{1}{2\pi i}\int_{\hat{\Gamma}} \hat{\mu}(x, t, s) \hat{w}(x, t, s) \frac{ds}{s - k}
\end{align}
for $x \geq X$ and $\tau \in \mathcal{I}$, where 
\begin{align}\label{IIbis hatmudef}
\hat{\mu} = (1,1,1) + (I - \hat{\mathcal{C}}_{\hat{w}})^{-1}\hat{\mathcal{C}}_{\hat{w}}(1,1,1) \in (1,1,1) + L^{2}(\hat{\Gamma}).
\end{align}
Furthermore, using \eqref{hatwLinfty}, \eqref{IIbis hatmudef}, and the fact that $\|\hat{\mathcal{C}}_{-}\|_{\mathcal{B}(L^{2}(\hat{\Gamma}))}< \infty$, we find 
\begin{align}\label{estimate on mu}
& \|\hat{\mu} - (1,1,1)\|_{L^2(\hat{\Gamma})} \leq  \frac{C\|\hat{w}\|_{L^2(\hat{\Gamma})}}{1 - \|\hat{\mathcal{C}}_{\hat{w}}\|_{\mathcal{B}(L^2(\hat{\Gamma}))}}
\leq  Cx^{-N} + C_N(\tau) x^{-1/2}(\ln x)^{1/2}
\end{align}
for all $\tau \in \mathcal{I}$ and all large enough $x$. To determine the long-time asymptotics of $\hat{n}$, define $\hat{n}^{(1)}$ as the nontangential limit
\begin{align*}
& \hat{n}^{(1)}(x,t):=\ntlim_{k\to \infty} k(\hat{n}(x,t,k) - (1,1,1))
= - \frac{1}{2\pi i}\int_{\hat{\Gamma}} \hat{\mu}(x,t,k) \hat{w}(x,t,k) dk.
\end{align*}
\begin{lemma}\label{lemma: hatmplp asymp sector I}
As $x \to \infty$, 
\begin{align}\label{limlhatm}
& \hat{n}^{(1)}(x,t) = -\frac{(1,1,1)}{2\pi i}\int_{\partial \mathcal{D}} \hat{w}(x,t,k) dk + O\big(x^{-N} + C_N(\tau) x^{-1}\ln x\big).
\end{align}
\end{lemma}
\begin{proof}
It suffices to write
\begin{align*}
\hat{n}^{(1)}(x,t) = & -\frac{(1,1,1)}{2\pi i}\int_{\partial \mathcal{D}} \hat{w}(x,t,k) dk  -\frac{(1,1,1)}{2\pi i}\int_{\hat{\Gamma}\setminus\partial \mathcal{D}} \hat{w}(x,t,k) dk
	\\
& -\frac{1}{2\pi i}\int_{\hat{\Gamma}} (\hat{\mu}(x,t,k)-(1,1,1))\hat{w}(x,t,k) dk,
\end{align*}
and then to use \eqref{estimate on mu} and Lemma \ref{whatlemma}.
\end{proof}
Let $\{F^{(l)}\}_{l \in \mathbb{Z}}$ be given by
\begin{align*}
F^{(l)}(x, \tau) = & - \frac{1}{2\pi i} \int_{\partial D_\epsilon(k_1)} k^{l-1}\hat{w}(x,t,k) dk	
= - \frac{1}{2\pi i}\int_{\partial D_\epsilon(k_1)}k^{l-1}(m^{k_1} - I) dk.
\end{align*}
Using \eqref{mk1estimate2} and $\hat{z}(\tau,k_{1})=1$, and recalling that $\partial D_\epsilon(k_{1})$ is oriented clockwise, we get
\begin{align}
F^{(l)}(x, \tau) &  =  k_{1}^{l-1}\frac{Y(x, \tau) m_1^{X}(q) Y(x, \tau)^{-1}}{z_{\star}\sqrt{x}} + O(C_N(\tau) x^{-1}) =  -ik_{1}^{l}Z(x, \tau) + O(C_N(\tau) x^{-1}) \label{asymptotics for Fl}
\end{align}
as $x \to \infty$ uniformly for $\tau \in \mathcal{I}$, where
\begin{align*}
Z(x, \tau) & = \frac{Y(x, \tau) m_1^{X} Y(x, \tau)^{-1}}{-ik_{1}z_{\star}\sqrt{x}} = 
\frac{1}{-ik_{1}z_{\star}\sqrt{x}} \begin{pmatrix}
0 & \hspace{-0.15cm}\frac{\beta_{12}}{d_{0} e^{x\tilde{\Phi}_{21}(\tau,k_{1})} \tilde{r}(k_{1})^{\frac{1}{2}}} & \hspace{-0.15cm}0 \\
d_{0} e^{x\tilde{\Phi}_{21}(\tau,k_{1})} \tilde{r}(k_{1})^{\frac{1}{2}} \beta_{21} & \hspace{-0.15cm}0 & \hspace{-0.15cm}0 \\
0 & \hspace{-0.15cm}0 & \hspace{-0.15cm}0
\end{pmatrix}\hspace{-0.07cm}.
\end{align*}

\begin{lemma}\label{lemma: some integrals by symmetry}
For $l \in \mathbb{Z}$ and $j=0,1,2$, we have
\begin{align}
& -\frac{1}{2\pi i}\int_{\omega^{j} \partial D_\epsilon(k_1)} k^{l}\hat{w}(x,t,k)dk = \omega^{j(l+1)} \mathcal{A}^{-j}F^{(l+1)}(x, \tau)\mathcal{A}^{j}, \label{int1} \\
& -\frac{1}{2\pi i}\int_{\omega^{j} \partial D_\epsilon(k_1^{-1})} k^{l}\hat{w}(x,t,k)dk = -\omega^{j(l+1)}\mathcal{A}^{-j}\mathcal{B} F^{(-l-1)}(x, \tau) \mathcal{B}\mathcal{A}^{j}. \label{int1 tilde}
\end{align}
\end{lemma}
\begin{proof}
These identities follow from $\hat{w}(x, t, k) = \mathcal{A} \hat{w}(x, t, \omega k) \mathcal{A}^{-1} = \mathcal{B} \hat{w}(x, t, k^{-1}) \mathcal{B}$ which hold for $k \in \partial \mathcal{D}$. A more detailed argument in a similar situation can be found in \cite[Proof of Lemma 12.4]{CLsectorIV}.
\end{proof}

It follows from Lemma \ref{lemma: some integrals by symmetry} that
\begin{align*}
&  \frac{-1}{2\pi i}\int_{\partial \mathcal{D}} \hat{w}(x,t,k) dk =  \sum_{j=0}^{2}  \frac{-1}{2\pi i}\int_{\omega^{j} \partial D_\epsilon(k_1)} \hat{w}(x,t,k) dk - \sum_{j=0}^{2}  \frac{1}{2\pi i}\int_{\omega^{j} \partial D_\epsilon(k_1^{-1})} \hat{w}(x,t,k) dk \\
& = \sum_{j=0}^{2} \omega^{j} \mathcal{A}^{-j}F^{(1)}(x, \tau)\mathcal{A}^{j} - \sum_{j=0}^{2} \omega^{j}\mathcal{A}^{-j}\mathcal{B} F^{(-1)}(x, \tau) \mathcal{B}\mathcal{A}^{j}.
\end{align*}
Hence, \eqref{limlhatm} and \eqref{asymptotics for Fl} show that $\hat{n}^{(1)}=(1,1,1)\hat{m}^{(1)}$, where
\begin{align}\nonumber
& \hat{m}^{(1)}(x,t) =  \; \sum_{j=0}^{2} \omega^{j} \mathcal{A}^{-j}F^{(1)}(x, \tau)\mathcal{A}^{j} - \sum_{j=0}^{2} \omega^{j}\mathcal{A}^{-j}\mathcal{B} F^{(-1)}(x, \tau) \mathcal{B}\mathcal{A}^{j} + O\big(x^{-N} + C_N(\tau) x^{-1}\ln x\big) \\
& = -ik_{1}\sum_{j=0}^{2} \omega^{j} \mathcal{A}^{-j}Z(x, \tau)\mathcal{A}^{j} + ik_{1}^{-1} \sum_{j=0}^{2} \omega^{j}\mathcal{A}^{-j}\mathcal{B} Z(x, \tau) \mathcal{B}\mathcal{A}^{j} + O\big(x^{-N} + C_N(\tau) x^{-1}\ln x\big) \label{mhatplp asymptotics}
\end{align}
as $x \to \infty$, uniformly for $\tau \in \mathcal{I}$. 

\section{Asymptotics of $u$}\label{uasymptoticssec}

By \eqref{recoveruvn}, we have $u(x,t) = -i\sqrt{3}\frac{\partial}{\partial x}n_{3}^{(1)}(x,t)$, where $n_{3}(x,t,k) = 1+n_{3}^{(1)}(x,t)k^{-1}+O(k^{-2})$ as $k \to \infty$. By inverting the transformations $n \mapsto n^{(1)}\mapsto n^{(2)}\mapsto n^{(3)}\mapsto n^{(4)}\mapsto \hat{n}$ using (\ref{n1def}), \eqref{n2def}, \eqref{n3def}, \eqref{IIbis def of mp3p}, and \eqref{Sector I final transfo}, we obtain
\begin{align*}
n = \hat{n}\Delta^{-1}(G^{(3)})^{-1}(G^{(2)})^{-1}(G^{(1)})^{-1}, \qquad k \in \mathbb{C}\setminus (\hat{\Gamma}\cup\bar{\mathcal{D}}),
\end{align*}
where $G^{(1)}$, $G^{(2)}$, $G^{(3)}$, $\Delta$ are given by (\ref{G1def}), \eqref{G2def}, \eqref{G3def}, and \eqref{Deltadef}, respectively. Thus
\begin{align}\nonumber
u(x,t) & = -i\sqrt{3}\frac{\partial}{\partial x}\bigg( \hat{n}_{3}^{(1)}(x,t) + \ntlim_{k\to \infty} k (\Delta_{33}(\tau,k)^{-1}-1) \bigg),
\end{align}
where $\hat{n}_{3}^{(1)}$ is defined via the expansion $\hat{n}_{3}(x,t,k) = 1+\hat{n}_{3}^{(1)}(x,t)k^{-1}+O(k^{-2})$ as $k \to \infty$. 
By virtue of (\ref{Delta33asymptotics}) and the change of variables formula
\begin{align}\label{ddxddtau}
\frac{\partial}{\partial x}\bigg|_{\text{$t$ const}} = \frac{\partial}{\partial x}\bigg|_{\text{$\tau$ const}} - \frac{\tau}{x} \frac{\partial}{\partial \tau}\bigg|_{\text{$x$ const}},
\end{align}
we obtain
\begin{align}\label{partialxDelta33}
\frac{\partial}{\partial x} \ntlim_{k\to \infty} k (\Delta_{33}(\tau,k)^{-1}-1) 
=  \frac{\tau}{x} \frac{\sqrt{3}}{2\pi} \frac{d}{d \tau}\int_{i}^{k_1(\tau)} \ln(1 + r_1(s)r_{2}(s)) (1 + s^{-2}) ds.
\end{align}
Since $r_1(s)r_{2}(s)$ vanishes to all orders at $s = i$ and $k_1 = ie^{i\tau + O(\tau^{3})}$ as $\tau \to 0$, the right-hand side of (\ref{partialxDelta33}) is $O(C_N(\tau) x^{-1})$. Hence
\begin{align}\label{IIbis recoverun}
u(x,t) = -i\sqrt{3}\frac{\partial}{\partial x} \hat{n}_{3}^{(1)}(x,t) + O(C_N(\tau) x^{-1}), \qquad x \to \infty,
\end{align}
uniformly for $\tau \in \mathcal{I}$.
By \eqref{mhatplp asymptotics}, we have
\begin{align*}
 \hat{n}_{3}^{(1)} = &\; Z_{12} \big(\omega i k_{1}^{-1} - \omega^{2} i k_{1} \big) + Z_{21} \big( \omega^{2} i k_{1}^{-1} - \omega i k_{1} \big) + O\big(x^{-N} + C_N(\tau) x^{-1}\ln x\big)
	 \\
 = &\; \frac{\beta_{12}(\omega i k_{1}^{-1}-\omega^{2} i k_{1})}{-i k_{1} z_{\star} \sqrt{x} d_{0}e^{x\tilde{\Phi}_{21}(\tau,k_{1})}\tilde{r}(k_{1})^{\frac{1}{2}}} + \frac{d_{0}e^{x\tilde{\Phi}_{21}(\tau,k_{1})}\tilde{r}(k_{1})^{\frac{1}{2}}\beta_{21}(\omega^{2} i k_{1}^{-1}-\omega i k_{1})}{-ik_{1}z_{\star} \sqrt{x}} 
	\\
& + O\big(x^{-N} + C_N(\tau) x^{-1}\ln x\big)
\end{align*}
as $x \to \infty$ uniformly for $\tau \in \mathcal{I}$. The identity $|d_0(x, \tau)| = e^{2\pi \nu}$ established in Lemma \ref{lemma: bound on Y} implies that $\frac{\beta_{12}}{d_{0}} = -\bar{d_{0}} \bar{\beta}_{21}$. Therefore, using also that $(\omega i k_{1}^{-1}-\omega^{2} i k_{1})/\tilde{r}(k_{1})^{\frac{1}{2}} = \tilde{r}(k_{1})^{\frac{1}{2}}(\omega^{2} i k_{1}^{-1}-\omega i k_{1}) \in \R$ and $-ik_{1}z_{\star}>0$,
\begin{align*}
\hat{n}_{3}^{(1)} = 2 i \, \im \frac{\beta_{12}(\omega i k_{1}^{-1}-\omega^{2} i k_{1})}{-i k_{1} z_{\star} \sqrt{x} d_{0}e^{x\tilde{\Phi}_{21}(\tau,k_{1})}\tilde{r}(k_{1})^{\frac{1}{2}}} + O\big(x^{-N} + C_N(\tau) x^{-1}\ln x\big) \qquad \mbox{as } x \to \infty.
\end{align*}
Since
\begin{align*}
& \beta_{12} = \frac{\sqrt{2\pi}e^{\frac{\pi i}{4}}e^{\frac{3\pi \nu}{2}}}{q \Gamma(i\nu)}, \qquad d_{0}(x, \tau) = e^{2\pi \nu}e^{i\arg d_{0}(x, \tau)}, \qquad |q| = \sqrt{e^{-2\pi \nu}-1}, \\
& |\Gamma(i\nu)| = \frac{\sqrt{2\pi}}{\sqrt{-\nu}\sqrt{e^{-\pi \nu}-e^{\pi \nu}}} = \frac{\sqrt{2\pi}}{\sqrt{-\nu}e^{\frac{\pi\nu}{2}}|q|}, \quad \frac{\omega i k_{1}^{-1}-\omega^{2} i k_{1}}{\tilde{r}(k_{1})^{\frac{1}{2}}} = -\sqrt{-1-2\cos (2\arg k_{1})}, \\
& 
\end{align*}
we find
\begin{align*}
\hat{n}_{3}^{(1)} \hspace{-0.07cm} =  & -2 i \frac{\sqrt{-\nu}\sqrt{-1-2\cos (2\arg k_{1})}}{-ik_{1}z_{\star}\sqrt{x}}  \sin \Big( \frac{\pi}{4}-\arg q - \arg \Gamma(i \nu) - \arg d_{0} - x \im \tilde{\Phi}_{21}(\tau,k_{1}) \hspace{-0.07cm} \Big) \\
& + O\big(x^{-N} + C_N(\tau) x^{-1}\ln x\big) \qquad \mbox{as } x \to \infty,
\end{align*}
uniformly for $\tau \in \mathcal{I}$. As in \cite{CLWasymptotics}, we can prove that the above asymptotics can be differentiated with respect to $x$ without increasing the error term. Hence, by substituting the above formula into \eqref{IIbis recoverun} and using (\ref{ddxddtau}) and the relation $\im \big(\tilde{\Phi}_{21}(\tau, k_1) - \tau \frac{d}{d\tau}\tilde{\Phi}_{21}(\tau, k_{1}(\tau))\big)=\im k_{1}$, we arrive at
\begin{align*}
 u  = & \; -2\sqrt{3} \bigg(\frac{\partial}{\partial x} - \frac{\tau}{x} \frac{\partial}{\partial \tau} \bigg)\bigg( \sin \Big( \frac{\pi}{4}-\arg q - \arg \Gamma(i \nu) - \arg d_{0} - x \im \tilde{\Phi}_{21}(\tau,k_{1}) \Big) 
 	\\
& \times \frac{\sqrt{-\nu}\sqrt{-1-2\cos (2\arg k_{1})}}{-ik_{1}z_{\star}\sqrt{x}} \bigg) + O\big(x^{-N} + C_N(\tau) x^{-1}\ln x\big) \\
= &\; \frac{A(\tau)}{\sqrt{x}} \cos \Big( \frac{\pi}{4}-\arg q - \arg \Gamma(i \nu) - \arg d_{0} - x \im \tilde{\Phi}_{21}(\tau,k_{1}) \Big) + O\big(x^{-N} + C_N(\tau) x^{-1}\ln x\big)
\end{align*}
as $x \to \infty$ uniformly for $\tau \in \mathcal{I}$, where $A(\tau)$ is as in the statement of Theorem \ref{asymptoticsth}.
Since $\arg q = \pi+\arg r_{2}(k_{1}) \pmod{2\pi}$ and $\cos(-x)=\cos(x)$, we can write this as
\begin{align*}
 u(x,t)  = &\; \frac{A(\tau)}{\sqrt{x}} \cos \Big( \frac{3\pi}{4}+\arg r_{2}(k_{1}) + \arg \Gamma(i \nu) + \arg d_{0} + x \im \tilde{\Phi}_{21}(\tau,k_{1}) \Big) 
 	\\
& + O\big(x^{-N} + C_N(\tau) x^{-1}\ln x\big)
\end{align*}
as $x \to \infty$ uniformly for $\tau \in \mathcal{I}$. The expression \eqref{def of nu beta} for $\arg d_{0}(x, \tau)$ follows from a long but direct computation using \eqref{d0def}, \eqref{delta expression in terms of log and chi}, and \eqref{def of chi}. This finishes the proof of Theorem \ref{asymptoticsth}.

\appendix 

\section{Model RH problem}\label{appendix A}

The lemma presented in this appendix is needed for the local parametrix near $k_{1}$. The proof involves parabolic cylinder functions as in \cite{I1981}, and we refer to \cite[Appendix A]{CLsectorV} for a constructive proof of a similar lemma.

Let $X = X_1 \cup \cdots \cup X_4 \subset \C$ be the cross consisting of the four rays
\begin{align} \nonumber
&X_1 = \bigl\{se^{\frac{i\pi}{4}}\, \big| \, 0 \leq s < \infty\bigr\}, && 
X_2 = \bigl\{se^{\frac{3i\pi}{4}}\, \big| \, 0 \leq s < \infty\bigr\},  
	\\ \label{Xdef}
&X_3 = \bigl\{se^{-\frac{3i\pi}{4}}\, \big| \, 0 \leq s < \infty\bigr\}, && 
X_4 = \bigl\{se^{-\frac{i\pi}{4}}\, \big| \, 0 \leq s < \infty\bigr\},
\end{align}
all oriented away from the origin.

\begin{lemma}[Model RH problem needed near $k=k_{1}$]\label{IIbis Xlemma 3}
Let $q \in \mathbb{C}$, and define
\begin{align*}
\nu = -\tfrac{1}{2\pi} \ln(1 + |q|^2) \leq 0.
\end{align*}
Define the jump matrix $v^{X}(z)$ for $z \in X$ by
\begin{align}\nonumber 
& \begin{pmatrix} 
1 & \frac{-\bar{q}}{1 + |q|^2} z_{(0)}^{-2i\nu}e^{\frac{iz^2}{2}} & 0 \\
0	& 1 & 0 \\
0 & 0 & 1
\end{pmatrix} \mbox{ if } z \in X_{1}, & & 
 \begin{pmatrix} 
1	& 0 & 0	\\
q z_{(0)}^{2i\nu}e^{-\frac{iz^2}{2}}	& 1 & 0 \\
0 & 0 & 1
\end{pmatrix} \mbox{ if } z \in X_{2}, 
	\\ \label{vXdef}
& \begin{pmatrix} 
1 & \bar{q} z_{(0)}^{-2i\nu} e^{\frac{iz^2}{2}} & 0 \\
0 & 1 & 0 \\
0 & 0 & 1
\end{pmatrix} \mbox{ if } z \in X_{3}, & & \begin{pmatrix} 
1 & 0 & 0 \\
\frac{-q}{1 + |q|^2}z_{(0)}^{2i\nu} e^{-\frac{iz^2}{2}} & 1 & 0 \\
0 & 0 & 1 \end{pmatrix} \mbox{ if } z \in X_{4},
\end{align}
where $z_{(0)}^{i\nu}$ has a branch cut along $[0,+\infty)$, such that $z_{(0)}^{i\nu} = |z|^{i\nu}e^{-\nu  \arg_{0}(z)}$, $\arg_{0}(z) \in (0,2\pi)$. Then the RH problem 
\begin{enumerate}[$(a)$]
\item $m^{X}(q, \cdot) = m^{X}(q, \cdot) : \C \setminus X \to \mathbb{C}^{3 \times 3}$ is analytic;

\item on $X \setminus \{0\}$, the boundary values of $m^{X}(q, \cdot)$ exist, are continuous, and satisfy $m_+^{X} =  m_-^{X} v^{X}$;

\item $m^{X}(q,z) = I + O(z^{-1})$ as $z \to \infty$, and $m^{X}(q,z) = O(1)$ as $z \to 0$;
\end{enumerate}
has a unique solution $m^{X}(q,z)$. This solution satisfies
\begin{align}\label{mXasymptotics}
m^{X}(q,z) = I + \frac{m_{1}^{X}(q)}{z} + O\biggl(\frac{q}{z^2}\biggr), \quad z \to \infty,  \quad m_{1}^{X}(q) :=\begin{pmatrix} 
0 & \beta_{12} & 0 \\ 
\beta_{21} & 0 & 0 \\
0 & 0 & 0 \end{pmatrix},
\end{align}  
where the error term is uniform for $\arg z \in [-\pi, \pi]$ and $q$ in compact subsets of $\mathbb{C}$, and
\begin{align}\label{betaXdef}
& \beta_{12} := \frac{\sqrt{2\pi}e^{\frac{\pi i}{4}}e^{\frac{3\pi \nu}{2}}}{q \Gamma(i\nu)}, \qquad \beta_{21} := \frac{\sqrt{2\pi}e^{-\frac{\pi i}{4}}e^{-\frac{5\pi \nu}{2}}}{-\bar{q} \Gamma(-i\nu)}.
\end{align}
\end{lemma}

\begin{remark}
Since $|\Gamma(i\nu)| = \frac{\sqrt{2\pi}}{\sqrt{-\nu}\sqrt{e^{-\pi\nu}-e^{\pi\nu}}}=\frac{\sqrt{2\pi}}{\sqrt{-\nu}e^{\frac{\pi\nu}{2}}|q|}$, it follows that $\beta_{12}\beta_{21} = \nu$ and that $m_{1}^{X}(q) = O(q)$ as $q \to 0$.
\end{remark}

%In fact,
%\begin{align}\label{II mXasymptotics precise 3}
%  m^X(q, z) = I + \begin{pmatrix} \frac{a_2(q)}{z^2} + \frac{a_4(q)}{z^4} + \cdots & \frac{\beta^X(q)}{z} + \frac{b_3(q)}{z^3} + \frac{b_5(q)}{z^5} + \cdots  \\ 
%   \frac{\overline{\beta^X(q)}}{z} + \frac{c_3(q)}{z^3} + \frac{c_5(q)}{z^5} + \cdots & \frac{d_2(q)}{z^2} + \frac{d_4(q)}{z^4} + \cdots \end{pmatrix}, \qquad z \to \infty,  \ q \in \mathbb{D}, 
%\end{align}  
%where $a_j(q) = O(q^2)$, $b_j(q) = O(q)$, $c_j(q) = O(q)$, $d_j(q) = O(q^2)$.
%
%
%\begin{figure}
%\begin{center}
% \begin{overpic}[width=.4\textwidth]{../Images/X.pdf}
%      \put(73,68){\small $X_1$}
%      \put(19,68){\small $X_2$}
%      \put(17,27){\small $X_3$}
%      \put(75,27){\small $X_4$}
%      \put(48,43){$0$}
%    \end{overpic}
%     \begin{figuretext}\label{X.pdf}
%        The contour $X = X_1 \cup X_2 \cup X_3 \cup X_4$.
%     \end{figuretext}
%     \end{center}
%\end{figure}
%

\subsection*{Acknowledgements}
Support is acknowledged from the Novo Nordisk Fonden Project, Grant 0064428, the European Research Council, Grant Agreement No. 682537, the Swedish Research Council, Grant No. 2015-05430, Grant No. 2021-04626, and Grant No. 2021-03877, the G\"oran Gustafsson Foundation, and the Ruth and Nils-Erik Stenb\"ack Foundation.

\bibliographystyle{plain}
\bibliography{is}

\begin{thebibliography}{99}
\small

\bibitem{B1872}
J. Boussinesq, Th\'eorie des ondes et des remous qui se propagent le long d'un canal rectangulaire horizontal, en communiquant au liquide contenu dans ce canal des vitesses sensiblement pareilles de la surface au fond, {\it J. Math. Pures Appl.} {\bf 17} (1872), 55--108. 

\bibitem{CLmain} 
C. Charlier and J. Lenells, On Boussinesq's equation for water waves, arXiv:2204.02365.

\bibitem{CLsectorV} 
C. Charlier and J. Lenells, Boussinesq's equation for water waves: asymptotics in Sector V, arXiv:2301.10669.

\bibitem{CLsectorIV} 
C. Charlier and J. Lenells, Boussinesq's equation for water waves: the soliton resolution conjecture for Sector IV, arXiv:2303.00434.

\bibitem{CLWasymptotics}
 C. Charlier, J. Lenells, and D. Wang, The "good" Boussinesq equation: long-time asymptotics, {\it Analysis \& PDE}, to appear, arXiv:2003.04789.
 
\bibitem{DTT1982}
P. Deift, C. Tomei, and E. Trubowitz, Inverse scattering and the Boussinesq equation, {\it Comm. Pure Appl. Math.} {\bf 35} (1982), 567--628.

\bibitem{DZ1993}
P. Deift and X. Zhou, A steepest descent method for oscillatory Riemann-Hilbert problems. Asymptotics for the MKdV equation, 
{\it Ann. of Math.} {\bf 137} (1993), 295--368.

\bibitem{H1973}
R. Hirota, Exact $N$-soliton solutions of the wave equation of long waves in shallow-water and in nonlinear lattices, {\it J. Math. Phys.} {\bf 14} (1973), 810--814. 

\bibitem{I1981}
A. R. Its, Asymptotic behavior of the solutions to the nonlinear Schr\"odinger equation, and isomonodromic deformations of systems of linear differential equations, {\it Dokl. Akad. Nauk SSSR} {\bf 261} (1981), 14--18 (in Russian); {\it Soviet Math. Dokl.} {\bf 24} (1982), 452--456 (in English).

 
\bibitem{Z1974}
V. E. Zakharov, On stochastization of one-dimensional chains of nonlinear oscillations, {\it Soviet Phys. JETP} {\bf 38} (1974), 108--110.

\end{thebibliography}

\end{document}